\documentclass[11pt]{article}
\usepackage[utf8]{inputenc}
\usepackage[title]{appendix}
\usepackage{comment}

\title{Quasi-geodesics in the Cannon--Thurston metric}
\author{Vaibhav Gadre, Joseph Maher, Catherine Pfaff, Caglar Uyanik }

\date{\today}

\usepackage[top=1.25in,bottom=1.25in,left=1in,right=1in]{geometry}
\usepackage[dvipsnames]{xcolor}
\usepackage{amsfonts, amsmath, amssymb, amscd, graphicx}
\usepackage{amsthm}
\usepackage{thmtools}
\usepackage[final, % change to final
colorlinks=true, linkcolor=purple, citecolor=blue, unicode]{hyperref}
\usepackage{cleveref}
\usepackage{enumitem}

\usepackage[nobysame,alphabetic]{amsrefs}
\usepackage{xspace}

\usepackage{tikz}
\usetikzlibrary{matrix,arrows,decorations.markings,decorations.pathmorphing,patterns}
\usetikzlibrary{calc}

\theoremstyle{plain}
\newtheorem{theorem}{Theorem}
\newtheorem{lemma}[theorem]{Lemma}
\newtheorem{proposition}[theorem]{Proposition}
\newtheorem{corollary}[theorem]{Corollary}

\newtheorem*{claim*}{Claim}

\theoremstyle{definition}
\newtheorem{definition}[theorem]{Definition}

\newtheorem{remark}[theorem]{Remark}

\newtheorem{remark-convention}[theorem]{Remark-Convention}

\newcommand{\fakeenv}{} %%% prints the emptystring

{ 
 \renewcommand{\fakeenv}{#2} %%% So now \fakeenv prints #2
 \theoremstyle{plain} 
 \newtheorem*{\fakeenv}{#1~\ref{#2}} %%% so now #2 is the name of a
 \begin{\fakeenv}
}
{
 \end{\fakeenv}
}

\newcommand{\HH}{\mathbb{H}} 
 
\newcommand{\RR}{\mathbb{R}}

\newcommand{\ogamma}{{\overline{\gamma}}}
\newcommand{\LL}{\Lambda}
\newcommand{\oLambda}{{\overline{\Lambda}}}

\renewcommand{\ge}{\geqslant}
\renewcommand{\le}{\leqslant}

\newcommand{\ws}{{\widetilde{S}}}
\newcommand{\wsr}{{\widetilde{S}_h \times \RR}}

\newcommand{\norm}[1]{\left\| {#1} \right\|}

\DeclareFontFamily{U}{mathx}{}
\DeclareFontShape{U}{mathx}{m}{n}{<-> mathx10}{}
\DeclareSymbolFont{mathx}{U}{mathx}{m}{n}
\DeclareMathAccent{\widecheck}{0}{mathx}{"71}

\newcommand{\PSL}{\text{PSL}}

\newcommand{\pa}{f\xspace}

\newcommand{\param}%
	{{\mathchoice{\mkern1mu\mbox{\raise2.2pt\hbox{$\centerdot$}}\mkern1mu}%
	{\mkern1mu\mbox{\raise2.2pt\hbox{$\centerdot$}}\mkern1mu}%
	{\mkern1.5mu\centerdot\mkern1.5mu}{\mkern1.5mu\centerdot\mkern1.5mu}}}

\newtheorem*{proposition:ft}{Proposition \ref{prop:fellow travel}}
\newtheorem*{proposition:pi}{Proposition \ref{prop:projection interval}}
\newtheorem*{theorem:hitting}{Theorem \ref{theorem:hitting}}
\newtheorem*{theorem:lebesgue}{Theorem \ref{theorem:lebesgue}}

\newtheorem*{lemma:bounded distance}{Lemma \ref{lemma:bounded distance}}
\newtheorem*{lemma:intersection interval qg}{Lemma \ref{lemma:intersection interval qg}}
\newtheorem*{prop:union qg}{Proposition \ref{prop:union qg}}
\newtheorem*{lemma:straight qg}{Lemma \ref{lemma:straight qg}}
\newtheorem*{cor:corner distance}{Lemma \ref{cor:corner distance}}
\newtheorem*{prop:vertical flow distance decreasing}{Proposition \ref{prop:vertical flow distance decreasing}}

\newlist{thmenum}{enumerate}{1} % to be used only inside 'theorem' environments
\setlist[thmenum]{leftmargin=3\parindent, label=\textup{(\arabic*)},
                  ref=\textup{\thetheorem.\arabic*}}
\crefname{thmenumi}{proposition}{propositions}

\begin{document}

\maketitle

\begin{abstract}

A closed fibered $3$-manifold admits a complete hyperbolic metric if and only if it has a fibration with a pseudo-Anosov monodromy. 
The stable and the unstable laminations associated to the pseudo-Anosov homeomorphism on the fiber surface give rise to a natural metric on the $3$-manifold, the Cannon--Thurston metric, which is quasi-isometric to the hyperbolic metric.

%Information that is related to the pseudo-Anosov on the fiber surface such as its singular flat metric or its stable/ unstable measured laminations, can be used to define other quasi-isometric metrics on the $3$-manifold such as the singular solv metric or the Cannon--Thurston metric.

\medskip
In this paper, we describe a specific family of quasi-geodesics in the Cannon--Thurston metric. 
%for a closed, hyperbolic, fibered $3$-manifold. 
We use the main results of this article in a companion paper to obtain statistics for typical geodesics with respect to various natural measures on the $2$-sphere, thus giving a geometric criterion for singularity between some of these measure classes.

\end{abstract}

\tableofcontents

%%%%%%%%%%%%%%%%%%%%%%%%%%%%%%%%%%%%%%%%%%%%%%%%%%%%%%%%%%%%%%%%%%%%%%%%%%%%%%%%%%
\section{Introduction}
%%%%%%%%%%%%%%%%%%%%%%%%%%%%%%%%%%%%%%%%%%%%%%%%%%%%%%%%%%%%%%%%%%%%%%%%%%%%%%%%%%

In this paper, we describe a specific family of quasigeodesics in the
Cannon--Thurston metric, with the following two properties, see
\Cref{section:effective} for precise statements.

\begin{enumerate}

\item For a quasigeodesic that corresponds to a bi-infinite geodesic in the
fiber, the distance of the quasigeodesic from the fiber is explicitly
described in terms of the angle the geodesic makes with
the stable and unstable laminations.

\item The projection from the fiber to the quasigeodesic along the
vertical flow is coarsely distance reducing.  

\end{enumerate}

There are a number of previous constructions of quasigeodesics in the
Cannon--Thurston metric, which we discuss in more detail in
\Cref{section:previous}. The properties above have not been verified in these constructions. 
%but they do not have the properties above.
We use these properties in \cite{gmpueffective} to give effective
estimates of the amount of time a typical geodesic spends close to the
fiber.

\medskip
In the remainder of this section, we first review the definition
of the Cannon--Thurston metric, and then discuss previous constructions
of quasigeodesics.  We then give a brief outline of the construction
we use.

%%%%%%%%%%%%%%%%%%%%%%%%%%%%%%%%%%%%%%%%%%%%%%%%%%%%%%%%%%%%%%%%%%%%%%%%%%%%%%%%%%
\subsection{The Cannon--Thurston metric}
%%%%%%%%%%%%%%%%%%%%%%%%%%%%%%%%%%%%%%%%%%%%%%%%%%%%%%%%%%%%%%%%%%%%%%%%%%%%%%%%%%

Suppose that $S$ is a closed orientable surface with genus at least two and  
%Suppose that 
$f: S \to S$ is an orientation preserving homeomorphism. The
\emph{mapping torus of $f$} is the closed $3$-manifold obtained as
$M_{f}=M:=S \times [0,1]/ (f(x),1)\sim (x,0)$.  The universal cover
$\widetilde M_f$ of the mapping torus $M_f$ is a (topological) product
$\widetilde S \times \RR$.  We call the $\mathbb{R}$ action on the
universal cover the \emph{suspension flow}.

\medskip
In his groundbreaking work on geometrization of $3$-manifolds, Thurston proved that a closed $3$-manifold that fibers over the circle admits a complete hyperbolic metric if and only if it is a mapping torus $M_f$ of a pseudo-Anosov mapping class $f$ of a fiber surface $S$ \cite{thurston}. 

\medskip Since $S$ has negative Euler characteristic, the universal
cover $\widetilde S$ can be identified with $\HH^2$ by choosing a complete
hyperbolic metric on $S$. With the metric chosen, we
denote the fiber surface as $S_h$ and the universal cover as
$\widetilde S_h$.  The fiber inclusions
$\iota_r : \widetilde S_h \to \widetilde S_h \times \{r\}$ are exponentially distorted in $\HH^3$. Despite the distortion, Cannon
and Thurston \cite{cannon-thurston} showed that the inclusions extend
to a continuous map
$\partial_\infty \widetilde S_h = S^1 \to S^2 = \partial_\infty \HH^3$
at infinity. The \emph{Cannon--Thurston maps} are continuous,
surjective, and finite-to-one.  The suspension flow gives for each
$r \in \RR$ a map $f_r$ such that $\iota_r = f_r \circ \iota_0 $ with
$f_1$ acting by the lift of the pseudo-Anosov $f$. So we only need to
consider $\iota = \iota_0$.

\medskip
The action of the pseudo-Anosov map $f$ on $S_h$ has two invariant measured geodesic laminations $(\Lambda_\pm)$, where $\Lambda_+$ is called \emph{stable lamination} and $\Lambda_-$ is called \emph{unstable lamination}.
Roughly speaking, $f$ 
stretches the leaves of $\Lambda_-$ by a factor of $k = k(f) >1$ and  contracts the leaves of $\Lambda_+$ by $\frac{1}{k}$. 
In terms of transverse measures, $f$ scales down (respectively up) the transverse measure $dy$ (respectively $dx$) of $\Lambda_-$ (respectively $\Lambda_+$) by a factor $k = k(f) >1$. 
The constant $k$ is called the \emph{stretch factor} of the pseudo-Anosov $f$. Topologically, the invariant laminations and the stretch factor do not depend on the choice of the hyperbolic metric on $S$. 

\medskip
The invariant laminations can be lifted to laminations
$\overline{\Lambda}_\pm$ of $\HH^2 = \widetilde S_h$. 
Cannon and Thurston then define the infinitesimal pseudometric on $\widetilde S_h \times \RR$ by 
\begin{equation}\label{eq:ct metric}
ds^2 = k^{2z} dx^2 + k^{-2z} dy^2 + (\log k)^2
dz^2. 
\end{equation}
where $k$ is the stretch factor of $f$. 
The infinitesimal pseudo-metric gives rise to a
pseudometric $d_{\widetilde S_h \times \RR}$ on
$\widetilde{S}_h \times \RR$ in the standard manner:
\begin{itemize}
    \item integrating the pseudometric along rectifiable paths gives a pseudo-distance; and then
    \item defining the distance between two points as the infimum of the lengths over all rectifiable paths connecting the two points.
\end{itemize}
The resulting pseudo-metric on $\widetilde{S}_h \times \RR$ is called the \emph{Cannon--Thurston metric}.

\medskip
Suppose that $A$ is a subset of $\widetilde S_h$. 
We define the \emph{suspension flow set $F(A)$} to be
\[
F(A) = \bigcup_{z \in \RR} F_z(A).
\]
If $A = \{p \}$ is a point, then $F(p)$ is just the suspension flow
line through $p$, and hence a geodesic in the Cannon--Thurston pseudometric.
If $A$ is a hyperbolic geodesic $\gamma$ in
$\widetilde S_h$, then the suspension flow set $F(\gamma)$ is called
the \emph{ladder} of $\gamma$.  We refer to $\gamma$ as the
\emph{base} of the ladder.

\medskip
Ladders over geodesics in $\widetilde S_h$ are quasiconvex, a special case of a more general result of
Mitra \cite{mitra}*{Lemma 4.1}. 
%We will state Mitra's result using the notation of this paper.
%\begin{theorem}\cite{mitra}*{Lemma 4.1}
%Suppose that $\pa$ is a pseudo-Anosov map and $\widetilde S_h$ is a hyperbolic metric. Then, there
%is a constant $K$, such that for any geodesic $\gamma$ in
%$\widetilde S_h$, the ladder $F(\gamma)$ is $K$-quasiconvex in
%$\widetilde S_h \times \RR$.
%\end{theorem}
Ladders over arbitrary geodesics are also quasi-isometric to $\HH^2$, though we do not use this fact directly.
However, because of the distortion, the image $\iota (\gamma)$ in $\widetilde S_h \times \RR$ need not be a quasi-geodesic, even up to reparameterization. In fact, when $\gamma$ is a leaf of an invariant lamination (or more generally, has a long enough segment that fellow travels a leaf), the image $\iota(\gamma)$ is  horocyclic in the ladder over $\gamma$.

\medskip
Our goal here is to explicitly construct quasi-geodesics in the Cannon--Thurston metric by straightening the ``horocyclic'' $\iota(\gamma)$ segments in the ladder over $\gamma$. We do this by defining a \emph{height function $h_\gamma(t)$} along the geodesic $\gamma$, where $h_\gamma(t)$ is roughly $\log \log $ the distance in the unit tangent bundle from the geodesic to the invariant laminations. We then show that the paths $(\gamma(t), h_\gamma(t))$, which we call \emph{test paths} are uniformly quasi-geodesic, that is their quasi-geodesic constants do not depend on $\gamma$.

%%%%%%%%%%%%%%%%%%%%%%%%%%%%%%%%%%%%%%%%%%%%%%%%%%%%%%%%%%%%%%%%%%%%%%%%%%%%%%%%%%
\subsection{Quasi-geodesics in the universal cover}\label{section:previous}
%%%%%%%%%%%%%%%%%%%%%%%%%%%%%%%%%%%%%%%%%%%%%%%%%%%%%%%%%%%%%%%%%%%%%%%%%%%%%%%%%%

McMullen \cite{mcmullen} constructed quasigeodesics in the singular solv metric on $\widetilde S \times \RR$ 
using saddle connections in the singular flat metric associated to the pseudo-Anosov $f$.
Hamenstaedt \cite{hamenstaedt} considered the more general case of
word hyperbolic extensions of surface groups, and Mj and Sardar
\cite{mj-sardar} and Kapovich and Sardar \cite{kapovich-sardar}
consider the general case of hyperbolic groups arising as hyperbolic-by-hyperbolic groups.

\medskip In all the above cases, the authors construct quasigeodesics
by connecting segments of vertical flow lines by horizontal
paths\footnote{We warn the reader that other authors, for example
  Kapovich and Sardar \cite{kapovich-sardar}, use the opposite
  convention for which directions are horizontal or vertical.}.
However, there is no \emph{a priori} upper bound for the length of the
vertical segments, and so the projection from the fiber geodesic to
the quasi-geodesic by the vertical flow need not be coarsely distance
non-increasing.

\medskip
%In proving the effective statements in \Cref{section:effective}, we need these properties, particularly the coarse distance non-increasing property. See \Cref{prop:vertical flow distance decreasing}.  
As we note in the example below, these properties, particularly the coarse distance non-increasing property, need not hold for all quasigeodesics, and they depend on the details of the construction.  
For an example, consider a geodesic $\gamma$ in the hyperbolic plane intersecting a horocycle $h$. 
See the illustration in in \Cref{fig:qg} in the upper half space model of $\HH^2$.

\begin{figure}[h]
\begin{center}
\begin{tikzpicture}[scale=0.85]

\tikzstyle{point}=[circle, draw, fill=black, inner sep=1pt]

\draw (0, 0) -- (8, 0);

\draw (0, 1) -- (8, 1) node [label=above:$h$] {} ;

\draw (1, 0) arc (180:0:3) node [label=above right:$\gamma$] {};

\draw (1, 0) -- (1, 3) node [midway, label=left:$\alpha$] {} -- (7, 3) -- (7, 0);

\draw (1, 0) -- (4, 3) node [midway, label=left:$\beta$] {} -- (7, 0);

\end{tikzpicture}
\end{center}
\caption{Two quasigeodesics in the upper half space model of $\HH^2$.} \label{fig:qg}
\end{figure}

\medskip
The two paths $\alpha$ and $\beta$ in \Cref{fig:qg} are quasigeodesics. For all horocycles $h$ intersecting $\gamma$, the
vertical projection from $h$ onto $\beta$ is uniformly
coarsely distance non-increasing, that is, there are constants $D, K, C > 0$ such that for any horocycle $h$ and any pair of points distance at least $D$ apart in the path metric on $h$, their projections to $\beta$ are distance at most $K$ times their distance along $h$ plus $C$. However, the
vertical projection from $h$ to $\alpha$ does not have this property since a uniform $D$ does not exist for $h$ as its height (imaginary part) goes to zero. 

\medskip
In this article, we prove that both properties hold for the test paths we build by relating properties of the test paths to a function on the unit tangent bundle of the surface.

%See TBA.
%In order to apply the ergodicity of the geodesic flow, we need to relate properties of quasigeodesics to a function on the unit tangent bundle of the surface.  

%We expect that it should be 
It may be possible to modify the
previous constructions of quasigeodesics to have the desired
properties, though we do not attempt it here.

%%%%%%%%%%%%%%%%%%%%%%%%%%%%%%%%%%%%%%%%%%%%%%%%%%%%%%%%%%%%%%%%%%%%%%%%%%%%%%%%%%
\subsection{A specific construction of quasigeodesics}\label{section:effective}
%%%%%%%%%%%%%%%%%%%%%%%%%%%%%%%%%%%%%%%%%%%%%%%%%%%%%%%%%%%%%%%%%%%%%%%%%%%%%%%%%%

In this section, we state the construction of the test paths precisely and also the properties that they satisfy.

\begin{definition}\label{def:non-exceptional} 
We say that a bi-infinite geodesic $\gamma$ in $\widetilde S_h$ is
\emph{non-exceptional} if
\begin{itemize}
\item its limit points $\gamma_-$ and $\gamma_+$ are distinct from the
limit points of any boundary leaf of any ideal complementary region of
either of the invariant laminations, and
\item the limit points have distinct images under the Cannon--Thurston
map, that is $\iota(\gamma_-) \not = \iota(\gamma_+)$.
\end{itemize}
\end{definition}

\medskip
Suppose that $\gamma$ is a non-exceptional geodesic in the universal
cover of the surface $\ws$ with unit speed
parametrization.  The inclusion map $\iota$ embeds $\gamma$ in $\wsr$
at height $z = 0$.  The image $\iota(\gamma)$ is, in general, not a
quasigeodesic.  However, we will show that the height for each point of $\iota(\gamma)$ can be changed in a specified way so that the resulting path is a quasigeodesic, i.e. there is a
function $h_\gamma(t)$ such that $(\gamma(t), h_\gamma(t))$ is an
(unparametrized) quasigeodesic.

\medskip
In order to define the function $h_\gamma$, we need the following mild
generalization of a measured lamination.  A
measured lamination is \emph{maximal} if every complementary region is
an ideal triangle.  By adding finitely many leaves to divide every ideal complementary region of a measured lamination $\Lambda$ into ideal triangles we may extend $\Lambda$ to a maximal lamination. There are thus only finitely many
such maximal laminations containing $\Lambda$.  We will call the union
of these maximal laminations the \emph{extended lamination}
$\overline{\Lambda}$, which in general is not itself a measured
lamination. See \Cref{def:extended_laminations} and \Cref{sec:extended_laminations} for more details.

\medskip
Let $\oLambda$ be the union of the two extended laminations obtained from the invariant laminations
for $\pa$, and let $\oLambda^1$ be its lift in $T^1(S_h)$.  Let
$h \colon T^1(S_h) \setminus \oLambda^1 \to \RR$ be a continuous
function.  Then $h$ determines an embedding in $\wsr$ of any non-exceptional
unit-speed geodesic $\gamma$ by the map $t \mapsto (\gamma(t), h(\gamma^1(t)))$,
where $\gamma(t)$ is the oriented geodesic and
$\gamma^1(t)$ is the unit tangent vector to $\gamma$ at $\gamma(t)$.  We
shall call $h$ the \emph{height function}, and the embedding
$\tau_\gamma(t) = ( \gamma(t), h(\gamma^1(t)) )$ the \emph{test path}
for $\gamma$ determined by $h$.

\medskip
We now specify the height function we will use, see
\Cref{section:quasigeodesics} for background and motivation.

\medskip

We shall write $\log_k$ for the logarithm function with base $k$,
where $k = k_f > 1$ is the stretch factor of $\pa$.
\begin{definition}\label{def:height function2}
Suppose that $\pa \colon S \to S$ is a pseudo-Anosov map, let
$(S_h, \Lambda)$ is a hyperbolic metric on $S$ together with a pair of
invariant measured laminations, and $\oLambda^1_-$ and
$\overline{\Lambda}^1_+$ the lifts in $T^1(S_h)$ of the extended
laminations given by the invariant laminations of $\pa$, and let
$\theta > 0$ be a positive constant.
We define the \emph{height function}
$h_\theta \colon T^1(S_h) \setminus (\overline{\Lambda}^1_+ \cup
\overline{\Lambda}^1_-) \to \RR$ to be
\[ h_\theta(v) = \log_k \left\lfloor \log \frac{1}{d_{\PSL(2, \RR)}(v,
  \overline{\Lambda}^1_+ )} - \log \frac{1}{\theta} \right\rfloor_1
- \log_k \left\lfloor \log \frac{1}{d_{\PSL(2, \RR)}(v,
  \overline{\Lambda}^1_- )} - \log \frac{1}{\theta}
\right\rfloor_1 \]
\end{definition}

Here $\lfloor x \rfloor_c = \max \{ x, c \}$ is the standard floor
function.  As the two extended laminations are a positive distance
apart in $T^1(S_h)$, for sufficiently small $\theta$, at most one
of the terms on the right hand side above will be non-zero.

\medskip
We prove that for a choice of $\theta$ sufficiently small, the test path determined
by the corresponding height function is a quasigeodesic in $\wsr$. The following is the main result of this article: 

\begin{theorem}\label{theorem:quasigeodesic-h2}
Suppose that $\pa \colon S \to S$ is a pseudo-Anosov map and
$(S_h, \Lambda)$ is a hyperbolic metric on $S$ together with a pair of
invariant measured laminations.  Then there are constants
$\theta > 0$, $Q \ge 1$ and $c \ge 0$, such that for any
non-exceptional geodesic $\gamma$ in $S_h$, with a unit speed
parametrization $\gamma(t)$, the test path
$\tau_\gamma(t) = (\gamma(t), h_\theta(\gamma^1(t)) )$ is an
unparametrized $(Q, c)$-quasigeodesic in $\widetilde S_h \times \RR$
with the same limit points as $\iota(\gamma)$, where $h_\theta$ is the
height function from \Cref{def:height function2}.
\end{theorem}

\medskip
We shall now fix a sufficiently small constant $\theta$ in \Cref{theorem:quasigeodesic-h2} and simplify notation to just write $h$ for $h_\theta$.  See
\Cref{section:height function} for the exact choice of $\theta$ that
we use.  Furthermore, we will write $h_\gamma(t)$ for
$h_\theta(\gamma^1(t))$.

\medskip
We will use one further property of these quasigeodesics.  The test
path $\tau_\gamma$ lies in the ladder $F(\gamma)$ determined by
$\gamma$, so vertical projection gives a map
$(\gamma(t), 0) \mapsto (\gamma(t), h_\gamma(t))$ from $\iota(\gamma)$
to the test path $\tau_\gamma$.  We will prove that this map is coarsely distance non-increasing.

\begin{proposition}\label{prop:vertical flow distance decreasing}
Suppose that $\pa \colon S \to S$ is a pseudo-Anosov map and
$(S_h, \Lambda)$ is a hyperbolic metric on $S$ together with a pair of
invariant measured laminations.  There are constants $K > 0$ and
$c \ge 0$ such that for any non-exceptional geodesic $\gamma$ with
unit speed parametrization, and any real numbers $s$ and $t$,
\[ d_{\wsr}( \tau_\gamma(s), \tau_\gamma(t) ) \le K d_{\wsr} \left(
\iota(\gamma(s)), \iota(\gamma(t)) \right) + c , \]
where the test path has the parametrization inherited from the
unit speed parametrization on $\gamma$.
\end{proposition}

%%%%%%%%%%%%%%%%%%%%%%%%%%%%%%%%%%%%%%%%%%%%%%%%%%%%%%%%%%%%%%%%%%%%%%%%%%%%%%%%%%
\subsection{Overview}
%%%%%%%%%%%%%%%%%%%%%%%%%%%%%%%%%%%%%%%%%%%%%%%%%%%%%%%%%%%%%%%%%%%%%%%%%%%%%%%%%%

In this section we give a brief overview of the argument, omitting
many technical details.  Each subsequent section has its own
introduction discussing these details further.

In \Cref{section:background2} we review the background material on
hyperbolic geometry and laminations that we will use, and set up some
notation.  The key observation is the following.  A \emph{rectangle}
is a subset of the universal cover of the fiber surface whose boundary
consists of two pairs of arcs, one pair from each of the stable
and unstable lamination, see for example \Cref{fig:opposite
  quadrants}.  The image of a rectangle under the vertical flow is
again a rectangle, but with one pair of sides scaled by a factor of
$k^z$, and the other pair scaled by a factor of $k^{-z}$, where $z$ is
the coordinate in the vertical flow direction.  The \emph{optimal
  height} of the rectangle is the $z$ value which makes both sides
equal length.

The bi-infinite leaves containing the sides of the rectangle divide
the plane into nine regions.  We call a non-compact region whose
boundary contains two arcs, one from each lamination, a
\emph{quadrant}.  The union of the vertical flow lines through a
quadrant is quasiconvex.  In \Cref{section:tame} we show that the
coarse intersection of two opposite quadrants is a regular
neighborhood of the optimal height rectangle, where the size of the
regular neighborhood depends on the area of the rectangle.  We call
this regular neighborhood of the optimal height rectangle a
\emph{bottleneck}, as any geodesic in the universal cover of the
$3$-manifold with endpoints in opposite quadrants passes through this
set.

In \Cref{section:small angle bottleneck} we show that whenever a
geodesic $\gamma$ in the universal cover of the fiber crosses a
lamination with small angle, then it crosses a rectangle with area
bounded below.  This controls the height of any quasigeodesic in the
Cannon--Thurston metric which connects the endpoints of the image of
$\gamma$ in the universal cover of the $3$-manifold $\widetilde{M}$.
In particular, the optimal height is approximately $\log \log \theta$,
where $\theta$ is the angle of intersection between $\gamma$ and the
lamination, with the sign depending on whether the lamination is the
stable or the unstable lamination.  Given our choice of the height function in
\Cref{def:height function2}, this shows that the test paths over
segments containing small angles of intersection stay within a bounded
distance of a geodesic in $\widetilde{M}$.

In \Cref{section:large angles}, we show that whenever a short
subinterval of the geodesic $\gamma$ meets both laminations with large
angles, $\gamma$ again crosses a rectangle of area bounded below.  The
final case to consider is when a segment of $\gamma$ spends a long
time in a single complimentary region.  In this case it does not
necessarily cross rectangles of area bounded below, but as we have an
explicit description of the Cannon--Thurston metric in these regions,
we show directly that the test paths we construct are quasigeodesic in
\Cref{section:straight qg}.

Finally, in \Cref{section:proj}, we verify that vertical projection
along the flow from the fiber to the quasigeodesic test path is
coarsely distance decreasing.

%%%%%%%%%%%%%%%%%%%%%%%%%%%%%%%%%%%%%%%%%%%%%%%%%%%%%%%%%%%%%%%%%%%%%%%%%%%%%%%%%%
\subsection{Acknowledgments}
%%%%%%%%%%%%%%%%%%%%%%%%%%%%%%%%%%%%%%%%%%%%%%%%%%%%%%%%%%%%%%%%%%%%%%%%%%%%%%%%%%

This work would not have been possible without the wonderful working
conditions of the American Institute of Mathematics. We would like to
thank them for their hospitality during the April 2022 workshop
“Random walks beyond hyperbolic groups” where this work started.  We
thank the referee of a previous version of this paper for useful
comments.

\medskip
The first author acknowledges the support of the Institut Henri
Poincaré (UAR 839 CNRS-Sorbonne Université) and LabEx CARMIN
(ANR-10-LABX-59-01).  The
second author thanks PSC-CUNY and the Simons Foundation for their
support.  The third author is grateful to the Institute for Advanced
Study for their hospitality and the Bob Moses fund for funding her
2024--2025 membership, she is funded by an NSERC Discovery Grant. The
last author gratefully acknowledges support from the NSF grant DMS--2439076.

%%%%%%%%%%%%%%%%%%%%%%%%%%%%%%%%%%%%%%%%%%%%%%%%%%%%%%%%%%%%%%%%%%%%%%%%%%%%%%%%%% 
\section{Hyperbolic geometry and (extended) laminations}\label{section:background2}
%%%%%%%%%%%%%%%%%%%%%%%%%%%%%%%%%%%%%%%%%%%%%%%%%%%%%%%%%%%%%%%%%%%%%%%%%%%%%%%%%%

In this section, we briefly review some useful facts about hyperbolic
geometry and  properties of hyperbolic geodesics. We also discuss properties of measured laminations and extended laminations. We refer reader to \cite{bh}, \cite{Casson-Bleiler}, and \cite{kapovich} for details.

%%%%%%%%%%%%%%%%%%%%%%%%%%%%%%%%%%%%%%%%%%%%%%%%%%%%%%%%%%%%%%%%%%%%%%%%%%%%%%%%%% 
\subsection{Geodesic laminations and pseudo-Anosov maps}
\label{sec:laminationsandpas}
%%%%%%%%%%%%%%%%%%%%%%%%%%%%%%%%%%%%%%%%%%%%%%%%%%%%%%%%%%%%%%%%%%%%%%%%%%%%%%%%%%

A \emph{geodesic lamination} on $S_h$ is a closed
union of simple pairwise disjoint geodesics.  
A \emph{transverse measure} on a geodesic
lamination $\Lambda$ is a positive measure $dm$ defined on local transverse arcs to the leaves of $\Lambda$ that 
\begin{itemize}
    \item is invariant under any isotopy preserving the transverse
    intersections with the leaves of $\Lambda$, and
\item is positive and finite on any
nontrivial compact transversals.
\end{itemize}
Such a measure lifts to a $\pi_1(S)$-invariant transverse measure on
the pre-image of $\Lambda$ in $\HH^2$.  By abuse of notation, we will
denote the pre-image also by $\Lambda$ and the lifted transverse measure also by
$dm$.  A geodesic lamination equipped with a transverse measure is
called a \emph{measured lamination}.  We will only consider measured
laminations, and so we will often just write lamination to mean
measured lamination.

\medskip

We say a measured lamination is \emph{filling} if there are no
essential simple closed curves disjoint from the lamination.  The
complement of a filling lamination is a union of ideal polygons with
finitely many sides. We say a leaf of the lamination is a
\emph{boundary leaf} if it is the boundary of an ideal polygon.  There
are finitely many ideal complementary regions in the compact
surface $S_h$, and so there are only finitely many boundary leaves in
$S_h$. This implies that there are countably many boundary leaves in
the universal cover $\widetilde S_h$.

\medskip

We say a measured lamination is
\emph{minimal} if every leaf is dense in the lamination.  A minimal
filling lamination has the following properties:
\begin{itemize}
    \item there are uncountably many leaves,
    \item no leaf is isolated, and
    \item the transverse measure is non-atomic.
\end{itemize}

We say a pair of measured laminations $\Lambda_+$ and $\Lambda_-$
\emph{bind} $S_h$ if each geodesic ray on $S_h$ crosses a leaf of
$\Lambda_+ \cup \Lambda_-$.

\begin{definition}
We shall write $(S_h, \Lambda)$ for a triple consisting of a
hyperbolic metric on a compact surface $S$, together with a pair of
minimal filling measured laminations $\Lambda_+$ and $\Lambda_-$ which
bind the surface.  We refer to such a triple $(S_h, \Lambda)$ as a
\emph{hyperbolic surface and a full pair of laminations}.
\end{definition}

A pseudo-Anosov map $\pa$ determines a pair of invariant measured
laminations called stable and unstable
laminations, which we shall denote
$(\Lambda_+, dx)$ and $(\Lambda_-, dy)$ respectively, where $dx$ and $dy$ are the transverse measures.  In particular, 
\begin{itemize}
    \item $f (\Lambda_+) = \Lambda_+$ and $f_\ast dx = k dx$, and
    \item $f(\Lambda_-) = \Lambda_-$ and $f_\ast dy = k^{-1} dy$,
\end{itemize}
where $k = k_\pa > 1$ is known as the \emph{stretch factor} of the pseudo-Anosov map.
We shall often write $\Lambda_+$ or $\Lambda_-$ to refer to the measured laminations if we
do not need to refer to the respective measures.  

%In the coordinate system described in \Cref{section:pseudometric}, leaves of $\Lambda_-$ correspond to lines parallel to the $x$-axis, and leaves of $\Lambda_+$ correspond to lines parallel to the $y$-axis.

We will use the following useful properties of geodesic laminations.

\begin{proposition}\cite{gmpueffective}*{Proposition 11}\label{prop:multi}
Suppose that $(S_h, \LL)$ is a hyperbolic surface together with a full
pair of measured laminations.  Then there exist constants
$\alpha_\LL, \epsilon_\LL, L_\LL > 0$ such that each of the following
properties holds for any pair of leaves $(\ell_+,\ell_-)$ with
$\ell_+$ in $\Lambda_+$ and $\ell_-$ in $\Lambda_-$:

\begin{thmenum}[label={(\ref{prop:multi}.\arabic*)}]

\item \label{prop:angle bound} If $\ell_+$ and $\ell_-$ intersect, then their angle of
intersection is at least $\alpha_\LL$.

\item \label{prop:disjoint leaves} 
If $\ell_+$ and $\ell_-$ are disjoint,
the distance between any two lifts of $\ell_+$ and $\ell_-$ in
$\widetilde S_h$ is at least $\epsilon_\LL$.

\item \label{prop:cobounded intersections}\label{c:L_pa} Any segment of a leaf of one of the
laminations of length at least $L_\LL$ intersects a leaf of the other
lamination.

\item \label{prop:complementary regions don't share} No ideal
complementary region in $\widetilde S_h \setminus \Lambda_+$ has an
ideal vertex in common with an ideal complementary region of
$\widetilde S_h \setminus \Lambda_-$.

\end{thmenum}

\end{proposition}

Cannon and Thurston showed that the Cannon--Thurston metric is
quasi-isometric to the hyperbolic metric.  We state this explicitly
below to fix notation for the quasi-isometry constants.

\begin{theorem}\cite{cannon-thurston}*{Theorem 5.1}\label{prop:qi}\label{c:Q_pa}\label{c:c_pa}
Suppose that $(S_h, \Lambda)$ is a hyperbolic metric on $S$ together
with a full pair of measured laminations.  Then the hyperbolic
metric $d_{\HH^2}$ and the Cannon--Thurston pseudometric
$d_{\widetilde S_h}$ on the universal cover $\widetilde S_h$ are
quasi-isometric, i.e. there are constants $Q_\LL \ge 1$ and
$c_\LL \ge 0$ such that for any points $p$ and $p'$ in the universal
cover,
\[ \frac{1}{Q_\LL} d_{\widetilde S_h}(p, p') - c_\LL \le d_{\HH^2}(p,
p') \le Q_\LL d_{\widetilde S_h}(p, p') + c_\LL. \]
\end{theorem}

%%%%%%%%%%%%%%%%%%%%%%%%%%%%%%%%%%%%%%%%%%%%%%%%%%%%%%%%%%%%%%%%%%%%%%%
\subsection{Close geodesics diverge exponentially}\label{section:exponential}
%%%%%%%%%%%%%%%%%%%%%%%%%%%%%%%%%%%%%%%%%%%%%%%%%%%%%%%%%%%%%%%%%%%%%%%

It is well known that if the lifts of two geodesics in $\PSL(2, \RR)$
are very close, then the distance between them increases exponentially
as you move away from the closest point between them until they are a
definite distance apart, and then grows linearly.  More precisely, if
they are distance $\theta$ apart, then the distance between them grows
exponentially for a distance of length roughly
$\log \tfrac{1}{\theta}$.  We will use the following version of this
result.

\begin{proposition}\label{prop:fellow travel}
There are constants $\theta_0 > 0$ and $L_0 \ge 1$, such that for any
two geodesics $\gamma_1$ and $\gamma_2$ in $\HH^2$ whose lifts in
$\PSL(2, \RR)$ have closest points distance $\theta \le \theta_0$
apart, if $\gamma_1$ has a unit speed parametrization $\gamma_1(t)$
with the closest point to $\gamma_2$ being $\gamma_1(0)$, then for all
$|t| \le \log \frac{1}{\theta}$,
\begin{equation}\label{eq:fellow travel}
\tfrac{1}{L_0} \theta e^{|t|} \le d_{\PSL(2, \RR)} ( \gamma^1_1(t),
\gamma_2^1 ) \le L_0 \theta e^{|t|}.
\end{equation}
Furthermore, the lower bound at $|t| = \log \tfrac{1}{\theta}$ holds
for all $t$ outside this range, i.e. for all
$|t| \ge \log \tfrac{1}{\theta}$,
\[ d_{\PSL(2, \RR)} ( \gamma^1_1(t), \gamma^1_2 ) \ge \tfrac{1}{L_0}. \]
\end{proposition}
We give a detailed proof of this result in Appendix
\ref{section:fellow travel} for the convenience of the reader.

\begin{definition}
\label{def:exponential interval}
Suppose that $\ell$ and $\gamma$ are geodesics with $\gamma$
parametrized with unit speed.  Suppose $\gamma(t)$ is the closest
point of $\gamma$ to $\ell$ and suppose that
$d_{\PSL(2, \RR)} (\gamma^1(t), \ell^1) = \theta$.  We call the interval
    \[
    E_\ell = [t -\log \tfrac{1}{\theta}, t_\ell + \log \tfrac{1}{\theta}] 
    \]
the \emph{exponential interval} for $\gamma$
with respect to $\ell$.
\end{definition}
Note that $E_\ell$ is the interval on which the exponential
estimate \eqref{eq:fellow travel} holds.

%%%%%%%%%%%%%%%%%%%%%%%%%%%%%%%%%%%%%%%%%%%%%%%%%%%%%%%%%%%%%%%%%%%%%%%%%%%%%%%%%%
\subsection{Quasigeodesics}
%%%%%%%%%%%%%%%%%%%%%%%%%%%%%%%%%%%%%%%%%%%%%%%%%%%%%%%%%%%%%%%%%%%%%%%%%%%%%%%%%%

We recall some basic facts about quasigeodesics, see for example
Bridson and Haefliger \cite{bh}*{III.H}.

\begin{definition}
Let $(X, d)$ be a geodesic metric space and let
$\gamma \colon I \to X$ be a path, where $I$ is a (possibly infinite)
connected subset of $\RR$.  Let $Q \ge 1$ and $c \ge 0$ be constants.
The path $\gamma$ is a \emph{(Q, c)-quasigeodesic} if for all $t_1$
and $t_2$ in $I$,
\[ \frac{1}{Q} | t_2 - t_1 | - c \le d( \gamma(t_1), \gamma(t_2) )
\le Q |t_2 - t_1| + c.  \]
\end{definition}

By \cite{bh}*{III.H Lemma 1.11}, given a $(Q, c)$-quasigeodesic, there
is a continuous $(Q, c')$-quasigeodesic with the same endpoints, so
for our purposes we may assume that all quasigeodesics are continuous,
and we will do so from now on.

\medskip
A \emph{reparametrization} of a path $\gamma \colon \RR \to X$ is the
path $\gamma \circ \rho$, where $\rho \colon \RR \to \RR$ is a proper
non-decreasing function.  We say a path $\gamma \colon \RR \to X$ is
an \emph{unparametrized} $(Q, c)$-quasigeodesic if there is a
reparametrization of $\gamma$ which is a $(Q, c)$-quasigeodesic.

\medskip
We will use the following stability property for quasigeodesics in
hyperbolic spaces, known as the Morse Lemma.

\begin{lemma}\cite{bh}*{Theorem 1.7}\label{lemma:morse}
Let $X$ be a $\delta$-hyperbolic space.  Then for any $Q$ and $c$
there is a constant $L$ such that any $(Q, c)$-quasigeodesic is
contained in an $L$-neighborhood of the geodesic connecting its
endpoints.
\end{lemma}

%%%%%%%%%%%%%%%%%%%%%%%%%%%%%%%%%%%%%%%%%%%%%%%%%%%%%%%%%%%%%%%%%%%%%%%%%%%%%%%%%%
\subsection{Nearest point projections and fellow traveling}
%%%%%%%%%%%%%%%%%%%%%%%%%%%%%%%%%%%%%%%%%%%%%%%%%%%%%%%%%%%%%%%%%%%%%%%%%%%%%%%%%%

Suppose that $\alpha$ is a subset of a Gromov hyperbolic space
$(X, d)$.  The \emph{nearest point projection}
$p_\alpha \colon X \to \alpha$ sends each point $x \in X$ to a closest point to $x$ in $\alpha$.  If $\alpha$ is $Q$-quasiconvex,
then the nearest point projection is $K$-coarsely well defined, where
$K$ depends only on $Q$ and the constant $\delta$ of hyperbolicity.  

\medskip Suppose that $\alpha$ and $\beta$ are two geodesics in $X$. 
We define the \emph{$K$-fellow traveling set} for $\alpha$ with respect to $\beta$ to be the subset
of $\alpha$ contained in a $K$-neighborhood of $\beta$, i.e.
$\alpha \cap N_K(\beta)$.  If the diameter of the projection image
$p_\alpha(\beta)$ is sufficiently large, then $p_\alpha(\beta)$ is
contained in a bounded neighborhood of the geodesic $\beta$, and so is
contained in a fellow traveling set.

\medskip
In the special case that $X$ is $\HH^n$ and $\alpha$
is a geodesic, the closest point on $\alpha$ to any point $x \in X$ is
unique, and so $p_\alpha$ is well-defined.  Furthermore, for any geodesic
$\beta$, the image $p_\alpha(\beta)$ is a subinterval of $\alpha$,
which we will refer to as the \emph{nearest point projection
  interval}, or just the \emph{projection interval}.  Similarly, any
$K$-fellow traveling set $\alpha \cap N_K(\beta)$ is also an
interval, which we shall call the \emph{$K$-fellow traveling
  interval}.

%\medskip

% For $\HH^n$, the hyperbolicity constant is $\delta = 2 \log 3$.  We shall write $\delta_2$ for the hyperbolicity constant for the pseudometric
% $d_{\widetilde S_h}$ on $\widetilde S_h$, and $\delta_3$ for the
% hyperbolicity constant for the pseudometric
% $d_{\widetilde S_h \times \RR}$ on $\widetilde S_h \times \RR$.  Both
% constants depend on the pseudo-Anosov $\pa$.

\medskip
A standard consequence of $\delta$-hyperbolicity is the useful property stated below that if the projection image of a geodesic $\beta$
onto another geodesic $\alpha$ is large, then the two geodesics fellow
travel and the projection image $p_\alpha(\beta)$ is contained in a
bounded neighborhood of $\beta$.

\begin{proposition} \cite{kapovich-sardar}*{Lemma 1.120}
\label{prop:neighbourhood}
Suppose that $X$ is a $\delta$-hyperbolic space and suppose that
$\alpha$ and $\beta$ are geodesics in $X$.  If the diameter of the
projection image $p_\alpha(\beta)$ is greater than $8 \delta$, then
the projection image is contained in a $6 \delta$-neighborhood of
$\beta$, i.e. $p_\alpha(\beta) \subseteq N_{6 \delta}(\beta)$, and so
$p_\alpha(\beta)$ is contained in the fellow traveling set
$\alpha \cap N_{6 \delta}(\beta)$.
\end{proposition}

In $\HH^2$, if two geodesics intersect at angle $\theta$, then the
size of the projection interval is roughly $\log (1/\theta)$.  In
fact,the same result holds for two geodesics that do not intersect,
but are distance $\theta$ apart.  We will use these properties in
\Cref{section:quasigeodesics} below.

\begin{proposition}\label{prop:projection interval}\label{c:T_0}
There is a constant $T_0 \ge 0$, such that for any unit speed geodesic
$\gamma_1$ in $\HH^2$ and any geodesic $\gamma_2$ such that

\begin{itemize}

\item $\gamma_2$ intersects $\gamma_1$ at the point $\gamma_1(0)$ at
an angle $0 < \theta \le \pi/2$, or

\item the distance from $\gamma_1$ to $\gamma_2$ is $\theta > 0$, and
the closest point occurs at $\gamma_1(0)$,

\end{itemize}

then the nearest point projection interval $p_{\gamma_1}(\gamma_2)$
is equal to $\gamma_1( [-T, T] )$, where
\[ \log \frac{1}{\theta} \le T \le \log \frac{1}{\theta} +
T_0, \]
and furthermore, for all $|t| \le \log \tfrac{1}{\theta}$, the
distance from $\gamma_1(t)$ to $\gamma_2$ is at most $3/2$.
\end{proposition}
This is well known, we provide the details in Appendix
\ref{section:hyperbolic} for the convenience of the reader.  In fact,
for small $|t|$, the two geodesics are exponentially close, see
\Cref{section:exponential} for further details.

\medskip We now record the useful fact that if two geodesics $\alpha$ and
$\beta$ intersect at angle $\theta$, then the size of their projection
intervals onto each other is roughly $\log \tfrac{1}{\theta}$, and
furthermore, for any other geodesic $\gamma$, the overlap between the
projection intervals for $\alpha$ and $\beta$ on $\gamma$ is bounded
in terms of $\theta$.

\begin{proposition}\label{prop:overlap}\cite{gmpueffective}*{Proposition 26}
For any constant $\alpha_\LL > 0$ there is a constant $\rho_\LL > 0$
such that for any two geodesics in $\HH^2$ which intersect at angle
$\theta \ge \alpha_\LL$, and for any other geodesic $\gamma$, the
intersection of the nearest point projection intervals of $\alpha$ and $\beta$ to $\gamma$
has diameter at most $\rho_\LL$.\qed
\end{proposition}

If two geodesics $\alpha$ and $\beta$ have
strictly nested projection intervals onto a third geodesic $\gamma$,
i.e. $p_{\gamma}(\alpha) \subset p_\gamma(\beta)$, and $\alpha$
intersects $\gamma$, then $\alpha$ and $\beta$ also intersect.

\begin{proposition}\cite{gmpueffective}*{Proposition 27}\label{prop:nested implies intersect}
Let $\gamma$ be a geodesic in $\HH^2$ which intersects a geodesic
$\ell_1$, with projection interval $I_1 \subset \gamma$.  Let $\ell_2$
be a geodesic with projection interval $I_2 \subset \gamma$, such that
$I_1 \subset I_2$.  Then $\ell_1$ and $\ell_2$ intersect.\qed
\end{proposition}

% \begin{proof}
% Consider the nearest point projection map $p \colon \HH^2 \to \gamma$.
% Consider the complement of the pre-image of $I_1$, i.e.
% $\HH^2 \setminus p^{-1}(I_1)$.  This has two connected components,
% which are separated by the geodesic $\ell_1$.  As $I_1$ is a strict
% subset of $I_2$, each endpoint of $\ell_2$ is contained in a different
% complementary component, and so the endpoints of $\ell_2$ are
% separated by $\ell_1$, and so the two geodesics $\ell_1$ and $\ell_2$
% intersect.
% \end{proof}

%%%%%%%%%%%%%%%%%%%%%%%%%%%%%%%%%%%%%%%%%%%%%%%%%%%%%%%%%%%%%%%%%%%%%%%%%%%%%%%%%%
\subsection{Extended laminations}
\label{sec:extended_laminations}
%%%%%%%%%%%%%%%%%%%%%%%%%%%%%%%%%%%%%%%%%%%%%%%%%%%%%%%%%%%%%%%%%%%%%%%%%%%%%%%%%%

Given a full lamination on a hyperbolic surface, we define the resulting extended lamination as below.

\begin{definition}
\label{def:extended_laminations}
Suppose that $\Lambda$ is a measured lamination on a hyperbolic surface $S_h$, and suppose that all complementary regions of $\Lambda$ are ideal polygons. 
We define the \emph{extended lamination} $\overline{\Lambda}$ to be the union of
$\Lambda$ with additional leaves connecting every pair of ideal points for each ideal complementary region.  We will call the additional
leaves \emph{extended leaves}.
\end{definition}

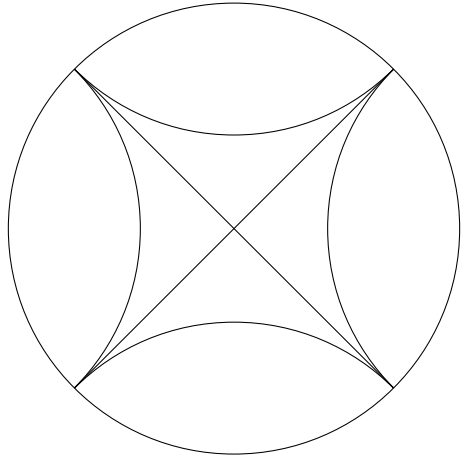
\begin{figure}[h]
\begin{center}
\begin{tikzpicture}[scale=0.75]

\tikzstyle{point}=[circle, draw, fill=black, inner sep=1pt]

\draw (0, 0) circle (4);

\begin{scope}
\clip (0, 0) circle (4);
\draw (0:5.66) circle (4);
\draw (90:5.66) circle (4);
\draw (180:5.66) circle (4);
\draw (270:5.66) circle (4);
\end{scope}

\draw (45:4) -- (225:4);
\draw (135:4) -- (315:4);

\end{tikzpicture}
\end{center}
\caption{Extended leaves in a complementary region with four
  sides.} \label{fig:extended lamination}
\end{figure}

If all ideal complementary regions are triangles, then the extended
lamination equals the original lamination.  If there are ideal polygons
with more than three vertices, the extended lamination is not even a
geodesic lamination, as there are intersecting leaves.  Extended
laminations are not minimal, as not every leaf is dense, but they are
still closed.  Although the extended lamination is not a lamination,
it still makes sense to assign measures to transverse arcs, and if we
assign the extended leaves measure zero, the resulting measure is the
same as the measure from the original measured lamination.  As the
only way to obtain a finite measure is to make the extended leaves
measure zero, the extended lamination is also not a geodesic current
as there are transverse arcs of zero measure.

\begin{proposition}\label{prop:non common leaves extended}
Suppose that $(S_h, \LL)$ is a hyperbolic metric on $S$ together with
a full pair measured laminations.  Then the corresponding extended laminations
$\overline{\Lambda}_+$ and $\overline{\Lambda}_-$
do not share any common leaves.
\end{proposition}

\begin{proof}
Suppose that the two extended laminations share a common leaf.  The
measured laminations do not share any common leaves, so the common
leaf must be an extended leaf.  But then the two laminations have ideal complementary regions that share a common point at
infinity, contradicting \Cref{prop:complementary regions don't share}.
\end{proof}

\begin{proposition}\label{prop:extended angle bound}
Suppose that $(S_h, \LL)$ is a hyperbolic metric on $S$ together with
a full pair of measured laminations.  Then there is a constant
$\alpha_\LL > 0$ such that any two leaves of the extended laminations
$\overline{\Lambda}_+$ and $\overline{\Lambda}_-$ which intersect,
intersect at angle at least $\alpha_\LL$.
\end{proposition}

Essentially the same proof as before works, but we give the details
for the convenience of the reader.

\begin{proof}
Suppose that there is a sequence of pairs of intersecting leaves
$\ell^-_n, \ell^+_n$, whose angles of intersection tend to zero.  By
compactness of $S$, we may pass to a convergent subsequence.  As
the extended laminations are closed, this limits to a pair of leaves with zero angle
of intersection, so $\overline{\Lambda}_+$ and $\overline{\Lambda}_-$
share a common leaf, contradicting \Cref{prop:non common leaves extended}.
\end{proof}

Every geodesic $\gamma$ in $S_h$ has a unique lift in the unit tangent
bundle $T^1(S_h)$.

\begin{proposition}\label{prop:unit tangent bound}
Suppose that $(S_h, \LL)$ is a hyperbolic metric on $S$ together with
a full pair of measured laminations.  Then there is a constant
$\alpha_\LL > 0$ such that the distance in $T^1(S_h)$ between any two
leaves of the extended laminations $\overline{\Lambda}^1_+$ and
$\overline{\Lambda}^1_-$ is at least $\alpha_\LL$.
\end{proposition}

\begin{proof}
Suppose that there is a sequence of pairs of leaves
$\ell^-_n, \ell^+_n$, whose lifts become arbitrarily close in
$T^1(S_h)$.  By compactness of $S$, we may pass to a convergent
subsequence.  As the extended laminations are closed, this limits to a
pair of leaves which are tangent, and hence equal, implying that the
extended laminations share a common leaf, contradicting \Cref{prop:non common leaves extended}.
\end{proof}

Finally, we record the fact that any sufficiently long segment of a
leaf of the extended lamination intersects the other extended lamination.
Essentially, the same argument as before works, but we give the details below.

\begin{proposition}\label{prop:extended intersections}
Suppose that $(S_h, \Lambda)$ is a hyperbolic metric on $S$ together
with a full pair of measured laminations.  Then there is a constant
$\overline{L}_\LL > 0$ such that any leaf of either of the extended
laminations $\overline{\Lambda}_+$ or $\overline{\Lambda}_-$ of length
at least $\overline{L}_\LL$, intersects the other lamination.
\end{proposition}

\begin{proof}
Breaking symmetry, suppose that the leaf belongs to $\overline{\Lambda}_+$. By swapping the laminations, the same argument works for $\overline{\Lambda}_-$.

\medskip
By \Cref{prop:cobounded intersections}, there is a constant $L_\LL$
such that every segment of a leaf of the invariant lamination $\Lambda_+$ of length at least $L_\LL$, intersects $\Lambda_-$. So it
suffices to consider extended leaves. Let $\sigma_n$ be a sequence of
segments of extended leaves $\ell_n \in \overline{\Lambda}_+$, such
that the $\sigma_n$ are disjoint from $\Lambda_-$, and the length of
the segments $\sigma_n$ tends to infinity.  As there are finitely many extended leaves for $\overline{\Lambda}_+$ in $S_h$, we may pass to a subsequence
that converges to an infinite ray of an extended leaf of $\overline{\Lambda}_+$ such that the ray is
disjoint from $\Lambda_-$. Note that such a ray is asymptotic to
a boundary leaf $\ell$ of $\Lambda_+$.  By Proposition \ref{prop:angle bound}
there is a lower bound $\alpha_\LL$ on the angle of intersection
between the leaves of the two invariant laminations. It follows that an
infinite subray of $\ell$ is also disjoint from $\Lambda_-$,
contradicting \Cref{prop:cobounded intersections}.
\end{proof}

%%%%%%%%%%%%%%%%%%%%%%%%%%%%%%%%%%%%%%%%%%%%%%%%%%%%%%%%%%%%%%%%%%%%%%%
\subsection{Complementary regions, polygons and cusps}
%%%%%%%%%%%%%%%%%%%%%%%%%%%%%%%%%%%%%%%%%%%%%%%%%%%%%%%%%%%%%%%%%%%%%%%
\label{sec:polygons}

In this section, we fix notation to describe subsets of either the
surface $S_h$, or the universal cover $\widetilde S_h$, which have
boundaries consisting of alternating arcs of the laminations $\Lambda_+$ and
$\Lambda_-$.

\medskip
Suppose that $\Lambda$ is an invariant lamination. Then $S_h \setminus \Lambda$ has finitely
many connected components, and each connected component lifts to an ideal
polygon in $\widetilde S_h$ with its boundary consisting of leaves of $\Lambda$.  We shall call
these complementary regions \emph{ideal polygons with boundary in $\Lambda$}, and there are countably many of these in the universal cover
$\widetilde S_h$.

\medskip
Suppose that $R$ is an ideal polygon with boundary in $\Lambda_+$ (respectively, 
$\Lambda_-$).
Suppose that $\ell$ is segment of a leaf
of $\Lambda_-$ (respectively, $\Lambda_+$) such that the endpoints of $\ell$ lie on adjacent
sides of $R$. 
The adjacent sides meet in an ideal point of $R$ and we call the component of $R \setminus \ell$ containing this ideal point a \emph{cusp} of $R$.  
We say a cusp of $R$ is \emph{maximal}, if it is not contained in any larger cusp.

\medskip
We now describe regions in $\widetilde S_h$ with boundary consisting of arcs alternately contained in $\Lambda_+$ and $\Lambda_-$.
A \emph{polygon} is a compact subset of
$\widetilde S_h$ homeomorphic to a disc, whose boundary consists of an
even number of arcs, alternately contained in $\Lambda_+$ and
$\Lambda_-$. We (partially) organize polygons as follows.
We call a polygon a \emph{rectangle} if its boundary consists
of four arcs.
We call the polygon a \emph{non-rectangular polygon} if it has more
than four sides.  We shall only consider polygons in $\widetilde{S}_h$
whose interiors embed in $S_h$.

\medskip The interior of a polygon may intersect other leaves of the
laminations.  If the interior of the polygon is disjoint from the
leaves of $\Lambda_+$ and $\Lambda_-$ then we
shall call it an \emph{innermost polygon}.  An innermost polygon can be 
either an innermost rectangle, or an innermost non-rectangular
polygon.  

\begin{definition}
Let $(S_h, \LL)$ be a hyperbolic metric on $S$ together with a full
pair of measured laminations.  We say that the laminations are
\emph{suited} if every ideal complementary region $R$ contains a
unique non-rectangular polygon.  In particular, any arc of the other
invariant lamination that lies in $R$ connects adjacent sides of $R$,
and so determines a cusp.
\end{definition}

The invariant laminations of a pseudo-Anosov map are suited; we include a proof below for convenience.

\begin{proposition}\label{prop:innermost polygon}
Suppose that $\pa \colon S \to S$ is a pseudo-Anosov map and
$(S_h, \Lambda)$ is a hyperbolic metric on $S$ together with a pair of
invariant measured laminations. Then the laminations are suited.
\end{proposition}

\begin{proof}
By collapsing the complementary regions of an invariant lamination, say $\Lambda_+$, we obtain a
a measured foliation $F_+$; see 
\cite{kapovich}*{Chapter 11} for details.  
Since the laminations are uniquely ergodic, the resulting foliations are also uniquely ergodic and in particular, contain no saddle connections. Thus, the image of an ideal
polygon with $n$ sides is a singular leaf of $F_+$ with a single $n$-prong
singularity. 
The measured foliation $F_-$ (given by collapsing complementary regions of $\Lambda_-$) also has a
singular leaf with an $n$-prong singularity at the same point. 
The pre-image of this leaf is an ideal polygon complementary to $\Lambda_-$.
As the limit points of the two singular leaves alternate at the
boundary at infinity, the two ideal polygons also have alternating
ideal points at infinity, and so the intersection of the two ideal
polygons is an innermost non-rectangular polygon $P$ with $2n$ sides.

\medskip
In particular, each side of the polygon $P$ which intersects the
interior of $R$ determines a maximal cusp, and all other arcs of the
other lamination which intersect $R$ therefore also lie in
cusps. Therefore these arcs have endpoints in adjacent sides of $R$,
and so determine (non-maximal) cusps in $R$.
\end{proof}

\Cref{fig:non-rectangular polygon} shows two ideal
quadrilaterals $R_+$ and $R_-$ intersecting in an innermost polygon
$P$ with eight sides.  The complement in each ideal quadrilateral of
the innermost polygon $P$ is a maximal cusp.  All other arcs of
$\Lambda_+$ intersecting $R_-$ lie in maximal cusps, and so determine
non-maximal cusps.

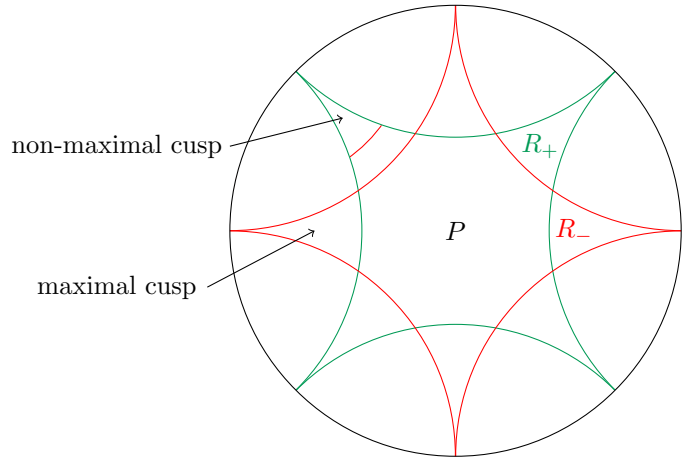
\begin{figure}[h]
\begin{center}
\begin{tikzpicture}[scale=0.75]

\tikzstyle{point}=[circle, draw, fill=black, inner sep=1pt]

\def\boundary{(0, 0) circle (4)}

\begin{scope}
\clip \boundary;
\draw [color=red] (135:4.72) circle (2.5);
\draw [color=white, fill=white] (90:5.66) circle (4);
\draw [color=white, fill=white] (180:5.66) circle (4);
\end{scope}

\begin{scope}
\clip \boundary;
\draw [color=ForestGreen] (0:5.66) circle (4);
\draw [color=ForestGreen] (90:5.66) circle (4);
\draw [color=ForestGreen] (180:5.66) circle (4);
\draw [color=ForestGreen] (270:5.66) circle (4);
\end{scope}
                     
\draw (1.5, 1.5) node [color=ForestGreen] {$R_+$};

\begin{scope}[rotate=45]
\clip \boundary;
\draw [color=red] (0:5.66) circle (4);
\draw [color=red] (90:5.66) circle (4);
\draw [color=red] (180:5.66) circle (4);
\draw [color=red] (270:5.66) circle (4);
\end{scope}

\draw (2.1, 0) node [color=red] {$R_-$};

\draw (0, 0) node {$P$};

\draw (-6, -1) node {maximal cusp};

\draw [arrows=->] (-4.4, -1) -- (-2.5, 0);

\draw (-6, 1.5) node {non-maximal cusp};

\draw [arrows=->] (-4, 1.5) -- (-2, 2);

\draw \boundary;

\end{tikzpicture}
\end{center}
\caption{Two ideal polygons intersecting in a non-rectangular
  polygon.} \label{fig:non-rectangular polygon}
\end{figure}

\begin{remark}\label{rem:full but not suited}
    It is possible to have a full pair of laminations $\LL$ such that 
    \begin{itemize}
        \item the pair is not suited, and yet
        \item the bi-infinite Teichmüller geodesic that they define gives a doubly degenerate hyperbolic 3-manifold with bounded geometry.
    \end{itemize}
    This can happen when one of the measured foliations has a saddle connection, and yet the bi-infinite Teichmüller geodesic lies in a thick part of Teichmüller space. 
\end{remark}
In light of \Cref{rem:full but not suited}, our discussion in this section and its uses in the proofs of the effective theorems does not extend without extra hypothesis to doubly degenerate hyperbolic 3-manifolds with bounded geometry.

\medskip
We observe below that there is an upper bound on the diameter of any
innermost polygon.

\begin{proposition}\label{prop:innermost polygon diameter bound}
Suppose that 
$(S_h, \Lambda)$ is a hyperbolic metric on $S$ together with a suited
pair of measured laminations.  Then there is a constant $D_\LL$ such that diameter of any innermost polygon with boundary in 
$\Lambda_+ \cup \Lambda_-$ is at most $D_\LL$.
\end{proposition}

\begin{proof}
By \Cref{prop:cobounded intersections}, there is an upper bound $L_\LL$
on the length of any side of an innermost polygon.  There are only
finitely many non-rectangular innermost polygons in $S_h$, and the
length of the boundary of a polygon is an upper bound on its diameter,
so we may choose $D_\LL$ to be $n L_\LL$, where $n$ is the maximum
number of sides of any innermost polygon with boundary in
$\Lambda_+ \cup \Lambda_-$.
\end{proof}

\medskip

Suppose that $P$ is an innermost non-rectangular polygon.  The intersection
of the extended leaves of $\Lambda_+$ and $\Lambda_-$ with $P$ gives
us a further finite collection of finite length geodesic segments in $P$ which we call the \emph{extended leaves} in $P$.

\begin{proposition}\label{prop:polygon}
Suppose that $(S_h, \Lambda)$ is a hyperbolic metric on $S$ together
with a suited pair of measured laminations.  Suppose that $P$ is an
innermost non-rectangular complementary polygon.  Then there is a
constant $\theta_P$ such that if the lift of a non-exceptional
geodesic (c.f. \Cref{def:non-exceptional}) $\gamma$ passes within
distance $\theta_P$ in $T^1(S_h)$ of an extended leaf in $P$, then its
distance in $T^1(S_h)$ is at least $\theta_P$ from all of the other
extended leaves in $P$.
\end{proposition}

\begin{proof}
If any two pre-images of geodesics in $S_h$ intersect in $T^1(S_h)$,
they are the same geodesic.  Any innermost non-rectangular
complementary polygon is compact and there are finitely many innermost
non-rectangular polygons $P_i$.  Hence, for each $P_i$ there is a
constant $\theta_{P_i}$ such that the $\theta_{P_i}$-neighborhoods of
all of the segments of the leaves in the polygon are disjoint in
$T^1(S_h)$. By \Cref{prop:innermost polygon diameter bound}, the polygon has bounded diameter, and hence by 
\Cref{prop:fellow travel}, there is a constant $\theta'_{P_i}$ such that
if a point on $\gamma$ is within distance $\theta'_{P_i}$ of one of
the leaves in the polygon, then it is within distance $\theta_{P_i}$
of that leaf for all $t$ for which $\gamma(t)$ is in the polygon, and
so is distance at least $\theta_{P_i}$ from all of the other leaves in
the polygon.  By choosing $\theta_P$ to be the minimum of
$\theta_{P_i}$ we conclude the proof.
\end{proof}

%%%%%%%%%%%%%%%%%%%%%%%%%%%%%%%%%%%%%%%%%%%%%%%%%%%%%%%%%%%%%%%%%%%%%%%
\subsection{Rectangles}\label{sec:outermost}
%%%%%%%%%%%%%%%%%%%%%%%%%%%%%%%%%%%%%%%%%%%%%%%%%%%%%%%%%%%%%%%%%%%%%%%

In this section, we analyze rectangles and define their measures.  Let $R$ be a
rectangle with opposite sides contained in leaves $\ell_1$ and
$\ell_2$ of one of the invariant laminations.  We say $R$ is
\emph{$(\ell_1, \ell_2)$-maximal} if it is not contained in any larger
rectangle with sides in $\ell_1$ and $\ell_2$.  We will show that if
two leaves $\ell_1$ and $\ell_2$ have a common leaf of intersection,
then they bound a unique $(\ell_1, \ell_2)$-maximal rectangle, with
upper and lower bounds on its measure.

\medskip
Note that a rectangle is non-innermost if its interior intersects other leaves of the laminations.

\medskip
Given a side of a rectangle in $\widetilde{S}_h$, we will distinguish between its hyperbolic and its Cannon--Thurston pseudometric length.  We will call the pseudometric length the \emph{measure} of the side, as it is defined in terms of the
measured laminations.  Note that every leaf (of the complementary
lamination) that crosses the interior of a side of a rectangle also
crosses the interior of the opposite side of that rectangle.  It
follows that opposite sides of a rectangle have equal measures, even
though they may have different hyperbolic lengths.  If the interior of
a side does not cross any leaves of the complementary lamination then
it has zero measure.  A rectangle is a \emph{square} if its sides have
equal measures.  We define the \emph{measure} of a rectangle to be the
product of the measures of two adjacent sides.  
By the definition, an innermost rectangle has measure zero.

\medskip
We now fix some notation to refer to the side measures of a
rectangle.  Let $\alpha^+$ be a side of the rectangle in $\Lambda_+$,
and let $\alpha^-$ be a side of the rectangle in $\Lambda_-$.  Define
$dx(R) = \int_{\alpha^- \cap R} \ dx$ and
$dy(R) = \int_{\alpha^+ \cap R} \ dy$.  By the discussion above, these
quantities do not depend on the choice of side.
We define the
measure of $R$ to be $dx(R) dy(R)$.  We say a non-innermost rectangle
$R$ has \emph{positive measure} if $dx(R) dy(R) > 0$, and this will be
the case if and only if its interior intersects leaves of both
invariant laminations.

\medskip
We specify below some notation for rectangles with positive measure, using the conventions in \Cref{fig:opposite quadrants}.

\medskip
Suppose that $R$ is a rectangle with positive measure.  The rectangle
has two sides contained in leaves of $\Lambda_+$, which we shall label
$\alpha^+$ and $\beta^+$.  Similarly, the rectangle has two sides
contained in leaves of $\Lambda_-$, which we shall label $\alpha^-$
and $\beta^-$.

\begin{figure}[h]
\begin{center}
\begin{tikzpicture}[scale=0.75]

\tikzstyle{point}=[circle, draw, fill=black, inner sep=1pt]

\def\boundary{(0, 0) circle (4)}

\def\redone{(270:11.70) circle (11)}
\def\redtwo{(90:11.70) circle (11)}
\def\greenone{(180:11.70) circle (11)}
\def\greentwo{(0:11.70) circle (11)}

\begin{scope}
    \clip \boundary;
    \clip \greenone;
    \draw [color=white, fill=black!20] \redone;
\end{scope}

\draw (-2.25, -2) node {$U$};

\begin{scope}
    \clip \boundary;
    \clip \greentwo;
    \draw [color=white, fill=black!20] \redtwo;
\end{scope}

\draw (1.5, 2) node {$V$};

\begin{scope}
\clip \boundary;
\draw [color=red] \redone;
\draw [color=red] \redtwo;
\draw [color=ForestGreen] \greenone;
\draw [color=ForestGreen] \greentwo;
\end{scope}

\draw (-1.5, -4.5) node [color=ForestGreen] {$\alpha^+$};
\draw (1.5, -4.5) node [color=ForestGreen] {$\beta^+$};
\draw (4.5, -1.5) node [color=red] {$\alpha^-$};
\draw (4.5, 1.5) node [color=red] {$\beta^-$};

\draw (0, 0) node {$R$};

\draw \boundary;

\end{tikzpicture}
\end{center}
\caption{Opposite quadrants.} \label{fig:opposite quadrants}
\end{figure}
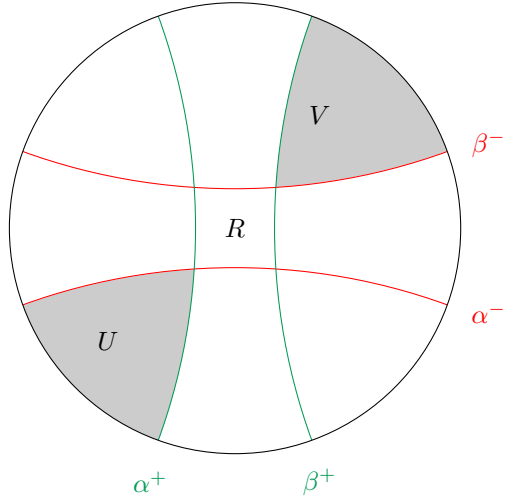

\medskip
Since all sides of $R$ have positive measure, the four leaves
containing the sides of $R$ have distinct endpoints at infinity (no
two are asymptotic).  They divide $\widetilde S_h$ into nine
complementary regions, of which only $R$ is compact.  We call a
(non-compact) complementary region in $\widetilde S_h \setminus R$ a
\emph{quadrant} if it meets exactly one corner of $R$.  We call a pair
of quadrants \emph{opposite}, if they meet opposite corners of $R$.
In \Cref{fig:opposite quadrants}, the regions $U$ and $V$ form a
pair of opposite quadrants.

\medskip

We define the \emph{optimal height} $z = z(R)$ of a positive measure rectangle $R$ to be the value at which the sides
of the rectangle $F_z(R)$ have equal measure, that is $F_z(R)$ is a
square in $F_z( \widetilde{S}_h)$.  So if the sides have measures
$dx(R)$ and $dy(R)$, the optimal height is
$\frac{1}{2} \log_k (dy(R) / dx(R))$, where $k$ is the stretch factor
of the pseudo-Anosov $\pa$, and the measure of each side at the
optimal height is $\sqrt{dx(R)dy(R)}$. In particular, the optimal
height for a square is at $z = 0$.

\medskip
Suppose $\ell_1$ and $\ell_2$ are distinct leaves in $\Lambda_-$. 
We say that a leaf $\ell_+ \in \Lambda_+$ is a \emph{common positive leaf} for $\ell_1$ and
$\ell_2$ if 
\begin{itemize}
    \item $\ell_+$ intersects both $\ell_1$ and $\ell_2$, and
    \item the arc of $\ell_+$ between $\ell_1$ and $\ell_2$ has positive measure.
\end{itemize}

This is illustrated in \Cref{fig:common leaf}.  Similarly, given distinct leaves $\ell_1$
and $\ell_2$ of $\Lambda_+$, we may define a leaf of $\Lambda_-$ to be common positive leaf for $\ell_1$
and $\ell_2$ in the same way. 

\medskip
Since invariant laminations do not contain isolated leaves, by \Cref{prop:angle bound}, the set of common positive leaves is a closed set. 

\medskip
The set of common positive leaves can be empty; for example, if no leaf of the complementary lamination intersects both $\ell_1$ and  $\ell_2$, or if $\ell_1$ and $\ell_2 $ are boundary leaves of an ideal complementary region in which case all complementary arcs between $\ell_1$ and $\ell_2$ have measure zero.

\medskip
We say that two leaves $\ell_1$ and $\ell_2$ of the same lamination
\emph{bound} a rectangle, if the leaves contain opposite sides of
a rectangle with positive measure. In particular, the set of common positive leaves is non-empty. Identifying $\ell_1$ with $\mathbb{R}$, we
conclude that the intersection points with $\ell_1$ of common positive
leaves is a closed bounded set. Since invariant laminations have no
isolated leaves, it follows that this closed bounded set contains no
isolated points. We call the common positive leaves that give the
extrema of this set the \emph{outermost} common positive leaves.  It
also follows that the arc of $\ell_1$ (similarly of $\ell_2$) between
the outermost common positive leaves has positive measure.  The arcs
of the outermost common positive leaves together with the arcs of
$\ell_1$ and $\ell_2$ that they determine, are the sides of an
\emph{$(\ell_1, \ell_2)$-maximal} rectangle, that is, the rectangle is
not contained in a larger rectangle with sides in $\ell_1$ and
$\ell_2$.

\medskip
Suppose that $R$ is an $(\ell_1, \ell_2)$-maximal rectangle. Then
the two sides of $R$ not contained in $\ell_1$ or $\ell_2$, being segments of the outermost common leaves, are contained in boundary leaves of the other lamination. Maximality of $R$ implies that these sides contain sides of a non-rectangular polygon.  The sides of $R$ in $\ell_1$ and
$\ell_2$ need not contain sides of an innermost non-rectangular
polygon, and so $R$ may be contained in a larger rectangle, which at
least one of $\ell_1$ and $\ell_2$ intersect in its interior.

\begin{figure}[h]
\begin{center}
\begin{tikzpicture}%[scale=0.75]

\tikzstyle{point}=[circle, draw, fill=black, inner sep=1pt]

\begin{scope}[rotate=270]

\draw [color=red] (0, 0) -- (0, 4) node [label=right:$\ell_1$] {};
\draw [color=red] (2, 0) -- (2, 4) node [label=right:$\ell_2$] {};

\draw [color=red] (0.5, 0) -- (0.5, 4);
\draw [color=red] (1.5, 0) -- (1.5, 4);

\draw [color=ForestGreen] (-1, 2) -- (3, 2) node [label=below:$\ell_+$] {};

\draw [color=ForestGreen] (-1, 1) -- (3, 1);
\draw [color=ForestGreen] (-1, 3) -- (3, 3);

\draw [very thick, color=ForestGreen] (0, 2) -- (2, 2) node [midway, label=right:$\alpha$]
{};

\draw [very thick, color=red] (0, 1) -- (0, 3) node [midway, label=above right:$I_1$] {};
\draw [very thick, color=red] (2, 1) -- (2, 3) node [midway, label=below right:$I_2$] {};

\end{scope}

\end{tikzpicture}
\end{center}
\caption{Leaves of $\Lambda_-$ intersecting a common leaf of $\Lambda_+$.} \label{fig:common leaf}
\end{figure}
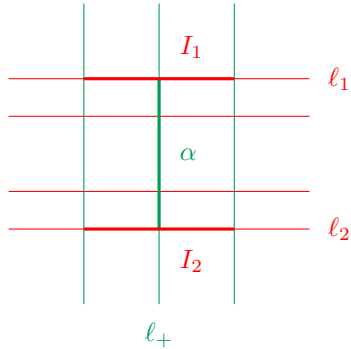

\medskip
We now show that there is a lower bound on the measure of an
$(\ell_1, \ell_2)$ maximal rectangle.

\begin{proposition}\label{prop:min area}
Suppose that $(S_h, \Lambda)$ is a hyperbolic metric on $S$ together
with a suited pair of measured laminations.  Then there is a
constant $A_\LL > 0$, such that any two leaves $\ell_1$ and $\ell_2$
of $\Lambda_+$ that have a common positive leaf in $\Lambda_-$ bound
an $(\ell_1, \ell_2)$-maximal rectangle of measure at least $A_\LL$.
Similarly, any two leaves $\ell_1$ and $\ell_2$ of $\Lambda_-$ that
have a common positive leaf in $\Lambda_+$ bound a maximal rectangle
of measure at least $A_\LL$.
\end{proposition}

\begin{proof}
Suppose that there is a sequence of pairs of leaves
$\ell_n, \ell'_n$ containing opposite sides $\alpha_n$ and $\alpha'_n$
of an $(\ell_n, \ell'_n)$-maximal rectangle $R_n$, so that the measure
of the $(\ell_n, \ell'_n)$-maximal rectangle $R_n$ tends to zero.  Let
$\beta_n$ and $\beta'_n$ be the other sides of the
$(\ell_n, \ell'_n)$-maximal rectangle.  As the rectangle is
$(\ell_n, \ell'_n)$-maximal, each side $\beta_n$ (respectively 
$\beta'_n$) contains a side of an innermost non-rectangular polygon $P_n$ (respectively $P'_n$) in
$\widetilde S_h \setminus (\Lambda_+ \cup \Lambda_-)$, whose interior is disjoint from $R_n$.

\medskip
As there are only finitely non-rectangular innermost polygons in
$S_h$, we may pass to a subsequence where the innermost polygons at
each end do not change. We denote these non-rectangular polygons $P$ and $P'$. By compactness of $S_h$, we may pass to a
subsequence of rectangles which converge to a (possibly degenerate) rectangle $R$
of measure zero.  A measure zero rectangle is either an innermost
rectangle, or a degenerate rectangle given by a subinterval of a leaf, or a point.  
In all cases, at least one of the two leaves $\ell$ and $\ell'$ that arise as limits of $\ell_n$ and $\ell'_n$ respectively, is a boundary leaf of both $P$ and $P'$. Breaking symmetry, suppose $\ell$ is a boundary leaf of both $P$ and $P'$. Orienting $\ell$, suppose that $P$ and $P'$ lie on opposite sides of $\ell$. Then $\ell$ cannot be a limit of other leaves from either side and is hence isolated, a contradiction. Suppose then that $P$ and $P'$ lie on the same side of $\ell$. Then $P$ and $P'$ are contained in a single ideal complementary region of the lamination containing $\ell$, a contradiction to the fact that the pair of laminations is suited. 
\end{proof}

We finally record one more useful fact that there is
an upper bound on the measure of any rectangle.

\begin{proposition}\label{prop:square upper bound}
Suppose that $(S_h, \Lambda)$ is a hyperbolic metric on $S$ together
with a pair of suited measured laminations with bounded geometry.
Then there is a constant $A_{\max}$ such that any rectangle $R$ with
sides in the invariant laminations has measure at most $A_{\max}$.
\end{proposition}

\begin{proof}
The vertical flow is measure preserving for rectangles.  Let $R_z$ be
the image of $R$ under the vertical flow at optimal height $z$, at
which height $R_z$ is a square, in the intrinsic Cannon--Thurston
metric on $\widetilde{S}_h \times \{ z \}$.

\medskip

As the pair of laminations has bounded geometry, the Teichm\"uller
geodesic determined by the pair of laminations is contained in a
compact subset of moduli space.  In particular, there are constants
$Q$ and $c$ such that for all $z$, the intrinsic Cannon--Thurston
metric on $\widetilde{S}_h \times \{ z \}$ is $(Q, c)$-quasi-isometric
to $\ws_h$.

\medskip

Every point in $\widetilde{S}_h$ is a bounded distance from an
innermost non-rectangular region, hence every point in
$\ws_h \times \{ z\}$ is a bounded distance from a non-rectangular
region.  As an innermost non-rectangular region cannot be contained in
a square, there is an upper bound on the diameter of any square, and
hence the measure of the square.
\end{proof}

%%%%%%%%%%%%%%%%%%%%%%%%%%%%%%%%%%%%%%%%%%%%%%%%%%%%%%%%%%%%%%%%%%%%%%%
\subsection{Corner segments and straight segments}\label{section:corners}
%%%%%%%%%%%%%%%%%%%%%%%%%%%%%%%%%%%%%%%%%%%%%%%%%%%%%%%%%%%%%%%%%%%%%%%

In this section we define some notation that will be useful for
describing specific subsegments of geodesics.

\begin{definition}
\label{def:corner segment}
Let $\gamma$ be a geodesic in $(S_h, \Lambda)$.  We say a compact
interval $I$ is a \emph{corner segment} if the segment $\gamma(I)$
\begin{itemize}
    \item is properly embedded in an innermost rectangle, and
    \item the endpoints of $\gamma(I)$ lies in different laminations.
\end{itemize}
We include the degenerate case in which $I$ is
a point and $\gamma(I)$ a vertex of an innermost rectangle.
\end{definition}

Corner segments can occur when a geodesic enters or exits an ideal
complementary region through one of its cusps.

\begin{definition}\label{def:straight}
Let $\gamma$ be a geodesic in $(S_h, \Lambda)$.  We say a compact
interval $I$ is a \emph{straight segment}, if it contains exactly two
corner segments, each one adjacent to an endpoint of $I$.
\end{definition}

We now show that for a suited pair of laminations there are exactly
two types of straight segments.  Recall that if $(S_h, \LL)$ is
suited, then every non-rectangular polygon $P$ is the intersection $P = R_+ \cap R_-$, where $R_+$ and $R_-$ are ideal complementary regions to $\Lambda_+$ and $\Lambda_-$ respectively, and each region of
$R_+ \setminus P$ and $R_- \setminus P$ is a cusp.

\begin{proposition}\label{prop:straight}
Let $(S_h, \LL)$ be a hyperbolic metric on $S$ together with a suited
pair of measured laminations, and let
$\gamma$ be a non-exceptional geodesic in $S_h$.  Then
every corner segment is contained in a straight segment, and
furthermore, every straight segment consists of either the
intersection of $\gamma$ with a single ideal complementary region, or the
intersection of $\gamma$ with the union of two ideal complementary regions
intersecting in a non-rectangular polygon.
\end{proposition}

\begin{proof}
As $\gamma$ is non-exceptional, the intersection of $\gamma$ with any
ideal complementary region $R$ is a compact subinterval.  
Suppose that $\gamma(I) = \gamma \cap R$ has both endpoints in cusps, i.e. neither
endpoint lies in the non-rectangular polygon.
Then both endpoints lie
in rectangles, and hence in corner segments. 
Thus, $\gamma(I)$ is a
straight segment contained in a single ideal complementary region.

\medskip
Now suppose $\gamma(I) = \gamma \cap R$ has both endpoints in the
non-rectangular polygon $P = R \cap R'$ contained in $R$, where $R'$ is an ideal complementary region of the other lamination. It follows that 
\begin{itemize}
    \item $\gamma(I)$ is contained in $\gamma(I') = \gamma \cap R'$, and
    \item both endpoints of $\gamma(I') $ are in cusps of $R'$.
\end{itemize}
Thus $\gamma(I)$ is again a straight segment contained in a single ideal complementary region. 

\medskip
We may now suppose that exactly one endpoint of $\gamma(I) = \gamma \cap R$ is contained in the non-rectangular polygon $P = R \cap R'$. 
Then $\gamma(I') = \gamma \cap R'$ contains the other endpoint of $\gamma \cap P$ and so $\gamma(I \cup I')$ is a straight segment with endpoints in cusps of $R$ and $R'$.

\end{proof}

Finally, we show that the intersection of $\gamma$ and an innermost polygon is contained in a straight segment.

\begin{proposition}\label{prop:straight dense}
Let $(S_h, \LL)$ be a hyperbolic metric on $S$ together with a suited
pair of measured laminations, and let
$\gamma$ be a non-exceptional geodesic in $S_h$.  Then
every intersection segment $\gamma \cap R$ with an innermost polygon $P$
is contained in a straight segment.  In particular, straight segments
are dense in $\gamma$.
\end{proposition}

\begin{proof}
Suppose that $P = R_+ \cap R_-$ is an innermost polygon, where $R_+$ and $R_-$ are ideal complementary regions of $\Lambda_+$ and $\Lambda_-$.  

\medskip
Suppose that $P$ is a rectangle.
If $\gamma \cap P$ is a corner
segment, then we are done by \Cref{prop:straight}.
So suppose that $\gamma$ intersects opposite sides of the rectangle $P$.
Breaking symmetry, assume that the endpoints
of $\gamma \cap P$ are contained in $R_+$.  Each side of $P$
in $R_+$ bounds a cusp in $R_-$, with one of the cusps contained in the other. The geodesic $\gamma$ therefore enters the smaller cusp $C$. Let $P'$ be the last innermost rectangle in $C$ that $\gamma \cap C$ intersects. Then $\gamma \cap P'$ is a
corner segment. 
Therefore $\gamma \cap P$ is contained in a segment
$\gamma \cap R_-$ which terminates in a corner segment at one end, and
so is a subset of a straight segment by \Cref{prop:straight}.

\medskip
Now suppose that $P$ is not a rectangle. Each region of $R_+ \setminus P$ and $R_- \setminus P$ is a cusp. If both endpoints of $\gamma \cap P$
lie in the same lamination, say $R_+$, then $\gamma \cap R_-$ has both
endpoints in cusps of $R_-$. Thus, $\gamma \cap P$ is contained in
the straight segment $\gamma \cap R_-$.  If both endpoints of $\gamma
\cap P$ lie in different laminations, then one endpoint of $\gamma
\cap (R_+ \cup R_-)$ lies in a cusp of $R_+ \setminus P$, and the
other endpoint lies in a cusp of $R_- \setminus P$. Again, 
$\gamma \cap P$ is contained in the straight segment $\gamma \cap (R_+
\cup R_-)$.

\medskip
As innermost regions are dense in $S_h$, and $\gamma$ is
non-exceptional, intersections of innermost regions are dense in
$\gamma$.  As every intersection with an innermost region is contained
in a straight segment, straight segments are dense in $\gamma$, as
required.
\end{proof}

%%%%%%%%%%%%%%%%%%%%%%%%%%%%%%%%%%%%%%%%%%%%%%%%%%%%%%%%%%%%%%%%%%%%%%%
\subsection{Bottlenecks}\label{section:bottlenecks}
%%%%%%%%%%%%%%%%%%%%%%%%%%%%%%%%%%%%%%%%%%%%%%%%%%%%%%%%%%%%%%%%%%%%%%%

The main result of this section is that a non-innermost rectangle in
$(\ws_h, \LL)$ creates a \emph{bottleneck} in $\wsr$, which we now define.

\begin{definition}\label{def:bottleneck}
Let $X$ be a geodesic metric space, and let $U$ and $V$ be subsets of
$X$.  A set $R \subset X$ is an $(r, K)$-bottleneck for $U$ and $V$ if
the distance from $U$ to $V$ is at least $r$, and any geodesic from
$U$ to $V$ passes within distance $K$ of $R$.
\end{definition}

We will show that an optimal height rectangle $F_z(R)$ is an
\emph{$(r, K)$-bottleneck} with respect to either pair of opposite
quadrants.  The constants $r$ and $K$ depend on the measure of the
rectangle (as well as various constants depending on the pseudo-Anosov
map $\pa$), and as the measure of the rectangle tends to zero, $r$ tends
to zero and $K$ tends to infinity.

\begin{lemma}\label{lemma:bottleneck}(Rectangles create
bottlenecks.)  Suppose that $(S_h, \Lambda)$ is a hyperbolic metric on
$S$ together with a full pair of measured laminations.  Let $R$
be a rectangle of measure at least $A > 0$ and optimal height $z$.
Then there are constants $r > 0$ and $K \ge 0$ (that depend on $\pa$
and $A$) such that the optimal height rectangle
$F_z(R) = R \times \{ z \}$ is an $(r, K)$-bottleneck for the flow
sets $F(U)$ and $F(V)$ over any pair of opposite quadrants $U$ and $V$
of $R$.
\end{lemma}

We start by showing that there is a lower bound on the distance
between the suspension flow sets over opposite quadrants of a
transverse rectangle $R$, in terms of the measure of $R$.

\begin{proposition}\label{prop:transverse separate}
Suppose that $(S_h, \Lambda)$ is a hyperbolic metric on $S$ together
with a full pair of measured laminations.  For any $A > 0$ there is
$r > 0$ such that for any rectangle of measure at least $A$, the
suspension flow sets over opposite quadrants are Cannon--Thurston
distance at least $r$ apart in $\widetilde S_h \times \RR$.
\end{proposition}

\begin{proof}
Suppose that $U$ and $V$ are opposite quadrants, as illustrated in
\Cref{fig:opposite quadrants}, so $U$ the region bounded by
$\alpha^+$ and $\alpha^-$, and $V$ the region bounded by $\beta^+$ and
$\beta^-$.  Let the corresponding suspension flow sets over these
regions be $F(U)$ and $F(V)$.  Suppose that the measures of the sides
of the rectangle $R$ are $dx(R) = a> 0$ and $dy(R) = b > 0$.  The
optimal height of $R$ is $z = \tfrac{1}{2}\log_k(b/a)$, and so the
measure of the diagonal (from the corner of $R$ meeting $U$ to the
corner of $R$ meeting at $V$) at this height is $\sqrt{2ab}$.

\medskip

We now give a lower bound for the distance in
$\widetilde S_h \times \RR$ between $F(U)$ and $F(V)$. Let $\gamma$ be
a shortest path in $\widetilde S_h \times \RR$ from $F(U)$ to $F(V)$.
By definition, $\int_\gamma dx \ge a$ and $\int_\gamma dy \ge b$.  Let
$z_+$ and $z_-$ be the largest and smallest height attained along
$\gamma$; these exist because $\gamma$ is compact.  The
Cannon--Thurston pseudo-metric length of the diagonal of $R$ at the
optimal height is an upper bound on the length of $\gamma$, so the
length of $\gamma$ is at most $\sqrt{2ab}$.  We deduce that the
difference in heights between any two points on $\gamma$ is at most
$\sqrt{2ab}$, in particular $z_+ - z_- \le \sqrt{2ab}$.  The
Cannon--Thurston distance in $\widetilde S_h \times \RR$ between
$\alpha^+$ and $\beta^+$ at any height $z \le z_+$ is at least
$a k^{z_-}$.  Similarly, the Cannon--Thurston distance between
$\alpha^-$ and $\beta^-$ at any height $z \ge z_-$ is at least
$b k^{-z_+}$.  Therefore the length of $\gamma$ is at least
$a k^{z_-} + b k^{-z_+}$.  If we set $z_0$ to be the average of $z_+$
and $z_-$, i.e. $z_0 = \tfrac{1}{2}(z_+ + z_-)$, then
\[ z_+ \le z_0 + \tfrac{1}{2} \sqrt{2ab} \quad \text{ and } \quad z_-
\ge z_0 - \tfrac{1}{2} \sqrt{2ab}.  \]
In particular,
\begin{align*}
  a k^{z_-} + b k^{-z_+} & \ge a k^{z_0 - \sqrt{ab/2}} + b k^{-z_0 -
                           \sqrt{ab/2}}, \\
  \intertext{which may be rewritten as}
  a k^{z_-} + b k^{-z_+} & \ge k^{-\sqrt{ab/2}} \left( a k^{z_0} + b
                           k^{-z_0} \right).
\end{align*}

\medskip
The right hand side is minimized when $z_0 = \frac{1}{2} \log_k(b/a)$,
so the length of $\alpha$ is at least
$r = k^{-\sqrt{ab/2}} 2 \sqrt{ab}$, which only depends on the measure $ab \ge A$ of the rectangle $R$.
\end{proof}

\begin{proposition}\label{prop:coarse intersection}
Suppose that $(X, d)$ is a $\delta$-hyperbolic metric space. Suppose that $U$ and
$V$ are convex sets in $X$ that are distance $r \ge 0$ apart.  Then any geodesic from $U$ to
$V$ is contained in $N_{2 \delta + r}(U) \cup N_{2 \delta + r}(V)$ and
intersects $N_{2 \delta + r}(U) \cap N_{2 \delta + r}(V)$.
\end{proposition}

\begin{proof}
Suppose that $\eta = [a, b]$ is a geodesic of length $r + \epsilon$, where $a \in U$ and $b \in V$.
Suppose $u$ is a point of $U$ and $v$ a point of $V$ and $\gamma = [u, v]$ a geodesic from $u$ to $v$. By the thin triangles property, $\gamma$ is contained in a
$2 \delta$-neighborhood of $[u, a] \cup [a, b] \cup [b, v]$, and
hence in a $2 \delta$-neighborhood of
$N_r(U) \cup N_r(V)$.

\medskip
The geodesic $\eta$ itself is contained in $N_r(U) \cap N_r(V)$. 
If $\gamma$ passes within distance $2 \delta$ of $\eta$, then $\gamma$ intersects $N_{2 \delta + r}(U) \cap N_{2 \delta + r}(V)$, as required. Otherwise, $\gamma$ is contained in a $2 \delta$-neighborhood of $[u, a] \cup [b, v]$,
and hence in a $2 \delta$-neighborhood of $U \cup V$.  In particular,
there is a point on $\gamma$ which is distance at most $2 \delta$ from
both $U$ and $V$, so $\gamma$ again intersects
$N_{2 \delta + r}(U) \cap N_{2 \delta + r}(V)$, as required.
\end{proof}

\begin{proposition}\label{prop:intersecting geodesics}
Suppose $\theta> 0$ and $r> 0$ are constants. Then there is a constant
$K> 0$ such that for any two geodesics $\gamma_1$ and $\gamma_2$ in
$\HH^2$ meeting at a point $x$ with an angle at least $\theta$, we
have $N_r(\gamma_1) \cap N_r(\gamma_2) \subseteq N_K(x)$ in the
hyperbolic metric.
\end{proposition}

\begin{proof}
Suppose that $y$ is a point distance at most $r$ from both $\gamma_1$ and $\gamma_2$.  Let $p_1$ and $p_2$ be the respective points on $\gamma_1$ and $\gamma_2$ closest to $y$. Then $x, y, p_1$
and $x, y, p_2$ form two right angled triangles, with angles
$\theta_1$ and $\theta_2$ at $x$ such that
$\theta_1 + \theta_2 = \theta$.  This is illustrated in \Cref{fig:intersecting geodesics}.

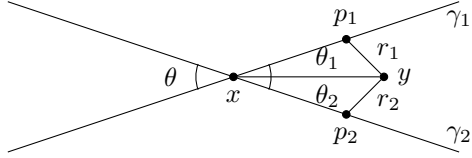
\begin{figure}[h]
\begin{center}
\begin{tikzpicture}%[scale=0.75]

\tikzstyle{point}=[circle, draw, fill=black, inner sep=1pt]

\draw (0, 0) -- (6, 2) node (x) [point, midway, label=below:$x$] {} node
[below] {$\gamma_1$};

\draw (0, 2) -- (6, 0) node [above] {$\gamma_2$};

\draw (5, 1) node (y) [point, label=right:$y$] {};

\draw (x) -- (y);% node [pos=0.8, label=below:$d$] {};

\draw (y) -- (4.5, 1.5) node [pos=0.66, label=right:$r_1$] {} node
[point, label=above:$p_1$] {};

\draw (y) -- (4.5, 0.5) node [pos=0.66, label=right:$r_2$] {} node
[point, label=below:$p_2$] {};

\draw ([shift=(160:0.5cm)] 3, 1) arc (160:200:0.5cm) node [midway,
label=left:$\theta$] {};

\draw ([shift=(20:0.5cm)] 3, 1) arc (20:-20:0.5cm);

\draw (4.25, 1.25) node {$\theta_1$};

\draw (4.25, 0.75) node {$\theta_2$};

\end{tikzpicture}
\end{center}
\caption{Intersecting geodesics.} \label{fig:intersecting geodesics}
\end{figure}

\medskip
In particular, $\min \{ \theta_1, \theta_2\} \ge \tfrac{1}{2} \theta$,
and up to relabeling, we may assume that
$\theta_1 \ge \tfrac{1}{2} \theta$.  Let $d = d(x, y)$. Using the
$\sin$ rule for right angled triangles in $\HH^2$,
\[  \sin \theta_1 = \frac{ \sinh r_1 }{ \sinh d }.  \]
which we may rewrite as
\[  \sinh d = \frac{ \sinh r_1 }{ \sin \theta_1 }.  \]

\medskip We will use the following elementary estimates: for
$x \le \tfrac{\pi}{2}$, $\sin x \ge \tfrac{1}{2} x$ and
$\sinh x \le \tfrac{1}{2} e^x$.  Together with
$\theta_1 \ge \tfrac{1}{2} \theta$ and $r_1 \le r$, we deduce
\[  \sinh d \le \tfrac{1}{\theta}  e^{r}.  \]

We may therefore choose $K = \sinh^{-1}(e^r / \theta)$ to conclude the proof.
\end{proof}

Suppose that $R$ is a rectangle with optimal height $z$.  We now show
that for opposite quadrants $U$ and $V$ for $R$, the intersection of
the metric regular neighborhoods of $F_z(U)$ and $F_z(V)$ with the fiber
$\widetilde S_h \times \{ z\}$ are contained in a bounded neighborhood
of the optimal height square $F_z(R)$.

\begin{proposition}\label{prop:close to rectangle}
Suppose that $(S_h, \Lambda)$ is a hyperbolic metric on $S$ together
with a full pair of measured laminations.  For any positive constants
$A > 0$ and $r> 0$, there is a positive constant $K > 0$ such that for
any opposite quadrants $U$ and $V$ of a rectangle $R$ with measure at
least $A$ and optimal height $z$
\[ N_r( F_{z}(U) ) \cap N_r( F_{z}(V) ) \subset N_K(F_z( R ) ). \]
in the
hyperbolic metric on $\widetilde S_h \times \{ z\}$.
\end{proposition}

\begin{proof}
We may assume that $z = 0$ and that $R$ is a square with the measures
of the sides satisfying $a \ge \sqrt{A}$.  We will choose $K$ to be
the constant from \Cref{prop:intersecting geodesics}, with the given
choice of $r$, and $\theta$ chosen to be $\alpha_\LL$, the minimal
angle of intersection of any two leaves from each invariant
lamination, from \Cref{prop:angle bound}.

\medskip
Let $x$ be a point that is distance at most $r$ from both $U$ and $V$.
If $x$ lies in $R$, there is nothing to prove. So we may assume that
$x$ does not lie in $R$.
We will then show that $x$ is distance at most $r$ from
each pair of intersecting geodesics containing the sides of $R$.
As these geodesics intersect in the corners of $R$, the result follows
from \Cref{prop:intersecting geodesics}.  We now give the details.

\medskip

Suppose $x$ lies in $U$.  We refer to \Cref{fig:opposite quadrants}.
Since any geodesic from $x$ to $V$ crosses both $\alpha^+$ and
$\alpha^-$, the point $x$ lies within distance $r$ of both $\alpha^+$
and $\alpha^-$.  These two geodesics meet at angle at least
$\alpha_\LL$, by \Cref{prop:angle bound}, so by
\Cref{prop:intersecting geodesics}, $x$ is within distance $K$ of
their intersection point, which is one of the corners of $R$.

The same argument works if $x$ lies in $V$, so we may now assume that
$x$ does not lie in either $U$ or $V$, and breaking symmetry, we may
assume that $x$ lies in one of the regions $W_1, W_2$ and $W_3$ in
\Cref{fig:opposite quadrants}.  We consider each case in turn.

\medskip

If $x$ lies in the region $W_1$, then $x$ is distance at most $r$ from
both $\alpha^+$ and $\alpha^-$, which intersect.  Similarly, if $x$
lies in the region $W_2$, then $x$ is distance at most $r$ from both
$\alpha^+$ and $\beta^-$, which intersect.  Finally, if $x$ lies in
the region $W_3$, then $x$ is distance at most $r$ from both $\beta^+$
and $\beta^-$, which intersect.  In all three cases, by
\Cref{prop:angle bound}, the angles of intersection of the leaves is
at least $\alpha_\LL$.  Hence, \Cref{prop:intersecting geodesics}
applies and $x$ is contained in a $K$-neighborhood of the intersection
points of the pairs of geodesics.  As the intersection points are
corners of the rectangle $R$, the result follows.
\end{proof}

By the quasi-isometry between the hyperbolic and Cannon--Thurston
metric, we deduce \Cref{prop:close to rectangle} also in the
Cannon--Thurston metric.  We now show that any geodesic in
$\widetilde S_h \times \RR$ between the suspension flow sets of
opposite quadrants is contained in a bounded neighborhood of the
suspension flow sets, and passes within a bounded neighborhood of the
optimal height rectangle.

\begin{proposition}
\label{prop:quadrant_height}
Suppose that $(S_h, \Lambda)$ is a hyperbolic metric on $S$ together
with a full pair of measured laminations.  For any positive constant
$A > 0$ there is a positive constant $K > 0$ such that if $\gamma$ is
any Cannon--Thurston geodesic in $\widetilde S_h \times \RR$
intersecting both suspension flow sets $F(U)$ and $F(V)$ of opposite
quadrants of a rectangle $R$ with measure at least $A$ and optimal
height $z$, then $\gamma$ is contained in $N_K(F(U) \cup F(V))$, and
intersects $N_K(F_z(R))$.
\end{proposition}

\begin{proof}
By \Cref{prop:transverse separate}, given $A > 0$, there is a constant
$r> 0$ such that the suspension flow sets $F(U)$ and $F(V)$ over
opposite quadrants $U$ and $V$ are Cannon--Thurston distance at least
$r$ apart.  Since $F(U)$ and $F(V)$ are convex, by \Cref{prop:coarse
  intersection} any geodesic from $F(U)$ to $F(V)$ is contained in
$N_{2 \delta + r}(F(U) \cup F(V))$, and passes through
$N_{2 \delta + r}(F(U)) \cap N_{2 \delta + r}(F(V))$.  It therefore
suffices to prove that there is a $K > 0$ such that
$N_{2 \delta + r}(F(U)) \cap N_{2 \delta + r}(F(V))$ is contained in
$N_K(F_z(R))$.

\medskip

We may now assume that the optimal height is $z = 0$, at which height
$R$ is a square of side lengths $a \ge \sqrt{A}$.  Let $(p, z)$ be a
point in $\widetilde S_h \times \RR$ that is Cannon--Thurston distance
at most $2 \delta + r$ from both $F(U)$ and $F(V)$.  A path starting
at $(p, z)$ with length at most $2 \delta + r$, when projected into
$S_h \times \{ z \}$ along suspension flow lines, changes its length
by a factor of at most $k^{2 \delta + r}$.  So in the Cannon--Thurston
metric on $S_h \times \{ z \}$, the point $(p, z)$ is distance at most
$r_1 = (2 \delta + r) k^{2 \delta + r}$ from both $F_z(U)$ and
$F_z(V)$.  In the intrinsic Cannon--Thurston pseudometric on
$S_h \times \{ z \}$, the distance from $U$ to $V$ is at least
$\max \{ a k^z, a k^{-z}\}$.  This implies that
$\max \{ a k^z, a k^{-z}\} \le 2r_1$, and thus
$| z | \le r_2 = \log_k (2 r_1) - \log_k (a)$.  Projecting down to
$z = 0 $ implies that $(p, 0)$ is distance at most $r_3 = r_1 k^{r_2}$
from both $F_{0}(U)$ and $F_{0}(V)$.

\medskip

By \Cref{prop:qi}, the point $(p, 0)$ lies hyperbolic distance at most
$r_4 = Q_\LL r_3 + c_\LL$ from both $F_{0}(U)$ and $F_{0}(V)$.  Let
$K_1$ be the constant from \Cref{prop:close to rectangle}, with $r$
chosen to be $r_4$, and the choice of $A$ from the initial assumption
above.  \Cref{prop:close to rectangle} now implies that as $(p, 0)$
lies in $N_{r_4}(F_0(U)) \cap N_{r_4}(F_0(V))$, the point $(p, 0)$
lies in $N_{K_1}(F_0(R))$.  As $|z| \le r_2$, the point $(p, z)$ lies
in a $(K_1 + r_2)$-neighborhood of $F_0(R)$.  So we may choose
$K = \max \{ 2\delta + r, K_1 + r_2 \}$, which only depends on
$A, \pa$ and $S_h$, as required.
\end{proof}

%%%%%%%%%%%%%%%%%%%%%%%%%%%%%%%%%%%%%%%%%%%%%%%%%%%%%%%%%%%%%%%%%%%%%%%%%%%%%%%%%%
\section{Quasigeodesics in the Cannon--Thurston metric}\label{section:quasigeodesics}
%%%%%%%%%%%%%%%%%%%%%%%%%%%%%%%%%%%%%%%%%%%%%%%%%%%%%%%%%%%%%%%%%%%%%%%%%%%%%%%%%%

In this section, from a non-exceptional geodesic $\gamma$ in
$\widetilde S_h$, we construct an explicit quasigeodesic in
$\widetilde S_h \times \RR$ whose limit points are the images of the
limit points of $\gamma$ by the Cannon--Thurston map.

\medskip
Recall that, by definition, the Cannon--Thurston images of the limit
points of a non-exceptional geodesic $\gamma$ in $\widetilde S_h$ are
distinct and hence determine a geodesic $\overline{\gamma}$ in
$\widetilde S_h \times \RR$, which we shall call the \emph{target
  geodesic}.  As the ladder $F(\gamma)$ of $\gamma$ is quasiconvex,
the target geodesic $\overline{\gamma}$ is contained in a bounded
neighborhood of $F(\gamma)$.  In particular, the nearest point
projection of $\overline{\gamma}$ to $F(\gamma)$ is quasigeodesic.  In
particular, there is a function $h_\gamma(t)$ such that the path
$(\gamma(t), h_\gamma(t))$ is quasigeodesic.

\medskip In this section we shall define the function $h_\gamma$ in
terms of the height function
$h \colon T^1(S_h) \setminus \oLambda^1 \to \RR$ from \Cref{def:height
  function2}, and we will give a specific choice of the constant used
in the definition of this function in \Cref{section:height function}.
We will set $h \colon T^1(S_h) \setminus \oLambda^1 \to \RR$, and then
define $h_\gamma(t) = h(\gamma^1(t))$.  We will call the path
$\tau_\gamma(t) = ( \gamma(t), h_\gamma(t) )$ in $F(\gamma)$ the
\emph{test path} associated to $\gamma$.  We call the function
$h_\gamma(t)$ the \emph{height function} for $\gamma$.  We will show
that the test path is an (unparametrized) quasigeodesic in
$\widetilde S_h \times \RR$ connecting the limit points of
$\iota(\gamma)$, and we emphasize that the unit speed parametrization
of $\gamma$ in $\HH^2$ does not give a quasigeodesic parametrization
of the test path.

\medskip

We now specify the height function we will use.  We shall write
$\log_k$ for the logarithm function with base $k$, where $k = k_f > 1$
is the parameter from the definition of the Cannon--Thurston metric,
which is the stretch factor of $\pa$ in the case that the laminations
are the invariant laminations of a pseudo-Anosov map $\pa$.
\begin{definition}\label{def:height function}
Let $(S_h, \Lambda)$ be a hyperbolic metric, together with a regular
pair of measured laminations.  Let $\oLambda^1_-$ and
$\overline{\Lambda}^1_+$ be the lifts in $T^1(S_h)$ of the extended
laminations determined by $\Lambda$, and let $\theta_\LL > 0$ be a
positive constant.
We define the \emph{height function}
$h_{\theta_\LL} \colon T^1(S_h) \setminus (\overline{\Lambda}^1_+ \cup
\overline{\Lambda}^1_-) \to \RR$ to be
\[ h_\theta(v) = \log_k \left\lfloor \log \frac{1}{d_{\PSL(2, \RR)}(v,
  \overline{\Lambda}^1_+ )} - \log \frac{1}{\theta_\LL} \right\rfloor_1
- \log_k \left\lfloor \log \frac{1}{d_{\PSL(2, \RR)}(v,
  \overline{\Lambda}^1_- )} - \log \frac{1}{\theta_\LL}
\right\rfloor_1 \]
\end{definition}

Here $\lfloor x \rfloor_c = \max \{ x, c \}$ is the standard floor
function.  As the two extended laminations are a positive distance
apart in $T^1(S_h)$, for sufficiently small $\theta_\LL$, at most one
of the terms on the right hand side above will be non-zero.

\medskip
We prove below that for a sufficiently small choice of $\theta_\LL$, the test path determined
by the corresponding height function is quasigeodesic.

\begin{theorem}\label{theorem:quasigeodesic-h}
Suppose that $(S_h, \Lambda)$ is a hyperbolic metric on $S$ together
with a suited pair of measured laminations.  Then there are constants
$\theta_\LL > 0$, $Q \ge 1$ and $c \ge 0$, such that for any
non-exceptional geodesic $\gamma$ in $S_h$, with a unit speed
parametrization $\gamma(t)$, the test path
$\tau_\gamma(t) = (\gamma(t), h_{\theta_\LL}(\gamma^1(t)) )$ is an
unparametrized $(Q, c)$-quasigeodesic in $\widetilde S_h \times \RR$
with the same limit points as $\iota(\gamma)$, where $h_{\theta_\LL}$ is the
height function from \Cref{def:height function2}.
\end{theorem}

As we shall fix a sufficiently small constant $\theta_\LL$, depending
only on $(S_h, \Lambda)$, we shall just write $h$ for
$h_{\theta_\LL}$.  See \Cref{section:height function} for the exact
choice of $\theta_\LL$ that we use.  Furthermore, we will write
$h_\gamma(t)$ for $h_{\theta_\LL}(\gamma^1(t))$.

\medskip
Here is the basic intuition behind \Cref{theorem:quasigeodesic-h}.
We partition the geodesic $\gamma(t)$ into arcs where over each arc either the geodesic is far from both laminations or it comes close to one of the laminations.
If it comes $r$-close to a lamination, say $\Lambda_+$ then there is a fellow travel of length
$\log \frac{1}{r}$ with a leaf $\ell$ of $\Lambda_+$.
This means that in $\widetilde{S}_h \times \RR$ there is a  
short cut through the ladder (hyperbolic plane) of suspension flow lines through $\ell$ which has maximum height
$\log_k ( \log \frac{1}{r})$.

\medskip
This intuition works well when the geodesic has a large intersection
with a complementary region that is an ideal triangle.  There are two
main technical issues to address.

\medskip
Firstly, there may be complementary ideal regions with more than three
sides.  Suppose that a geodesic crosses a non-triangular ideal complementary
region entering and leaving the region through non-adjacent cusps and with a very small angle.  As the geodesic traverses through the middle of the
complementary region, it is far from all boundary leaves, and yet the height of the test path in
$\widetilde S_h \times \RR$ should be large.  In this portion the geodesic is close to an extended leaf and so we overcome the difficulty 
by modifying the distance functions to consider distances to the extended laminations.

\medskip

Secondly, there can be long geodesic segments that are close to leaves
of one of the laminations, but do not spend very long in any single
complementary region.  In this case, we show that the geodesic
$\gamma$ passes though rectangles formed by the stable and unstable
laminations, whose area is bounded below.  Furthermore, the endpoints
of the geodesic lie in opposite quadrants of the complementary regions
formed by the leaves of the laminations containing the sides of the
rectangles.  The image of each rectangle in $\widetilde M$, at optimal
height under the vertical flow, is a bottleneck for the union of flow
lines through opposite quadrants, and so the target geodesic in
$\widetilde{M}$ passes within a bounded distance of the optimal height
rectangle.  By definition of the height function, the test path also
passes within a bounded distance of the optimal height rectangle, and
hence is close to the target geodesic.  The remaining technical step
to check is that the segments of the test path passing through the
bottlenecks have bounded length, and we also do this in
\Cref{section:tame}.

%%%%%%%%%%%%%%%%%%%%%%%%%%%%%%%%%%%%%%%%%%%%%%%%%%%%%%%%%%%%%%%%%%%%%%%
\subsection{The test path is quasigeodesic}\label{section:qg}
%%%%%%%%%%%%%%%%%%%%%%%%%%%%%%%%%%%%%%%%%%%%%%%%%%%%%%%%%%%%%%%%%%%%%%%

We will use the following criteria due to Farb \cite{farb} (see also
Hoffoss \cite{hoffoss}*{page 216}) to ensure that a path is a
quasigeodesic.

\begin{definition}\label{def:uniform progress}
Suppose that $(X, d)$ is a metric space.
Suppose that $\tau$ is a path in $X$ with unit
speed parametrization and suppose that $\gamma$ is a geodesic in $X$.
We say that $\tau$ \emph{makes $(L, M)$-uniform progress along $\gamma$} if
there is a constant $L$ such that for any two points distance at least
$L$ apart along $\tau$, the distance between their nearest point
projections to $\gamma$ is at least $M$.
\end{definition}

\begin{theorem}\cite{farb}\cite{hoffoss}*{page
  216}\label{theorem:uniform progress h} 
Suppose that $\tau$ is parametrized unit speed path in $\HH^3$ such that 
\begin{itemize}
    \item $\tau$ is contained in a bounded neighborhood of a geodesic $\gamma$, and
    \item there is a constant $L > 0$ such that $\tau$ makes $(L, 1)$-uniform progress along $\gamma$.
\end{itemize}
Then $\tau$ is a quasigeodesic.
\end{theorem}

As the Cannon--Thurston pseudo-metric is quasi-isometric to the hyperbolic metric on $\HH^3$, it suffices to show the following two results.

\begin{lemma}\label{lemma:bounded distance}
Suppose that $(S_h, \Lambda)$ is a hyperbolic metric on $S$ together
with a suited pair of measured laminations.  Then there is a
constant $K$ such that for any non-exceptional geodesic $\gamma$ in
$\HH^2$, the corresponding test path $\tau_\gamma$ is contained in a
$K$-neighborhood of $\overline{\gamma}$ in
$\widetilde{S}_h \times \RR$.
\end{lemma}

Let $Q_M$ and $c_M$ be the quasi-isometry constants between
$\widetilde{S}_h \times \RR$ and $\HH^3$.  Then $(L, 1)$-uniform
progress in $\HH^3$ will be implied by $(L', M)$-uniform progress in
$\widetilde{S}_h \times \RR$, where $M = Q_M + c_M$.  In fact,
assuming \Cref{lemma:bounded distance}, any point on
$\tau_\gamma$ is distance at most $K$ from $\overline{\gamma}$.
Therefore, it suffices to show that there is a constant $L$, such that
any pair of points distance at least $L$ apart along $\tau_\gamma$ are
distance at least $M + 2K$ apart in $\widetilde{S}_h \times \RR$.

\begin{lemma}\label{lemma:progress}
Suppose that $(S_h, \Lambda)$ is a hyperbolic metric on $S$ together
with a suited pair of measured laminations.  Then for any constant
$M$ there is a constant $L$ such that for any non-exceptional geodesic
$\gamma$, and any two points $p$ and $q$ distance at least $L$ apart
along the test path $\tau_\gamma$, the distance in
$\widetilde{S}_h \times \RR$ between $p$ and $q$ is at least $M$.
\end{lemma}

There are two key tools which we will use to show that these
conditions hold.

\begin{enumerate}

\item[(1)] \emph{Tame bottlenecks:} Test path segments over corner segments are tame bottlenecks.

\medskip
Recall that an interval $I$ is a \emph{corner segment} if $\gamma(I)$
is embedded in an innermost rectangle and the endpoints of $\gamma(I)$ lie in different laminations. We prove that the test path
segment $\tau_\gamma(I)$ over $\gamma(I)$ is a bottleneck for the flow
sets over the two complementary components of
$\gamma(\RR \setminus I)$.  We also show that the bottleneck is \emph{tame}, that is, $\tau_\gamma(I)$ has bounded length.

\item[(2)] \emph{Straight intervals:} Test path segments over straight intervals are quasigeodesic.

\medskip
Recall that an interval $I$ is \emph{straight} if $\gamma(I)$ contains
exactly two corner segments, each one adjacent to an endpoint of $I$. We prove that the test paths segments over straight intervals are quasigeodesic. There are exactly two cases: either $\gamma(I)$ is the
intersection of $\gamma$ with a single ideal complementary
region, or the intersection with two complementary regions intersecting in a
non-rectangular polygon.

\end{enumerate}

We now state precise versions of the results described in the two
items above, and two additional properties of the construction
we will use.  In the next two sections we use these results to deduce that the test path is quasigeodesic, by showing how Farb's criterion, namely \Cref{theorem:uniform progress h} applies. 

\medskip
We first state the tame bottlenecks result.  We shall prove it in
\Cref{section:tame}.

\begin{lemma}\label{cor:corner distance}
Suppose that $(S_h, \Lambda)$ is a hyperbolic metric on $S$ together
with a suited pair of measured laminations.  Then there are constants
$r > 0$ and $K \ge 0$, such that for any corner segment
$I = [t_1, t_2]$ of any non-exceptional geodesic $\gamma$, there are
parameters $u \le t_1 \le t_2 \le v$ such that the set
$\tau_\gamma([u, v])$ is an $(r, K)$-bottleneck for the ladders over
$\gamma((-\infty, u])$ and $\gamma([v, \infty))$.  Furthermore, the
length of $\tau_\gamma([u, v])$ is at most $K$.
\end{lemma}

We now state the result for test paths over straight intervals. We shall prove this in \Cref{section:straight
  qg}.

\begin{lemma}\label{lemma:straight qg}
Suppose that $(S_h, \Lambda)$ is a hyperbolic metric on $S$ together
with a suited pair of measured laminations.  Then there are
constants $Q$ and $c$ such that for any non-exceptional geodesic
$\gamma$, and any straight interval $I$, the test path
$\tau_\gamma(I)$ is $(Q, c)$-quasigeodesic.
\end{lemma}

%%%%%%%%%%%%%%%%%%%%%%%%%%%%%%%%%%%%%%%%%%%%%%%%%%%%%%%%%%%%%%%%%%%%%%%
\subsubsection{The test path is close to a geodesic}\label{section:close}
%%%%%%%%%%%%%%%%%%%%%%%%%%%%%%%%%%%%%%%%%%%%%%%%%%%%%%%%%%%%%%%%%%%%%%%

In this section, we use \Cref{cor:corner distance} and
\Cref{lemma:straight qg} to show that the test path $\tau_\gamma$ lies in a
bounded neighborhood of the geodesic $\overline{\gamma}$, verifying the first condition in \Cref{theorem:uniform progress h}.

\begin{lemma:bounded distance}%\label{lemma:bounded distance}
Suppose that $(S_h, \Lambda)$ is a hyperbolic metric on $S$ together
with a suited pair of measured laminations.  Then there is a
constant $K> 0$ such that for any non-exceptional geodesic $\gamma$ in
$S_h$, and for any geodesic $\overline{\gamma}$ in
$\widetilde S_h \times \RR$ connecting the limit points of
$\iota(\gamma)$, the corresponding test path $\tau_\gamma$ is
contained in a $K$-neighborhood of $\overline{\gamma}$, i.e.
\[ \tau_\gamma \subseteq N_K( \overline{\gamma} ).  \]
\end{lemma:bounded distance}

This result is obtained as follows. By \Cref{cor:corner distance}, 
the
test path over any corner segment of $\gamma$ is a bounded distance from the
geodesic $\overline{\gamma}$. By \Cref{prop:straight}, every straight segment  has
both endpoints a bounded distance from the geodesic $\ogamma$.
Stability of quasigeodesics then implies that the entire straight
segment is contained in a bounded neighborhood of $\ogamma$.

\medskip
We now give the details for the proof of \Cref{lemma:bounded
  distance}.

\begin{proof}[Proof (of \Cref{lemma:bounded distance})]
As $\gamma$ is non-exceptional, its intersection with
each ideal complementary region has finite diameter.  The ideal
complementary regions are dense in $\ws_h$, and hence such intersection segments are dense in $\gamma$. By
\Cref{prop:straight}, each such intersection segment of $\gamma$ is contained in a straight segment.

\medskip
By \Cref{cor:corner distance}, there is a constant $K_1$ such that if
$\gamma(t)$ is in a corner segment, then $\tau_\gamma(t)$ is
distance at most $K_1$ from $\overline{\gamma}$.  By
\Cref{lemma:straight qg}, there are constants $Q$ and $c$ such that the
test path over any straight segment is $(Q, c)$-quasigeodesic.  As
each endpoint of a straight segment is within distance $K_1$ of
$\ogamma$, we may extend the straight segment at each end by paths of
length at most $K_1$, so that the endpoints of the extended straight
segment lie on $\ogamma$, and furthermore, this new path is
$(Q, c + K_1)$-quasigeodesic.

\medskip
By stability for quasigeodesics, there is a constant $L$ such that any
$(Q, c + K_1)$-quasigeodesic is contained in an $L$-neighborhood of
any geodesic connecting its endpoints.  In particular, every straight
segment is contained in an $L$-neighborhood of $\ogamma$.  As straight
segments are dense in $\gamma$, the result follows.
\end{proof}

%%%%%%%%%%%%%%%%%%%%%%%%%%%%%%%%%%%%%%%%%%%%%%%%%%%%%%%%%%%%%%%%%%%%%%%
\subsubsection{The test path makes uniform progress}
%%%%%%%%%%%%%%%%%%%%%%%%%%%%%%%%%%%%%%%%%%%%%%%%%%%%%%%%%%%%%%%%%%%%%%%

We now prove that the test path makes uniform
progress, verifying the second condition in \Cref{theorem:uniform progress h}.

\medskip
As the Cannon--Thurston pseudo-metric on $\widetilde{S}_h \times \RR$
is quasi-isometric to the hyperbolic metric, it suffices to show that $\tau_\gamma$ makes
$(L, M)$-uniform progress along $\ogamma$, for any $L > 0$ and for
$M = Q_M+c_M$, where $Q_M$ and $c_M$ are the quasi-isometry constants
between $\widetilde{S}_h \times \RR$ and $\HH^3$.  In fact, we will show:

\begin{proposition}
Suppose that $(S_h, \Lambda)$ is a hyperbolic metric on $S$ together
with a suited pair of measured laminations.  Then for any constant
$M \ge 0$ there is a constant $L \ge 0$ such that for any
non-exceptional geodesic $\gamma$, the test path $\tau_\gamma$ makes
$(L, M)$-uniform progress along $\overline{\gamma}$.
\end{proposition}

We now give a brief overview of the argument.

\begin{enumerate}

\item Suppose that the point $\gamma_\tau(t_1)$ is
an $(r, K)$-bottleneck for $U_1 = F(\gamma((-\infty, a_1))$ and
$V_1 = F(\gamma([b_1, \infty)))$, and similarly the point $\gamma_\tau(t_2)$
is an $(r, K)$-bottleneck for $U_2 = F(\gamma((-\infty, a_2))$
and $V_1 = F(\gamma([b_2, \infty)))$.  Then both points $\tau_\gamma(t_1)$ and
$\tau_\gamma(t_2)$ are bottlenecks for $U = U_1 \cap U_2$ and
$V = V_1 \cap V_2$, and furthermore, the distance between $U$ and $V$
is at least $d_{\wsr}( \tau_\gamma(t_1), \tau_\gamma(t_2) ) - 2K$.
This gives lower bounds on the distances between points on
the test path.  It remains to show that for any two points on the test
path sufficiently far apart along
$\tau_\gamma$, there are a pair of bottlenecks a definite distance
apart in $\wsr$.

\item Let $\tau_\gamma(t_1)$ and $\tau_\gamma(t_2)$ be two
points on the test path sufficiently far apart.  Suppose the
corresponding segment $\gamma([t_1, t_2])$ contains a long straight segment $\gamma(I)$.  Then the test path
$\tau_\gamma(I)$ is quasigeodesic, and so it has definite length. Furthermore, 
the corner segments at each end of $\gamma(I)$ give a pair of
bottlenecks.  Therefore, $\tau_\gamma(t_1)$ and $\tau_\gamma(t_2)$ are a definite distance apart in $\wsr$.

\item Now suppose that the corresponding segment $\gamma([t_1, t_2])$ of
the geodesic in $\ws_h$ does not contain any long straight segments.
This implies that there is a constant $L$ such that every segment
of the test path over $\gamma([t_1, t_2])$ with length $L$ contains a corner segment. 
By \Cref{cor:corner distance}, it has a tame bottleneck.  The bottleneck sets for successive tame bottlenecks are nested and so the distance in $\wsr$
increases linearly in the number of bottlenecks. Thus the test path
makes definite progress, as required.

\end{enumerate}

We start with Step 1 by showing that pairs of bottlenecks separate
points along the test path $\tau_\gamma$.

\begin{proposition}\label{prop:nesting distance}
Suppose that $(S_h, \Lambda)$ is a hyperbolic metric on $S$ together
with a suited pair of measured laminations.  Let $\gamma_\tau(t_1)$
be an $(r, K)$-bottleneck for $U_1 = F(\gamma((-\infty, a_1])$ and
$V_1 = F(\gamma([b_1, \infty)))$, and similarly let $\gamma_\tau(t_2)$
be an $(r, K)$-bottleneck for $U_2 = F(\gamma((-\infty, a_2])$ and
$V_1 = F(\gamma([b_2, \infty)))$.

\medskip
Let $a = \min \{ a_1, a_2\}$ and let $b = \max \{ b_1, b_2 \}$.  Then
for any two points $p$ and $q$ separated by $[a, b]$, the distance in
$\widetilde{S}_h \times \RR$ between $\tau_\gamma(p)$ and
$\tau_\gamma(q)$ is at least the distance between $\tau_\gamma(t_1)$
and $\tau_\gamma(t_2)$, up to bounded error $2K$, i.e.
\[ d_{\widetilde{S}_n \times \RR}( \tau_\gamma(p), \tau_\gamma(q)
) \ge d_{\widetilde{S}_n \times \RR}( \tau_\gamma(t_1),
\tau_\gamma(t_2) ) - 2K.\]
\end{proposition}

\begin{proof}
We may assume that $p \le a \le b \le q$.  Let $\eta$ be a geodesic in
$\widetilde{S}_h \times \RR$ connecting $\tau_\gamma(p)$ and
$\tau_\gamma(q)$.

\medskip
By \Cref{cor:corner distance}, there are constants $r$ and $K$ such that the
points $\tau_\gamma(t_1)$ and $\tau_\gamma(t_2)$ are $(r, K)$-bottlenecks
for
$U = F(\gamma((-\infty, a_1]) \cap F(\gamma((-\infty, a_2]) =
F(\gamma((\infty, a]))$ and
$V = F(\gamma([b, \infty)) = F(\gamma([b_1, \infty)) \cap
F(\gamma([b_2, \infty)))$.

\medskip
As $\tau_\gamma(p) \in U$ and $\tau_\gamma(q) \in V$, the geodesic
$\eta$ passes within distance $K$ of both $\tau_\gamma(t_1)$ and
$\tau_\gamma(t_2)$, so the length of $\eta$ is at least
$d_{\widetilde{S}_n \times \RR}( \tau_\gamma(t_1), \tau_\gamma(t_2) )
- 2 K$, as required.
\end{proof}

By \Cref{lemma:bounded distance}, the test path $\tau_\gamma$ is contained in a $K$-neighborhood of
$\overline{\gamma}$. Hence, for an point $p$ on the test path, its nearest
point projection to $\overline{\gamma}$ is distance at most $K$ away.
Therefore, the distance between the nearest points on $\ogamma$ of any
points $p$ and $q$ on the test path is at least
$d_{\widetilde{S}_h \times \RR}(p, q) - 2K$ apart, so it suffices to
estimate distance between points on the test path in
$\widetilde{S}_h \times \RR$.

\medskip

For Step 2, suppose that a long segment 
$\tau_\gamma(I)$ contains a long straight subsegment.

\begin{proposition}
\label{prop:long straight}
Suppose that $(S_h, \Lambda)$ is a hyperbolic metric on $S$ together
with a suited pair of measured laminations.  Then for any constant $M$
there is a constant $L$ such that for any straight segment $I$ and any interval $[p,q]$ such that $\tau_\gamma([p, q]) \cap \tau_\gamma(I)$ has
arc length at least $L$, the distance in $\wsr$ between $\tau_\gamma(p)$ and $\tau_\gamma(p)$ is at least $M$.
\end{proposition}

\begin{proof}
Let $I = [t_1, t_2]$.  By \Cref{lemma:straight qg}, there are
constants $Q$ and $c$ such that the test path $\tau_\gamma(I)$ over the straight segment is $(Q, c)$-quasigeodesic.  By \Cref{def:straight}, $\gamma(I)$ has corner segments at both ends. By
\Cref{cor:corner distance}, there are constants $r$ and $K$, and points
$u_i \le t_i \le v_i $ such that the endpoints $\tau_\gamma(t_i)$ are
$(r, K)$-bottlenecks with respect to $U_i = F(\gamma((-\infty, u_i]))$ and
$V_i = F(\gamma([v_i), \infty))$.  Furthermore, the lengths of the
segments $\tau_\gamma([u_i, t_i])$ and $\tau_\gamma([t_i, v_i])$ are
at most $K$.  In particular, the segment $\tau_\gamma([u_1, v_2])$,
lying between the initial quadrant for $t_1$ and the terminal quadrant
for $t_2$, is $(Q, c + 2 K)$-quasigeodesic.

\medskip
Suppose that $[p, q] \subseteq [u_1, v_2]$.  
Since $\tau_\gamma([u_1, v_2])$ is $(Q, c + 2K)$-quasigeodesic and the arc length of $\tau_\gamma([p, q]) \cap \tau_\gamma(I)$ is at least $L$, the distance
in $\wsr$ between $\tau_\gamma(p)$ and $\tau_\gamma(q)$ is at least
$L / Q - c - 2K$.

\medskip
Now suppose that $[p, q]$ is not contained in $[u_1, v_2]$.
Then at least one
endpoint of $[p, q]$ lies outside $[u_1, v_2]$.  Up to reversing the
orientation on $\gamma$, we may assume that $p < u_1$, and hence $\tau_\gamma([p, q]) \cap \tau_\gamma(I) = \tau_\gamma([t_1, q])$. 
By assumption, the
arc length of $\tau_\gamma([t_1, q])$ is at least $L$.
So the distance in
$\wsr$ between $\tau_\gamma(t_1)$ and $\tau_\gamma(q)$ is at least
$L / Q - c - 2 K$.  The point $\tau_\gamma(p)$ lies in $U_1$, and so by the tame bottleneck property any
geodesic from $\tau_\gamma(p)$ to $\tau_\gamma(q)$ passes within
distance $K$ of $\tau_\gamma(t_1)$.
So the distance from $\tau_\gamma(p)$ to $\tau_\gamma(q)$ is at least
$L / Q - c - 3 K$.

\medskip
The result follows if given $M$ we choose $L = Q ( M + c + 3 K)$,
where $Q, c$ and $K$ only depend on $(S_h \Lambda)$.
\end{proof}

We now assume that a segment $\tau_\gamma(I)$ contains no straight segment
of arc length greater than $L$. By \Cref{def:straight}, it follows that every segment $\tau_\gamma(I)$ of arc length $L$
contains a corner segment.  We now show that in this case the test
path makes definite progress.

\begin{proposition}
Suppose that $(S_h, \Lambda)$ is a hyperbolic metric on $S$ together
with a suited pair of measured laminations.  For constants $L$
and $M$ as in \Cref{prop:long straight} there is a constant $N$ such that for any segment $I$ of
length $NL$ in which every segment of $\tau_\gamma(I)$ of arc length
$L$ contains a corner segment, the distance in $\wsr$ between the
endpoints of $\tau_\gamma(I)$ is at least $M$.
\end{proposition}

\begin{proof}
Let $I_1 , \ldots , I_n$ be consecutive intervals along $I$ such
that each test path segment $\tau_\gamma(I_j)$ has arc
length $L$, and each segment $\gamma(I_j)$ contains a corner segment.
By \cref{lemma:bottleneck}, each corner segment is an
$(r, K)$-bottleneck, with respect to a pair of sets
$U_j = F( (-\infty, u_j] )$ and $V_j = F( [v_j, \infty) )$, with the
arc length of $\tau_\gamma([u, v])$ at most $K$.  We may assume that
$L$ is greater than $K$, and so for any pair of intervals $I_j$ and
$I_{j+2}$, the segments $\tau_\gamma(u_j, v_j)$ and
$\tau_\gamma(u_{j+2}, v_{j+2})$ are disjoint, and the bottleneck sets
are nested, $U_j \subset U_{j+2}$ and $V_{j+2} \subset V_j$.  In
particular, the distance between $U_j$ and $V_{j+2}$ is at least $2r$.
The result now follows by choosing $N \ge M/r + 2$.
\end{proof}

We conclude the proof that test paths are quasigeodesics, assuming
\Cref{cor:corner distance} and \Cref{lemma:straight qg}.  The
remainder of this section is devoted to the proof of these two
results.

%%%%%%%%%%%%%%%%%%%%%%%%%%%%%%%%%%%%%%%%%%%%%%%%%%%%%%%%%%%%%%%%%%%%%%%%%%%%%%%%%%
\subsection{Properties of the height function}\label{section:height}
%%%%%%%%%%%%%%%%%%%%%%%%%%%%%%%%%%%%%%%%%%%%%%%%%%%%%%%%%%%%%%%%%%%%%%%%%%%%%%%%%%

In this section, we specify the exact choice $\theta_\LL$ for the height function in \Cref{def:height function2}.  To do so, we first define a function that is 
intermediate between the height function and the distance in
$\PSL(2, \RR)$, which we call the radius function.  We discuss its
geometric interpretation, and use it to show that both the radius and the corresponding height functions are Lipschitz.  We also show that both the radius and height functions at a point of intersection of a geodesic $\gamma$ with a leaf of an invariant lamination can be
estimated in terms of the angle of intersection.

%%%%%%%%%%%%%%%%%%%%%%%%%%%%%%%%%%%%%%%%%%%%%%%%%%%%%%%%%%%%%%%%%%%%%%%%%%%%%%%%%%
\subsubsection{The radius function}
%%%%%%%%%%%%%%%%%%%%%%%%%%%%%%%%%%%%%%%%%%%%%%%%%%%%%%%%%%%%%%%%%%%%%%%%%%%%%%%%%%

In \Cref{def:height function2}, the height function is defined in terms of the distance in $\PSL(2, \RR)$
from the geodesic to the extended laminations.  Intuitively, we can
think of distances in $\PSL(2, \RR)$ as extending the angle of
intersection between the geodesic and the laminations to a continuous
function along the geodesic.  We define an intermediate function, which we call the radius function, to be roughly the exponential of
the height function, or equivalently, the logarithm of the distance
function in $\PSL(2, \RR)$.  Intuitively, this extends the length
of the projection interval to a continuous
function along $\gamma$.  We give below the precise definition, and then
use it to estimate the radius function at an intersection point
between a geodesic $\gamma$ and a lamination in terms of the angle of
intersection.

\medskip
The angle $\theta$ of intersection between $\gamma$ and a leaf
$\ell \in \Lambda$ determines the radius of both the exponential
interval $E_{\ell}$ and the projection interval $I_{\ell}$.  The radii
of these intervals is equal to $\log \tfrac{1}{\theta}$, up to bounded
additive error.  For a single leaf $\ell$, we define the \emph{radius
  function} $\rho_{\gamma, \ell}$ to be
\begin{equation}\label{eq:radius leaf}
\rho_{\gamma, \ell}(t) = \left\lfloor \log \frac{1}{d_{\PSL(2,
    \RR)}(\gamma^1(t), \ell^1)} \right\rfloor_1,
\end{equation}
where $\ell^1$ is the lift of $\ell$ in $\PSL(2, \RR)$.
Up to bounded additive error, the value of the radius function at a point of intersection equals the radius of the projection interval for the leaf of intersection.  At other points $t$, again up to bounded additive error, the value of the
radius function equals the largest radius of an interval centered
at $t$ that is contained in the projection interval.

\medskip
As the distance to one of the extended laminations is the infimum of
the distance to any leaf of the laminations, we may define radius
functions for the extended laminations as follows,
\[ \rho_{\gamma, \overline{\Lambda}_+}(t) = \sup_{\ell \in
  \overline{\Lambda}_+} \rho_{\gamma, \ell}(t) \text{ and }
\rho_{\gamma, \overline{\Lambda}_-}(t) = \sup_{\ell \in
  \overline{\Lambda}_-} \rho_{\gamma, \ell}(t). \]

We now estimate the radius function for a lamination at an intersection point using the radius function for the leaf of
intersection.  The exponential bounds on the distance between the
lifts of two geodesics to the unit tangent bundle, from
\Cref{prop:fellow travel}, become linear bounds for the logarithm of
the reciprocal of the distance function.  In particular, taking
logarithms of \eqref{eq:fellow travel} gives
\begin{equation}\label{eq:one over log}
\log \tfrac{1}{\theta} - |t| - \log L_0 \le \log
\frac{1}{d_{\PSL(2, \RR)}(\gamma^1(t), \ell^1)} \le \log \tfrac{1}{\theta}
- |t| + \log L_0,
\end{equation}
and these bounds hold for $|t| \le \log \tfrac{1}{\theta}$.

\medskip
Suppose that in $\PSL(2, \RR)$ the point on $\gamma$ closes to $\ell$ is $\gamma(t_\ell)$, with
$d_{\PSL(2, \RR)}(\gamma(t), \ell) = \theta_\ell$.  Recall that the
\emph{exponential interval} $E_\ell$ for $\ell$ is
$[t_\ell - \frac{1}{\log \theta_\ell}, t_\ell + \log
\tfrac{1}{\theta_\ell}]$.  For a compact interval $I \subset \RR$ with
length $|I|$ and midpoint $m$, define the \emph{absolute value
  function} $| \cdot |_I$ to be $|t|_I = \lfloor |I| - |t| \rfloor_0$,
as illustrated in \Cref{fig:absolute value function}.  We remark that
for $t \in I$, $|t|_I$ is equal to the distance from $t$ to the
nearest endpoint of $I$, so for any $t \in I$, the interval
$[ t - |t|_I, t + |t|_I ] \subseteq I$.

\begin{figure}[h]
\begin{center}
\begin{tikzpicture}%[scale=0.75]

\tikzstyle{point}=[circle, draw, fill=black, inner sep=1pt]

\draw [arrows=->] (0, 0) -- (8, 0) node [label=below:$t$] {};

\draw [thick, arrows=|-|] (2, 0) -- (6, 0) node [pos=0.2,
label=below:$I$] {};

\draw (4, 0) -- (4, -0.2) node [label=below:$m$] {};

\draw (2, 0) -- (4, 2) -- (6, 0) node [midway, label=right:${|t|_I = \lfloor m
- |t| \rfloor_0}$] {};

\end{tikzpicture}
\end{center}
\caption{An absolute value function.} \label{fig:absolute value function}
\end{figure}
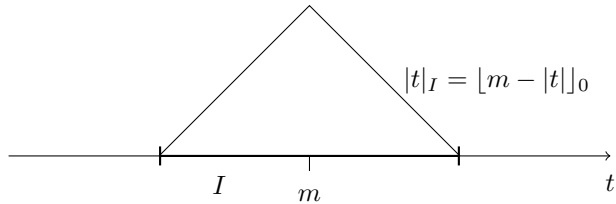

With this notation, we may rewrite \eqref{eq:one over log} as
\begin{equation}\label{eq:interval abs}
|t|_{E_\ell} - K \le \rho_{\gamma, \ell}(t) \le |t|_{E_\ell} + K,
\end{equation}
where $K = \log L_0$.

\medskip
We now use the above observations to show that if the value of the
radius function $\rho_{\gamma, \ell}(t)$ is sufficiently large, then
there is an interval centered at $t$ of radius
$\rho_{\gamma, \ell}(t)$ (up to bounded additive error) contained in
the exponential interval $E_\ell \subset \gamma$ determined by the
leaf $\ell$.

\begin{proposition}\label{prop:radius function}
There is a constant $K$ such that for any geodesics $\gamma$ and
$\ell$ in $\HH^2$, with unit speed parametrizations, then if
$\rho_{\gamma, \ell}(t) \ge K$, then the interval centered at $t$ of
radius $\rho_{\gamma, \ell}(t) - K$ is contained in the exponential
interval $E_\ell \subset \gamma$ determined by $\ell$.
\end{proposition}

\begin{proof}
We shall choose $K = \log L_0$, where $L_0$ is the constant from
\eqref{eq:fellow travel}.

\medskip
Suppose that in $\PSL(2, \RR)$ the closest point on $\gamma$ to $\ell$ is $\gamma(t_\ell)$ and let the closest distance be $\theta_\ell$.  Then
the exponential interval is
$E_\ell = [t_\ell - \log \tfrac{1}{\theta_\ell}, t_\ell + \log
\tfrac{1}{\theta_\ell}]$.  Using \eqref{eq:interval abs}, if
$\rho_{\gamma, \ell}(t) \ge K$, then $t \in E_\ell$ and
$\left| \rho_{\gamma, \ell}(t) - |t|_{E_\ell} \right| \le K$.  For
points $t \in E_\ell$, the value of $|t|_{E_\ell}$ is equal to the
distance from $t$ to the nearest endpoint of $E_\ell$, and so
\[ [t - |t|_{E_\ell}, t + |t|_{E_\ell}] \subseteq E_\ell. \]
Again, using \eqref{eq:interval abs} to estimate
$\rho_{\gamma, \ell}(t)$ in terms of the absolute value function
$|t|_{E_\ell}$ gives
\[ [t - ( \rho_{\gamma, \ell}(t) - K), t + ( \rho_{\gamma, \ell}(t) -
K ) ] \subseteq E_\ell, \]
as required.
\end{proof}

We now show that if $\gamma(t_1)$ is an intersection point for
$\gamma$ with a geodesic $\ell_1$, and if the radius function at $t_1$
for another geodesic $\ell_2$ is sufficiently larger than
$\rho_{\gamma, \ell_1}(t_1)$, then the endpoints of $\ell_1$ separate the endpoints of
$\ell_2$ and hence $\ell_1$ and $\ell_2$
intersect.

\begin{proposition}\label{prop:radius implies intersect}
There is a constant $K \ge 1$, such that for any geodesics $\gamma$
and $\ell_1$ in $\HH^2$ which intersect at $\gamma(t_1)$ at angle
$\theta_1$, then for any other geodesic $\ell_2$, if the radius
function at $t_1$ is sufficiently large, i.e.
$\rho_{\gamma, \ell_2}(t_1) \ge \log \tfrac{1}{\theta_1} + K$, then
$\ell_1$ and $\ell_2$ intersect.
\end{proposition}

\begin{proof}
Set $K = T_0 + K_1 + 1$, where $T_0$ and $K_1$ are respective constants from
\Cref{prop:projection interval} and \Cref{prop:radius function}.

\medskip
By \Cref{prop:projection interval}, the projection interval
$I_{\ell_1} \subset \gamma$ for $\ell_1$ is contained in
$[t_1 - \log \tfrac{1}{\theta_1} - T_0, t_1 + \log \tfrac{1}{\theta_1}
+ T_0]$.  As we have chosen $K \ge K_1 + T_0 + 1$, and we have assumed
that $\rho_{\gamma, \ell_2}(t_1) \ge \rho_{\gamma, \ell_1}(t_1) + K$,
\Cref{prop:radius function} implies that the exponential interval
$E_{\ell_2}$ contains the interval centered at $t_1$ of radius
$\log \tfrac{1}{\theta_1} + K - K_1 \ge \log \tfrac{1}{\theta_1} + T_0
+ 1$, so $I_{\ell_1} \subset E_{\ell_2}$, where the inclusion is
strict.

\medskip
As the exponential interval $E_{\ell_2}$ is contained in the
projection interval $I_{\ell_2}$, this implies that the projection
interval $I_{\ell_1} \subset I_{\ell_2}$.  As $\ell_1$ intersects
$\gamma$, \Cref{prop:nested implies intersect} implies that $\ell_1$
and $\ell_2$ intersect, as required.
\end{proof}

For later use, we record the following estimate of the radius function
of the extended lamination at an intersection point in terms of the
radius function for the leaf of intersection.

\begin{proposition}\label{prop:radius estimate}
There is a constant $K \ge 1$, such that for any closed hyperbolic
surface $S_h$ and lamination $\overline{\Lambda}$, if a geodesic
$\gamma$ intersects a leaf $\ell \in \Lambda$ at angle $\theta$ at
$\gamma(t)$, then
\begin{equation}\label{eq:radius estimate}
\rho_{\gamma, \ell}(t) \le \rho_{\gamma, \overline{\Lambda}}(t)
\le \rho_{\gamma, \ell}(t) + K,
\end{equation}
where $\overline{\Lambda}$ is the extended lamination corresponding to
$\Lambda$.
\end{proposition}

\begin{proof}
The left hand inequality follows directly from
$\rho_{\gamma, \overline{\Lambda}}$ being the supremum of
$\rho_{\gamma, \ell}$ over all leaves $\ell \in \overline{\Lambda}$.

\medskip
If $\ell$ is a leaf of a (non-extended) lamination, then it is
disjoint from all leaves in the corresponding extended
lamination.  \Cref{prop:radius implies intersect} shows that the value
of the radius function determined by $\ell$, at the intersection point
of $\ell$ and $\gamma$, is at least the radius
function at that point determined by any other leaf of the extended
lamination, up additive error at most $K$, where $K$ is the constant
from \Cref{prop:radius implies intersect}.
\end{proof}

%%%%%%%%%%%%%%%%%%%%%%%%%%%%%%%%%%%%%%%%%%%%%%%%%%%%%%%%%%%%%%%%%%%%%%%%%%%%%%%%%%
\subsubsection{The choice of constant for the height function}\label{section:height function}
%%%%%%%%%%%%%%%%%%%%%%%%%%%%%%%%%%%%%%%%%%%%%%%%%%%%%%%%%%%%%%%%%%%%%%%%%%%%%%%%%%

The constant $\theta_\LL$ in the definition of the height function
needs to be chosen to be sufficiently small, and it depends on the
hyperbolic metric $S_h$ and the pair of regular laminations $\Lambda$.
We now give an explicit choice of $\theta_\LL$ which suffices for our
purposes.  For closed subsets $A$ and $B$ of the unit tangent bundle
$T^1(S_h)$, let
\[ d_{\PSL(2, \RR)}(A, B) = \min \{ d_{\PSL(2, \RR)}(a, b) \mid a \in A, b \in
B \}. \]

The constant $\theta_\LL$ needs to be less than half the distance
between the extended laminations.  However, our argument uses various
properties of the geometry of the laminations, and
so $\theta_\LL$ will also depend on:

\begin{enumerate}

\item The constant $\alpha_\LL$ from \Cref{prop:angle bound}, giving
the smallest angle of intersection between leaves of the two
laminations.

\item The constant $L_\LL$ from \Cref{prop:cobounded intersections},
giving the diameter of the compact complementary regions of
$S_h \setminus (\Lambda_+ \cup \Lambda_-)$.

\item The constant $T_0$ from \Cref{prop:projection
  interval}, giving the size of the nearest point projection intervals
between geodesics in $\HH^2$.

\item The constant $\rho_\LL$ from \Cref{prop:overlap}, giving an upper bound
on the diameter of the overlap between the nearest point projections
of any two intersecting geodesics $\ell^+ \in \Lambda^-$ and
$\ell^- \in \Lambda^-$ to any other geodesic $\gamma$.

\item The constants $Q_\LL$ and $c_\LL$ from \Cref{prop:qi} giving the
quasi-isometry between the hyperbolic metric $d_{\HH^2}$ and the
Cannon--Thurston pseudometric $d_{\widetilde S_h}$ on $\widetilde S_h$.

\item The constants $\theta_0$ and $L_0$ from 
\Cref{prop:fellow travel}, giving the bi-Lipschitz bounds on the rates of
divergence of lifts of close geodesics.

\item The constant $D_\LL$ from \Cref{prop:innermost polygon diameter
  bound}, giving an upper bound on the diameter of any innermost polygon.

\item The constant $\theta_P$ from \Cref{prop:polygon},
which ensures that the lift of a non-exceptional geodesic is close in $T^1(S_h)$ to at
most one of the extended leaves in an innermost non-rectangular polygon.

\end{enumerate}

All constants above depend only on $(S_h , \Lambda)$, and do not
depend on the non-exceptional geodesic $\gamma$.  We now
define $\theta_\LL$.
\begin{definition}\label{def:theta}
Let
\[ \theta_{\min} = \min \{ \alpha_\LL, \rho_\LL, \tfrac{1}{2}
d_{\PSL(2, \RR)}(\overline{\Lambda}^1_-, \overline{\Lambda}^1_+), \theta_0 ,
\tfrac{1}{L_0} , \theta_P , 1 \}, \]
and then set
\[ \theta_\LL = \theta_{\min}^6 e^{-6( T_0 + L_\LL + 3 \rho_\LL + D_\LL
  + Q_\LL c_\LL)},  \]
where all of these constants from Propositions \ref{prop:cobounded
  intersections}, \ref{prop:projection interval}, \ref{prop:overlap},
\ref{prop:qi}, \ref{prop:fellow travel}, \ref{prop:innermost polygon
  diameter bound} and \ref{prop:polygon} only depend on $(S_h, \LL)$.
\end{definition}

%%%%%%%%%%%%%%%%%%%%%%%%%%%%%%%%%%%%%%%%%%%%%%%%%%%%%%%%%%%%%%%%%%%%%%%%%%%%%%%%%%
\subsubsection{Estimating the height function}
%%%%%%%%%%%%%%%%%%%%%%%%%%%%%%%%%%%%%%%%%%%%%%%%%%%%%%%%%%%%%%%%%%%%%%%%%%%%%%%%%%

If $\gamma(t)$ is a point of intersection of $\gamma$ and a leaf
$\ell$ of one of the (non-extended) laminations, then we can use the
angle of intersection between $\gamma$ and $\ell$ to estimate the
value of the height function at the intersection point.  In fact, we
can define and use a (signed) height function for a single leaf.
%by replacing one of the laminations by a lamination consisting of a single leaf.

\medskip
We define the (signed) height function for a single leaf
$\ell$ of one of the invariant laminations, as follows,
\[ h_{\gamma, \ell}(t) = \begin{cases}
\phantom{-} \log_k \left\lfloor \rho_{\gamma, \ell}(t) - \log \tfrac{1}{\theta_\LL}
\right\rfloor_1  & \text{ if } \ell \in \Lambda_+ \\[10pt]
- \log_k \left\lfloor \rho_{\gamma, \ell}(t) - \log
\tfrac{1}{\theta_\LL} \right\rfloor_1 & \text{ if } \ell \in \Lambda_- \end{cases} , \]
where again, by our choice of $\theta_\LL$, at most one of the
terms on the right hand side above may be non-zero.

\medskip
We can also rewrite the height function in terms of the radius
functions for the extended laminations,
\[ h_{\gamma}(t) = \log_k \left\lfloor \rho_{\gamma,
  \overline{\Lambda}_+}(t) - \log \tfrac{1}{\theta_\LL}
\right\rfloor_1 - \log_k \left\lfloor \rho_{\gamma,
  \overline{\Lambda}_-}(t) - \log \tfrac{1}{\theta_\LL} \right\rfloor_1
. \]

We now show that at an intersection point, the signed height function
for the leaf of intersection, which depends only on the angle of
intersection, can be used to approximate the height function for the
invariant laminations.

\begin{proposition}\label{prop:height estimate}
Suppose that $(S_h, \Lambda)$ is a hyperbolic metric on $S$ together
with a suited pair of measured laminations.  Then there is a
constant $K$ such that if $\gamma$ is any non-exceptional geodesic
$\gamma$ in $\widetilde S_h$, and $\ell^+$ is a leaf of the invariant
lamination $\Lambda_+$ which intersects $\gamma$ at $\gamma(t)$ at
angle $\theta$, then
\begin{align*}
h_{\gamma, \ell^+}(t) \le \ & h_\gamma(t) \le h_{\gamma, \ell^+}(t) + K & \text{ if }  \theta \le \theta_\LL \\
& h_\gamma(t) \le K &  \text{ if }  \theta \ge \theta_\LL.
\end{align*}
Similarly, if $\gamma(t)$ is an intersection point of $\gamma$ with a
leaf $\ell^- \in \Lambda_-$ of angle $\theta$, then
\begin{align*}
  h_{\gamma, \ell^-}(t) - K \le \
  & h_\gamma(t) \le h_{\gamma, \ell^-}(t) & \text{ if }  \theta \le \theta_\LL \\
- K \le \   & h_\gamma(t) &  \text{ if }  \theta \ge \theta_\LL.
\end{align*}
\end{proposition}

We remark that in \Cref{prop:height estimate}, although the definition
of the height function depends on the cutoff constant $\theta_\LL$,
the additive error constant $K$ depends only on $(S_h, \Lambda)$.

\medskip
As the radius functions determine the height function, we can now
complete the proof of \Cref{prop:height estimate}, showing that the
value of the height function at an intersection point is equal to the
value of the leafwise height function for the leaf of intersection at
the intersection point, up to bounded additive error.

\begin{proof}[Proof (of \Cref{prop:height estimate})]
Up to reparametrizing $\gamma$ by a translation, suppose that $\gamma(0)$ is an intersection point of $\gamma$
with $\ell_1 \in \Lambda_+$ with angle $\theta$.  The argument is
the same in the other case up to swapping the laminations and
reversing the sign of the height function. 

\medskip
We choose $K = \log_k K_1$, where $K_1 \ge 1$ is the constant from
\Cref{prop:radius implies intersect}.

\medskip
We first show the upper bound.  As $\overline{\Lambda}_+$ is closed,
there is a leaf $\ell_2$ in the extended lamination
$\overline{\Lambda}_+$, realizing the height function,
i.e. $h_\gamma(0) = h_{\gamma, \ell_2}(0)$.  As $\ell_1$ is a leaf of
$\Lambda_+$, is it disjoint from all other leaves in the extended
lamination $\overline{\Lambda}_+$.  Therefore, by \Cref{prop:radius
  implies intersect} the radius function for $\ell_1$ is a coarse
upper bound for the radius function for $\ell_2$ at $t = 0$, i.e.
\[ \rho_{\gamma, \ell_2}(0) \le \rho_{\gamma, \ell_1}(0) + K_1, \]
where $K_1$ is the constant from \Cref{prop:radius implies intersect}.
Subtracting $\log \tfrac{1}{\theta_\LL}$ from each side gives
\[ \rho_{\gamma, \ell_2}(0) - \log \tfrac{1}{\theta_\LL} \le
\rho_{\gamma, \ell_1}(0) - \log \tfrac{1}{\theta_\LL} + K_1. \]
Using the elementary observation that $\lfloor x + y \rfloor_1 \le
\lfloor x \rfloor_1 + \lfloor y \rfloor_1$, and as $K_1 \ge 1$, 
\[ \left\lfloor \rho_{\gamma, \ell_2}(0) - \log \tfrac{1}{\theta_\LL}
\right\rfloor_1 \le \left\lfloor \rho_{\gamma, \ell_1}(0) - \log
\tfrac{1}{\theta_\LL} \right\rfloor_1 + K_1. \]
As $\log_k(x)$ is $(1/\log k)$-Lipschitz for $x \ge 1$,
\[  h_{\gamma, \ell_2}(0) \le  h_{\gamma, \ell_1}(0) + K_1/\log k,   \]
as required.

\medskip
For the lower bound, if $\theta \le \theta_\LL$, then by \Cref{def:height function2}, $h_{\gamma, \ell^+_1}(t)$ is a lower bound for
$h_\gamma(t)$.  However, if $\theta \ge \theta_\LL$, then the
contribution of distance to leaves of $\Lambda_+$ to the height
function may be zero, and the height function may be determined by
distance to leaves of $\overline{\Lambda}_-$, and so then there is no
lower bound.
\end{proof}

%%%%%%%%%%%%%%%%%%%%%%%%%%%%%%%%%%%%%%%%%%%%%%%%%%%%%%%%%%%%%%%%%%%%%%%%%%%%%%%%%%
\subsubsection{The radius function is Lipschitz}
%%%%%%%%%%%%%%%%%%%%%%%%%%%%%%%%%%%%%%%%%%%%%%%%%%%%%%%%%%%%%%%%%%%%%%%%%%%%%%%%%%

A function $f \colon \RR \to \RR$ is \emph{$c$-Lipschitz} if
$ | f(x) - f(y) | \le c |x - y|$, and for differentiable functions
this is equivalent to $|f'(x)| \le c$.  In this section, we show that
the radius function is $1$-Lipschitz.

\begin{proposition}\label{prop:exp 1lip}
Suppose that $(S_h, \Lambda)$ is a hyperbolic metric on $S$ together
with a suited pair of measured laminations.  Suppose that $\gamma$
is a non-exceptional geodesic in $\widetilde{S}_h$ with unit speed
parametrization $\gamma(t)$. Then the radius functions
$\rho_\gamma(t), \rho_{\gamma, \Lambda_+}(t)$ and
$\rho_{\gamma, \Lambda_-}(t)$ are all $1$-Lipschitz.
\end{proposition}

In fact, as the derivative of $\log_k (x)$ takes values between $0$
and $1/\log k$ for $x \ge 1$, \Cref{prop:exp 1lip} also implies that
the height function, which is defined in terms of $\log$ of the radius
function, is $(1/\log k)$-Lipschitz, which we record for future
reference.

\begin{corollary}\label{cor:height lipschitz}
Suppose that $(S_h, \Lambda)$ is a hyperbolic metric on $S$ together
with a suited pair of measured laminations.  The for any
non-exceptional geodesic $\gamma$ with unit speed parametrization
$\gamma(t)$, the height function $h_\gamma(t)$ is
$(1/\log k)$-Lipschitz. \qed
\end{corollary}

\begin{proof}
Recall the definition of the height function,
\[ h_{\gamma}(t) = \log_k \left\lfloor \rho_{\gamma,
  \overline{\Lambda}_+}(t) - \log \tfrac{1}{\theta_\LL}
\right\rfloor_1 - \log_k \left\lfloor \rho_{\gamma,
  \overline{\Lambda}_-}(t) - \log \tfrac{1}{\theta_\LL}
\right\rfloor_1 . \]
If $f(x)$ is $1$-Lipschitz, then $\lfloor f(x) - a \rfloor_b$ is also
$1$-Lipschitz for any $a$ and $b$.  As the derivative of $\log_k(x)$
takes values in $(0, 1/\log k]$ for $x \ge 1$, each term on the right
hand side above is $(1/\log k)$-Lipschitz.  As a sum or difference of
$(1/\log k)$-Lipschitz functions is $(1/\log k)$-Lipschitz, the result
follows.
\end{proof}

Suppose that $\gamma$ is a geodesic with unit speed parametrization
$\gamma \colon \RR \to \widetilde{S}_h$.
Suppose that $R$ is either an
ideal complementary region of one of the invariant laminations, or a
compact complementary region of their union.  We define the
\emph{intersection interval} $I_R$ to be the closure of the pre-image
$\gamma^{-1}(R)$.  If $R$ is an innermost polygon, then we say that
$I_R$ is an \emph{innermost intersection interval}, i.e. the interior
of $\gamma(I_R)$ is disjoint from the invariant laminations.

\medskip
From \Cref{cor:height lipschitz} we deduce below that the test path
over an innermost intersection interval has bounded arc length.

\begin{corollary}\label{prop:innermost bounded length}
Suppose that $(S_h, \Lambda)$ is a hyperbolic metric on $S$ together
with a suited pair of measured laminations.  Then there is a
constant $K$, such that for any innermost intersection interval
$\gamma(I_R)$, the arc length of $\tau_\gamma(I_R)$ is at most $K$.
\end{corollary}

\begin{proof}
By \Cref{prop:cobounded intersections}, as the interior of
$\gamma(I_R)$ is disjoint from both laminations, the hyperbolic length
of $\gamma(I_R)$ is at most $L_\LL$.  By \Cref{cor:height lipschitz}
the height function is $(1/\log k)$-Lipschitz.  As $\gamma(I_R)$ is
disjoint from the invariant laminations, its length is determined by
the vertical $z$-coordinate, so the length of $\tau_\gamma(I_R)$ is at
most $K = L_\LL / \log k$, which only depends on $(S_h, \LL)$, as
required.
\end{proof}

We prove below that the radius function determined by a single leaf is $1$-Lipschitz. Since the height is defined in terms of distance in the unit
tangent bundle to the two extended invariant laminations, and the distance to an extended lamination is the infimum of the distance
to all of the leaves in the lamination, the required Lipschitz property for the the height function will follow from the fact that the supremum of
$1$-Lipschitz functions is $1$-Lipschitz.

\begin{proposition}\label{prop:ell lip}
Suppose that $(S_h, \Lambda)$ is a hyperbolic metric on $S$ together
with a suited pair of measured laminations.  Then for any geodesic
$\gamma$ with unit speed parametrization $\gamma(t)$, and any distinct
geodesic $\ell$, the radius function $\rho_{\gamma, \ell}(t)$ is
$1$-Lipschitz.
\end{proposition}

\begin{proof}
We simplify notation by setting
$d(t) = d_{\PSL(2, \RR)}(\gamma^1(t), \ell^1)$.  It suffices to bound the
derivative of $\log d_{\PSL(2, \RR)}( \gamma^1(t), \ell^1 ) = \log d(t)$,
as this is equal to the negative of the radius function where $d(t)
\le 1/e$.  When $d(t) \ge 1/e$, the radius function is the constant
function $1$, and is automatically $1$-Lipschitz.

\medskip Let $\alpha$ be a geodesic arc in $\PSL(2, \RR)$, realizing
$d(t) = d_{\PSL(2, \RR)}(\gamma^1(t), \ell^1)$, i.e.
$\text{length}(\alpha) = d(t)$.  Let
$\phi_t \colon \PSL(2, \RR) \to \PSL(2, \RR)$ be the geodesic flow on
$\PSL(2, \RR)$.  Then $\phi_h \alpha$ is a path in $\PSL(2, \RR)$ from
$\gamma(t+h)$ to $\ell$.  It is well known that the geodesic flow
$\phi_h$ in $\PSL(2, \RR)$ expands or contracts distances by at most
$e^h$, see for example \cite{manning}*{page 75}.  So
$\text{length}(\phi_h \alpha) \le e^h d(t)$.  Since the length of
$\phi_t \alpha$ is an upper bound for the distance from $\gamma(t+h)$
to $\ell$, we have
\begin{equation}\label{eq:1lip upper}
d(t+h) \le e^h d(t).
\end{equation}

Similarly, let $\beta$ be a geodesic in $\PSL(2, \RR)$ realizing the
distance from $\gamma(t+h)$ to $\ell$.  Then $\phi_{-h} \beta$ is a
path from $\gamma(t)$ to $\ell$.
This gives an upper bound of
$e^h \text{length}(\beta)$ on the distance from $\gamma(t)$ to
$\ell$, and hence
\begin{equation}\label{eq:1lip lower}
d(t) \le e^h d(t+h).
\end{equation}

Combining \eqref{eq:1lip upper} and \eqref{eq:1lip lower} gives
\[ \log e^{-h} d(t) - \log d(t) \le \log d(t+h) - \log d(t) \le \log
e^h d(t) - \log d(t) \]
which simplifies to
\[ \left| \log d(t+h) - \log d(t) \right| \le | h |. \]
Thus, the radius function determined by a single geodesic is
$1$-Lipschitz, as required.
\end{proof}

We may now complete the proof of \Cref{prop:exp 1lip}.

\begin{proof}[Proof of \Cref{prop:exp 1lip}]
As
$\rho_\gamma(t) = \max \{ \rho_{\gamma, \Lambda_+}(t), \rho_{\gamma,
  \Lambda_-}(t) \}$ it suffices to show that
$\rho_{\gamma, \Lambda_+}(t)$ and $\rho_{\gamma, \Lambda_-}(t)$ are
$1$-Lipschitz.  We give the argument for $\Lambda_+$, the same
argument works for $\Lambda_-$.

The distance from $\gamma^1(t)$ to the extended lamination
$\overline{\Lambda}^1_+$ is the infimum of distances to each leaf
$\ell^1 \in \overline{\Lambda}^1_+$, that is
\[ d_{\PSL(2, \RR)}(\gamma^1(t), \overline{\Lambda}^1_+) = \inf_{\ell \in
  \overline{\Lambda}_+} d_{\PSL(2, \RR)}(\gamma^1(t), \ell^1). \]

\medskip

Recall the definition of the radius function,
\begin{align*}
\rho_{\gamma, \Lambda_+}(t) = & \left\lfloor \log \frac{1}{d_{\PSL(2, \RR)}( \gamma^1(t),
  \overline{\Lambda}^1_+)} \right\rfloor_1, \\
\intertext{as the reciprocal function is decreasing, and the
  logarithm function is increasing,}
\rho_{\gamma, \Lambda_+}(t) = & \sup_{\ell \in
                      \overline{\Lambda}_+} \rho_{\gamma, \ell}(t).
\end{align*}
The radius function for an individual leaf is $1$-Lipschitz by
\Cref{prop:ell lip}, and a supremum of $1$-Lipschitz functions is
$1$-Lipschitz, and so the radius function is $1$-Lipschitz, as
required.
\end{proof}

%%%%%%%%%%%%%%%%%%%%%%%%%%%%%%%%%%%%%%%%%%%%%%%%%%%%%%%%%%%%%%%%%%%%%%%%%%%%%%%%%%
\subsection{Tame bottlenecks}\label{section:tame}
%%%%%%%%%%%%%%%%%%%%%%%%%%%%%%%%%%%%%%%%%%%%%%%%%%%%%%%%%%%%%%%%%%%%%%%%%%%%%%%%%%

In this section, we prove \Cref{cor:corner distance} that corner segments create tame bottlenecks. To do so, we first define transverse rectangles for a geodesic, namely rectangles of positive measure such that the geodesic crosses all leaves containing the sides of the rectangle.
We then show that transverse rectangles give rise to bottlenecks.  
Our construction does not come with a bound on the arc length of
the test path segment between the bottleneck sets. However, we show
that by making the rectangles smaller and thus increasing the size of the
bottleneck sets, there is a transverse rectangle with a bound on arc length of the test path segment, i.e. the
bottlenecks are tame.

%%%%%%%%%%%%%%%%%%%%%%%%%%%%%%%%%%%%%%%%%%%%%%%%%%%%%%%%%%%%%%%%%%%%%%%%%%%%%%%%%%
\subsubsection{Transverse rectangles}
%%%%%%%%%%%%%%%%%%%%%%%%%%%%%%%%%%%%%%%%%%%%%%%%%%%%%%%%%%%%%%%%%%%%%%%%%%%%%%%%%%

\begin{definition}
\label{def:transverse rectangle}
Given a geodesic $\gamma$, we say that a rectangle of positive measure
is a \emph{transverse rectangle} for $\gamma$ if $\gamma$ crosses all
leaves containing the sides of the rectangle.
\end{definition}

\medskip

\begin{figure}[h]
\begin{center}
\begin{tikzpicture}[scale=0.75]

\tikzstyle{point}=[circle, draw, fill=black, inner sep=1pt]

\def\boundary{(0, 0) circle (4)}

\def\redone{(270:11.70) circle (11)}
\def\redtwo{(90:11.70) circle (11)}
\def\greenone{(180:11.70) circle (11)}
\def\greentwo{(0:11.70) circle (11)}

\begin{scope}
    \clip \boundary;
    \clip \greenone;
    \draw [color=white, fill=black!20] \redone;
\end{scope}

\draw (-2.25, -2) node {$U$};

\begin{scope}
    \clip \boundary;
    \clip \greentwo;
    \draw [color=white, fill=black!20] \redtwo;
\end{scope}

\draw (1.5, 2) node {$V$};

\begin{scope}
\clip \boundary;
\draw [color=red] \redone;
\draw [color=red] \redtwo;
\draw [color=ForestGreen] \greenone;
\draw [color=ForestGreen] \greentwo;
\end{scope}

\draw (-1.5, -4.5) node [color=ForestGreen] {$\alpha^+$};
\draw (1.5, -4.5) node [color=ForestGreen] {$\beta^+$};
\draw (4.5, -1.5) node [color=red] {$\alpha^-$};
\draw (4.5, 1.5) node [color=red] {$\beta^-$};

\draw [thick, arrows=->] (235:4) -- (45:4) node [label=right:$\gamma$] {};

\draw (-2.5, 0) node {$W_1$};
\draw (-2, 2) node {$W_2$};
\draw (0, 2.5) node {$W_3$};

\draw (-0.25, 0.25) node {$R$};

\draw \boundary;

\end{tikzpicture}
\end{center}
\caption{A transverse rectangle.} \label{fig:transverse rectangle}
\end{figure}
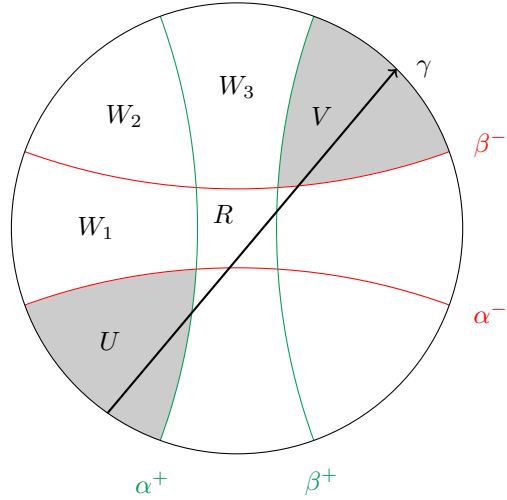

A choice of unit speed parametrization for a geodesic $\gamma$ orders the leaves of the laminations intersecting $\gamma$. Using the ordering and our conventions for rectangles illustrated in \Cref{fig:transverse rectangle}, we specify some notation for transverse rectangles. Note that the geodesic need not itself intersect the rectangle.  

\begin{definition}\label{def:leaf order}
Suppose that $\gamma$ is a non-exceptional geodesic parametrized with
unit speed.  Suppose that $\gamma$ intersects leaves $\ell_1$ and
$\ell_2$ in either invariant lamination $\Lambda_+ \cup \Lambda_-$ at
points $\gamma(t_1)$ and $\gamma(t_2)$, respectively.  If
$t_1 \le t_2$ then we say that $\ell_1 \le_\gamma \ell_2$.
\end{definition}

Suppose that $R$ is a rectangle with positive measure transverse to a
geodesic $\gamma$.  The rectangle has two sides contained in leaves of
$\Lambda_+$, which we shall label $\alpha^+$ and $\beta^+$ so that
$\alpha^+ \le_\gamma \beta^+$.  Similarly, the rectangle has two sides
contained in leaves of $\Lambda_-$, which we shall label $\alpha^-$
and $\beta^-$ so that $\alpha^- \le_\gamma \beta^-$.

\medskip

As $\gamma$ is oriented, it has an initial limit point
$\overline{\gamma}_-$ and a terminal limit point
$\overline{\gamma}_+$.  We call the quadrant whose limit set contains
$\overline{\gamma}_-$ the \emph{initial quadrant}, and the quadrant
whose limit set contains $\overline{\gamma}_+$ the \emph{terminal
  quadrant}.  Using our notation illustrated in
\Cref{fig:transverse rectangle}, the region $U$, with boundary
contained in $\alpha^+$ and $\alpha^-$, is the initial quadrant, and
the region $V$, with boundary contained in $\beta^+$ and $\beta^-$, is
the terminal quadrant.

\medskip
Suppose that a non-exceptional geodesic $\gamma$ is transverse to a
rectangle $R$ with optimal height $z$.  The main result of this
section is that the optimal height rectangle $F_z(R)$ is a
\emph{$(r, K)$-bottleneck} with respect to the flow sets over the initial and terminal
quadrants.  The constant $K$ depends on the measure of the rectangle
(as well as various constants depending on $(S, \LL)$), and tends to
infinity as the area of the rectangle tends to zero.  In particular,
the geodesic $\overline{\gamma}$ in $\widetilde S_h \times \RR$ with
the same limit points as the path $\iota(\gamma)$ passes within a
bounded distance of the square $F_z(R)$.

\begin{lemma}\label{lemma:transverse}(Transverse rectangles create
bottlenecks.)  Suppose that $(S_h, \Lambda)$ is a hyperbolic metric on
$S$ together with a suited pair of measured laminations.  Suppose
that $\gamma$ is a non-exceptional geodesic that intersects a
rectangle $R$ with measure at least $A > 0$ and optimal height $z$.
Then there are constants $r > 0$ and $K \ge 0$ (that depends on $\LL$
and $A$) such that the optimal height rectangle
$F_z(R) = R \times \{ z \}$ is an $(r, K)$-bottleneck for the flow
sets $F(U)$ and $F(V)$ over the initial and terminal quadrants of $R$.

\medskip
In particular, the geodesic $\overline{\gamma}$ in
$\widetilde S_h \times \RR$ with the same limit points as
$\iota(\gamma)$ passes within Cannon--Thurston distance $K$ of the
optimal height rectangle $F_z(R)$ in $\widetilde S_h \times \RR$.
\end{lemma}

%%%%%%%%%%%%%%%%%%%%%%%%%%%%%%%%%%%%%%%%%%%%%%%%%%%%%%%%%%%%%%%%%%%%%%%
\subsubsection{Outermost rectangles for small angles}\label{section:small angle}
%%%%%%%%%%%%%%%%%%%%%%%%%%%%%%%%%%%%%%%%%%%%%%%%%%%%%%%%%%%%%%%%%%%%%%%

Suppose that a leaf $\ell$ intersects $\gamma$ at $\gamma(t)$. 
Since the laminations are closed, there is a
unique transverse rectangle containing $\gamma(t)$ whose side along $\ell$ has the largest measure among all such rectangles.  We
call this the \emph{outermost} transverse rectangle determined by
$\gamma$ and $\ell$. We show that as long as the
angle between $\ell$ and $\gamma$ is sufficiently small, the area of the outermost rectangle is bounded below.

\begin{definition}
Suppose that $\gamma$ is a non-exceptional geodesic intersecting a leaf 
$\ell_- \in \Lambda_-$ at a point $p$.  We say that a transverse rectangle $R$
containing $p$ is \emph{outermost} if it has the following properties.
\begin{itemize}
\item The sides $\alpha_+ \le_\gamma \beta_+$ of $R$ are contained in the outermost
leaves of $\Lambda_+$ intersecting both $\gamma$ and
$\ell_-$.
\item The sides $\alpha_- \le_\gamma \beta_-$ are contained in the outermost leaves
of $\Lambda_-$ intersecting all $\alpha_+, \beta_+$ and
$\gamma$.
\end{itemize}
Similarly, if $p$ is an intersection point of $\gamma$ and a leaf
$\ell_+$ of $\Lambda_+$, we say a transverse rectangle $R$ is
\emph{outermost} if it has the properties above with the two invariant
laminations swapped.
\end{definition}

\begin{proposition}\label{prop:small angle rectangle}
Suppose that $(S_h, \Lambda)$ is a hyperbolic metric on $S$ together
with a suited pair of measured laminations.  Let $Q_\LL$ and
$c_\LL$ be the constants from \Cref{prop:qi} and let $\theta_\LL$ be
the constant from \Cref{def:theta}.  Then there are positive constants
$A > 0$ and $K \ge 0$ such that for any non-exceptional geodesic
$\gamma$ in $S_h$, with unit speed parametrization, and with
$\gamma(t)$ a point of intersection between $\gamma$ and a leaf
$\ell \in \Lambda_+$ with angle $\theta \le \theta_\LL$, then the
outermost transverse rectangle $R$ determined by $\gamma \cap \ell$
has the following properties.
%
%\begin{enumerate}
\begin{thmenum}[label={(\ref{prop:small angle rectangle}.\arabic*)}]
\item \label{prop:long side} Let $\ell$ have a unit speed
parametrization in the hyperbolic metric.  Then the segment $\ell$
lying between $\alpha_-$ and $\beta_-$ is an interval
$\ell([- r_1, r_2])$ where $r_i = \log \tfrac{1}{\theta}$ up to
additive error at most
$\tfrac{1}{6} \log \tfrac{1}{\theta_\LL} - c_\LL$, where $\ell$ has
unit speed parametrization with intersection point $\ell(0)$ with $\gamma$.
\item \label{prop:long side measure} The sides of $R$ in $\Lambda_+$
have measure $dy(R)$ satisfying
\[ 0 < \tfrac{5}{3 Q_\LL} \log \tfrac{1}{\theta_\LL} \le
\tfrac{2}{Q_\LL} ( \log \tfrac{1}{\theta} - \tfrac{1}{6} \log
\tfrac{1}{\theta_\LL} ) \le dy(R) \le 2 Q_\LL ( \log \tfrac{1}{\theta}
+ \tfrac{1}{6} \log \tfrac{1}{\theta_\LL} ). \]
\item \label{prop:area bound} The measure of the rectangle $R$ is at
least $A = A_\LL / (3 Q^2_\LL) > 0$.
\item \label{prop:short side} The sides of $R$ in $\Lambda_-$ have
measure $dx(R)$ satisfying
\[ 0 < \frac{ A_f }{ 2 Q_\LL (\log \tfrac{1}{\theta} + \tfrac{1}{6} \log
  \tfrac{1}{\theta_\LL}) } \le dx(R) \le \frac{A_\LL}{
  \tfrac{2}{Q_\LL} ( \log \tfrac{1}{\theta} - \tfrac{1}{6} \log
  \tfrac{1}{\theta_\LL} ) } , \]
where $A_\LL > 0$ is the constant from \Cref{prop:min area}.
\item \label{prop:small angle optimal height} The rectangle $R$ has optimal height
$\log_k \log \frac{1}{\theta}$ up to additive error at most $K$.
%
%\end{enumerate}
\end{thmenum}

\medskip The same holds for $\gamma$ intersecting a leaf $\ell$ of
$\Lambda_-$, except bounds on the measures of the sides in $\Lambda_+$
and $\Lambda_-$ are swapped, and the optimal height of the outermost
transverse rectangle is $- \log_k \log \frac{1}{\theta}$ up to
additive error at most $K$.
\end{proposition}

\begin{proof}
Given a hyperbolic metric $S_h$ and a suited pair of laminations
$\LL$, we recall various constants from previous results.  Let
$\theta_\LL > 0$ be the constant from \Cref{def:theta}.  In this
argument, we will use the fact that
$\theta_\LL \le \alpha_\LL^2 e^{-2T_0 - 2 L_\LL - 2 Q_\LL c_\LL}$,
%
%thetacondition
%
where $\alpha_\LL$ is defined in \Cref{prop:angle bound}, $T_0$ is the
constant from \Cref{prop:projection interval}, $L_\LL$ is the constant
from \Cref{prop:cobounded intersections}, and $Q_\LL$ and $c_\LL$ are
the quasi-isometry constants from \Cref{prop:qi}.

\medskip
To simplify expressions, we will define a sequence of
constants $T_i$ during the argument.  They will all depend only on $\PSL(2, \RR)$ or
$(S_h, \LL)$ and in particular, not on
$\theta$.

\medskip
Suppose that a leaf $\ell$ of $\Lambda_-$ intersects a non-exceptional geodesic
$\gamma$ at the point $\gamma(t)$ with angle $\theta \le \theta_\LL$.
We parametrize $\ell$ with unit speed so that the intersection
point is $\ell(0)$.

\medskip
By \Cref{prop:projection interval}, the image of $\gamma$
under nearest point projection to $\ell$ is equal to an interval
$\ell( I_\gamma )$, where $I_\gamma = [-T_\gamma, T_\gamma]$ is a
symmetric interval about $t=0$, where $T_\gamma$ satisfies
$\log \tfrac{1}{\theta} \le T_\gamma \le \log \tfrac{1}{\theta} +
T_0$.

\medskip
By \Cref{prop:cobounded intersections}, there is a constant
$L_\LL >0$ such that every segment of $\ell$ of length $L_\LL$ intersects
a leaf of $\Lambda_-$.
In particular, for any interval of $\ell$ with length $L_\LL$,
nested sufficiently far inside the projection interval for $\gamma$
onto $\ell$, 
\begin{itemize}
    \item there will be a leaf of $\Lambda_-$ intersecting the 
     interval, and
     \item the nearest point projection interval of this leaf
     will be contained in the nearest point projection interval for
     $\gamma$ on $\ell$.
\end{itemize}
It follows that the endpoints of this leaf and the endpoints of
$\gamma$ are linked on the boundary circle, and so $\gamma$ and the
leaf intersect, as in \Cref{prop:nested implies intersect}.

\medskip
Similarly, for any interval of $\ell$ of length $L_\LL$, sufficiently
far from the projection interval for $\gamma$, the leaves that
intersect the interval will have projection intervals onto $\ell$
which are disjoint from the projection interval of $\gamma$ onto
$\ell$, and so these leaves are disjoint from $\gamma$.

\medskip
We wish to produce leaves of $\Lambda_-$ intersecting $\ell$ close to
the endpoints of $\ell(I_\gamma)$.  
To do so, we begin by picking four intervals $I_1, \ldots , I_4$ in $\ell$, all of length $L_\LL$, chosen so that
\begin{itemize}
    \item the intervals $\ell(I_1)$ and $\ell(I_4)$ are the innermost among intervals in $\ell - \ell(I_\gamma)$ such that any leaf of $\Lambda_+$ intersecting $\ell(I_1)$ or $\ell(I_4)$ does not intersect $\gamma$; and
    \item the intervals $\ell(I_2)$ and $\ell(I_3)$ are the outermost among intervals in $\ell(I_\gamma)$ so that any leaf of $\Lambda_+$ that intersects them also intersects $\gamma$.
\end{itemize}

We claim that the choice of the intervals is possible by sufficient
nesting that does not depend on $\theta$. We choose $T_1$ to be an
upper bound on the radius of the nearest projection of any leaf
$\ell_+ \in \Lambda_+$ to any leaf $\ell_- \in \Lambda_+$.  As the
angle of intersections of leaves in $\Lambda_+$ with leaves in
$\Lambda_-$ is at most $\alpha_\LL$ by \Cref{prop:angle bound}, we may
choose
\begin{equation}\label{eq:T1}
T_1 = T_0 + \log \tfrac{1}{\alpha_\LL},
\end{equation}
by \Cref{prop:neighbourhood}.  This choice of $T_1$ does not depend on
the angle $\theta$ between $\gamma$ and $\ell$.  We will choose the
outer intervals $I_1$ and $I_4$ to be distance $T_1$ outside
$I_\gamma$, and we will choose the inner intervals $I_2$ and $I_3$ to
be nested distance $T_1$ inside $I_\gamma$.  This is illustrated in
\Cref{fig:ell intervals}.

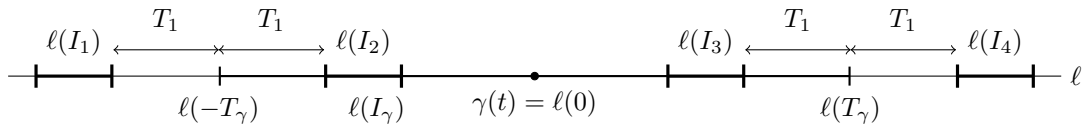
\begin{figure}[h]
\begin{center}
\begin{tikzpicture}[scale=0.7]

\tikzstyle{point}=[circle, draw, fill=black, inner sep=1pt]

\draw (-5, 3) -- (15, 3) node [right] {$\ell$};

\draw (5, 3) node [point, label=below:${\gamma(t) = \ell(0)}$] {};

\draw [thick, arrows=|-|] (-1, 3) node
[label=below:$\ell(-T_\gamma)$] {} -- (11, 3) node
[label=below:$\ell(T_\gamma)$] {} node [pos=0.25,
label=below:$\ell(I_\gamma)$] {};

\draw [very thick, arrows=|-|] (-4.5, 3) -- (-3, 3) node [midway,
label=above:$\ell(I_1)$] {};

\draw [very thick, arrows=|-|] (1, 3) -- (2.5, 3) node [midway,
label=above:$\ell(I_2)$] {};

\draw [very thick, arrows=|-|] (7.5, 3) -- (9, 3) node [midway,
label=above:$\ell(I_3)$] {};

\draw [very thick, arrows=|-|] (13, 3) -- (14.5, 3) node [midway,
label=above:$\ell(I_4)$] {};

\draw [arrows=<->] (-3, 3.5) -- (-1, 3.5) node [midway, label=above:$T_1$] {};

\draw [arrows=<->] (-1, 3.5) -- (1, 3.5) node [midway, label=above:$T_1$] {};

\draw [arrows=<->] (9, 3.5) -- (11, 3.5) node [midway, label=above:$T_1$] {};

\draw [arrows=<->] (11, 3.5) -- (13, 3.5) node [midway, label=above:$T_1$] {};

\end{tikzpicture}
\end{center}
\caption{Projection intervals on $\ell$.} \label{fig:ell intervals}
\end{figure}

\medskip
%
%thetacondition
%
In order for our construction to work, the projection interval
$I_\gamma$ for $\gamma$ must be sufficiently large.  We will require
$T_\gamma \ge T_1 + L_\LL$.  Equivalently,
$T_\gamma = \log \tfrac{1}{\theta} \ge T_0 + \log
\tfrac{1}{\alpha_\LL} + L_\LL$, which is satisfied as long as
$\theta \le \alpha_\LL e^{-T_0 - L_\LL}$.  This is guaranteed by our
choice of $\theta_\LL$ from \Cref{def:theta}, and so $T_\gamma$ is
large enough to construct the nested intervals.

\medskip
Consider now the first interval $\ell(I_1)$.  Since its length is
$L_f$, there is a leaf $\ell^-_1$ of $\Lambda_-$ which intersects
$\ell$, say at $\ell(t_1)$ for $t_1$ in $I_1$.  Since the angle of
intersection between any two leaves is at least $\alpha_\LL$, the
nearest point projection interval $\ell(I_{\ell^-_1})$ onto $\ell$ is
contained in a symmetric interval of radius
$\log \tfrac{1}{\alpha_\LL} + T_0$ centered at $t_1$.  This radius is
equal to our choice of $T_1$, and so $\ell(I_{\ell^-_1})$ is disjoint
from the projection interval $\ell(I_\gamma)$ for $\gamma$.  In
particular, $\ell^-_1$ is disjoint from $\gamma$.  This is illustrated
in \Cref{fig:angle rectangle}.  Exactly the same argument shows that
there is a leaf $\ell^-_4$ of $\Lambda_-$ intersecting $\ell(I_4)$
which is disjoint from $\gamma$.  In \Cref{fig:angle rectangle} we
have drawn $\gamma(t)$ in the interior of the rectangle we construct,
but $\gamma(t)$ may in fact lie on the boundary, as $\ell$ may be a
boundary leaf of the rectangle.

\medskip
Now consider the second interval $\ell(I_2)$.  As this interval has
length $L_f$, there is a leaf $\ell^-_2$ of $\Lambda_-$ which
intersects $\ell$, say at $\ell(t_2)$ in this interval $\ell(I_2)$.
Again, as the angle of intersections is at least $\alpha_\LL$, the
nearest point projection interval $\ell(I_{\ell^-_2})$ for this leaf
onto $\ell$ is contained in a symmetric interval of radius
$\log \tfrac{1}{\alpha_\LL} + T_0$ centered at $t_2$.  This radius is
equal to our choice of $T_1$, and so $\ell(I_{\ell^-_2})$ is contained
inside the projection interval $\ell(I_\gamma)$ for $\gamma$.  Hence
$\ell^-_2$ and $\gamma$ intersect.  Exactly the same argument shows
that there is a leaf $\ell^-_3$ of $\Lambda_-$ intersecting
$\ell(I_3)$ which also intersects $\gamma$.

\begin{figure}[h]
\begin{center}
\begin{tikzpicture}[scale=0.7]

\tikzstyle{point}=[circle, draw, fill=black, inner sep=1pt]
\tikzstyle{redpoint}=[circle, draw, fill=red, inner sep=1pt]

\draw [color=white, fill=black!20] (-0.25, 2) -- (9.75, 2) -- (10.25, 4) -- (0.25, 4) -- cycle;

\draw [color=ForestGreen] (-0.5, 1) -- (10.5, 1);
\draw [color=ForestGreen] (-0.5, 5) -- (10.5, 5);

\draw [color=ForestGreen] (-2.5, 2) -- (12.5, 2);
\draw [color=ForestGreen] (-3, 3) -- (13, 3) node [right] {$\ell$};
\draw [color=ForestGreen] (-2.5, 4) -- (12.5, 4);

\draw [color=red] (-0.75, 0) -- (0.75, 6) node [redpoint, midway, label=above
right:$\ell(t_2)$] {} node [above] {$\ell^-_2$};

\draw [color=red] (9.25, 0) -- (10.75, 6) node [redpoint, midway, label=below
right:$\ell(t_3)$] {} node [above] {$\ell^-_3$};

\draw [color=red] (-2.5, 1) -- (-1.5, 5) node [redpoint, midway, label=below
right:$\ell(t_1)$] {} node [above] {$\ell^-_1$};

\draw [color=red] (11.5, 1) -- (12.5, 5) node [redpoint, midway, label=below
right:$\ell(t_4)$] {} node [above] {$\ell^-_4$};

\draw [thick] (-1.75, -0.5) -- (-0.5, 2.5) -- (10.5, 3.5) -- (11.75, 6.5)
node [label=right:$\gamma$] {};

\draw (5, 3) node [point, label=below:${\gamma(t) = \ell(0)}$] {};

\draw (7, 3.5) node {$\theta$};

\draw ([shift=(0:2cm)]5,3) arc (0:5:2cm);

\draw [thick, arrows=<->] (-0.75, -0.5) -- (9.25, -0.5) node [midway,
label=below:$\ge 2 T_\gamma - 2 T_2$] {};

\draw [thick, arrows=<->] (-2.5, -1.5) -- (11.5, -1.5) node [midway,
label=below:$\le 2 T_\gamma + 2 T_2$] {};

\draw [thick, arrows=<->] (14, 4) -- (14, 2) node
[label=below:$\ge \frac{A_\LL}{2 T_\gamma + 2 T_2}$] {};

%\draw [thick, arrows=<->] (15.5, 5) -- (15.5, 1) node
%[label=below:$\le \frac{A_\LL}{2 T_\gamma - 2 T_2}$] {};

\end{tikzpicture}
\end{center}
\caption{Small angles give long transverse rectangles.} \label{fig:angle rectangle}
\end{figure}
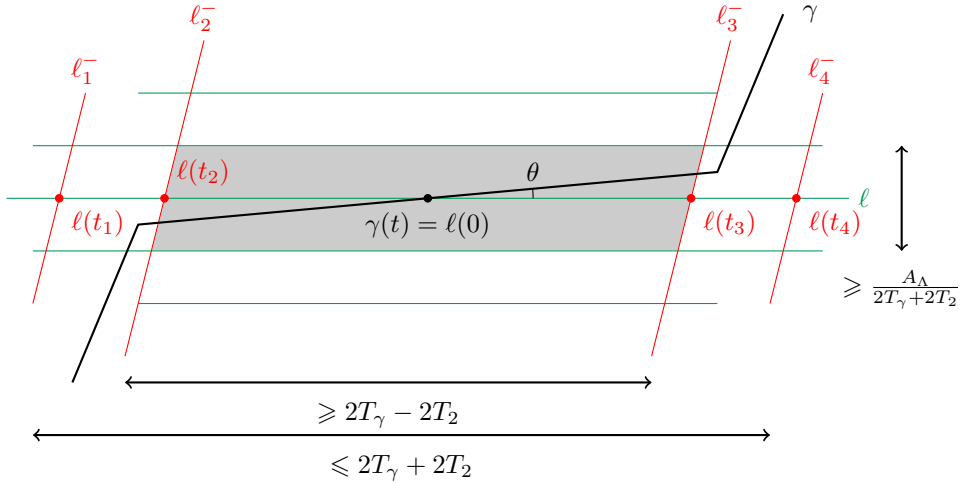

\medskip
The subinterval of $\ell$ between $\ell^-_2$ and $\ell^-_3$ will be
contained in the rectangle we construct.  We now verify that this
subinterval satisfies the bounds on hyperbolic length from
\Cref{prop:long side}.  Let $\rho_\theta$ be the minimum distance from
$\ell(0)$ to either of the leaves $\ell^-_2$ and $\ell^-_3$.  By our
choice of intervals above,
\begin{align*}
  \log \tfrac{1}{\theta} - T_1 - L_\LL & \le \rho_\theta \le \log
\tfrac{1}{\theta} + T_0 - T_1. \\
\intertext{Recall that $T_1 = T_0 + \log \tfrac{1}{\alpha_\LL}$, which
  gives}
  \log \tfrac{1}{\theta} - T_0 - \log \tfrac{1}{\alpha_\LL} - L_\LL & \le \rho_\theta \le \log
  \tfrac{1}{\theta} - \log \tfrac{1}{\alpha_\LL}.
\end{align*}  
In order to show the bounds from \Cref{prop:long side}, it therefore
suffices to show that
\begin{align}
\tfrac{1}{6} \log \tfrac{1}{\theta_\LL} - c_\LL & \ge T_0 + \log
\tfrac{1}{\alpha_\LL} + L_\LL. \nonumber \\
\intertext{Equivalently,}
\log \tfrac{1}{\theta_\LL} & \ge 6 ( T_0 + \log
\tfrac{1}{\alpha_\LL} + L_\LL + c_\LL ). \label{eq:theta pa estimate} \\
\intertext{
  and \eqref{eq:theta pa estimate} follows directly from our choice of
$\theta_\LL$ in \Cref{def:theta}, which in fact satisfies the stronger
estimate
}
%
%thetacondition  
% 
  \log \tfrac{1}{\theta_\LL} & \ge 6( T_0 + \log \tfrac{1}{\alpha_\LL}
                               + L_\LL + Q_\LL c_\LL
                               ), \label{eq:theta pa estimate 2}
\end{align}
as $Q_\LL \ge 1$.

\medskip
The bound on the hyperbolic length of the subinterval of $\ell$
between $\ell^-_2$ and $\ell^-_3$ immediately gives the bound on the
measure of the subinterval from \Cref{prop:long side measure}, using
the quasi-isometry between the hyperbolic and flat metric,
\Cref{prop:qi}.  In particular, as $\theta \le \theta_\LL$, the
measure of this subinterval of $\ell$ is at least
$\tfrac{2}{Q_\LL}(\log \tfrac{1}{\theta} - \tfrac{1}{6} \log
\tfrac{1}{\theta_\LL}) \ge \frac{5}{3 Q_\LL} \log
\tfrac{1}{\theta_\LL}$, which is positive.  Hence, by \Cref{prop:min
  area}, the leaves $\ell^-_2$ and $\ell^-_3$ bound a maximal
rectangle of area at least $A_\LL$ containing this subinterval of
$\ell$.

\medskip
By our choice of intervals, the distance between $t_2$ and $t_3$ is at
least $2( T_\gamma - T_1 - L_\LL)$.  Similarly, the distance between
$t_1$ and $t_4$ is at most $2( T_\gamma + T_1 + L_\LL)$.  For
notational convenience, set
\begin{equation}\label{eq:T2}
T_2 = T_1 + L_\LL.
\end{equation}

\medskip

Since the arc of $\ell$ between the intersection points with
$\ell^-_1$ and $\ell^-_4$ contains the arc of $\ell$ between
$\ell^-_2$ and $\ell^-_3$, by \Cref{prop:min area}, the leaves
$\ell^-_1$ and $\ell^-_4$ also bound a rectangle of area at least
$A_\LL$.  Note that the hyperbolic length of the arc of $\ell$ between
$\ell^-_1$ and $\ell^-_4$ is at most $2(T_\gamma + T_2)$.  So the
horizontal measure of the rectangle bounded by $\ell^-_1$ and
$\ell^-_4$ is at most $2 Q_\LL(T_\gamma + T_2) + c_\LL$.  For
notational convenience, set $T_3 = T_2 + \tfrac{1}{2} Q_\LL c_\LL$,
and observe that \eqref{eq:theta pa estimate 2} shows that
\begin{equation}\label{eq:tgt3}
\log \tfrac{1}{\theta_\LL} \ge 6 T_3 \ge 2 T_3,
\end{equation}
where we use the weaker lower bound to simplify constants in the
following calculations.  In particular,
$T_\gamma - T_3 \ge \log \tfrac{1}{\theta_\LL} - T_3 \ge T_3 > 0$ is
positive.

The vertical measure of the rectangle bounded by $\ell^-_1$ and
$\ell^-_4$ is therefore at least
$A_\LL / ( 2 Q_\LL ( T_\gamma + T_2 ) + c_\LL ) \ge A_\LL / ( 2 Q_\LL
( T_\gamma + T_3 ) )$, where the last inequality follows by our choice
of $T_3$, and the fact that $Q_\LL \ge 1$.

\medskip

The intersection of the two rectangles above is a transverse rectangle
for $\gamma$.  This is because, by construction, the leaves $\ell^-_2$
and $\ell^-_3$ intersect $\gamma$.  Furthermore, as $\ell^-_1$ and
$\ell^-_4$ are disjoint from $\gamma$, and lie on opposite sides of
$\gamma$, every leaf of $\Lambda_+$ which intersects both $\ell^-_1$
and $\ell^-_4$ also intersects $\gamma$.  The measure of the arc of
$\ell$ between $\ell^-_2$ and $\ell^-_3$ is therefore equal to the
measure of each side of the rectangle in $\Lambda_+$, and so
\Cref{prop:long side measure} holds.

\medskip
By the measure estimates above, the measure of the transverse rectangle is at least
\[ A_\LL \frac{ \tfrac{2}{Q_\LL} ( T_\gamma - T_3 ) }{ 2 Q_\LL (
  T_\gamma + T_3 ) } = \frac{A_\LL}{Q^2_\LL} \ \frac{ T_\gamma - T_3
}{ T_\gamma + T_3 }. \] Since $T_\gamma \ge 2 T_3$, by
\eqref{eq:tgt3}, it follows that the measure of the transverse
rectangle is at least $\tfrac{1}{3 Q_\LL^2} A_\LL$, so we may choose
$A_1 = \tfrac{1}{3 Q_\LL^2} A_\LL > 0$.  This completes the proof of
\Cref{prop:area bound}.  \Cref{prop:short side} is an immediate
consequence of \Cref{prop:long side} and \Cref{prop:area bound}.

\medskip

We now estimate the optimal height of the rectangle to show the final
statement \Cref{prop:small angle optimal height}.  Recall that the
optimal height $z$ is $\tfrac{1}{2} \log_k ( y/x )$, where $x$ and $y$
are the measures of the horizontal and vertical sides.  Therefore
\begin{align}\label{eq:rectangle optimal height}
\tfrac{1}{2} \log_k \frac{ \tfrac{2}{Q_\LL}(T_\gamma - T_3) }{
  \frac{A_\LL}{ \tfrac{2}{ Q_\LL } (T_\gamma - T_3) } } & \le z \le
\tfrac{1}{2} \log_k \frac{ 2 Q_\LL (T_\gamma + T_3) }{ \frac{A_\LL}{ 2
    Q_\LL (T_\gamma + T_3) } }, \\
\intertext{which we may rewrite as}
\label{eq:rectangle optimal height2}
\log_k \tfrac{2}{Q_\LL \sqrt{A_\LL}} ( T_\gamma - T_3 ) & \le z \le
\log_k 2 \tfrac{ Q_\LL }{ \sqrt{A_\LL} } ( T_\gamma + T_3 ).
\intertext{We will use the following elementary bounds: for $x \ge a$,
$x/2 \le x - a$; and for $a \ge 2, b \ge 2 $, $\log_k(a+b) \le \log_k a +
                                                        \log_k b$. This gives}
\log_k T_\gamma - \log_k Q_\LL \sqrt{A_\LL} & \le z \le \log_k T_\gamma +
\log_k 2 \tfrac{ Q_\LL }{ \sqrt{A_\LL} } + \log_k T_3 \nonumber \\
\intertext{Using $\log \tfrac{1}{\theta} \le T_\gamma \le \log \tfrac{1}{\theta}
+ T_0$ gives}
\label{eq:rectangle optimal height3}
\log_k \log \tfrac{1}{\theta} - \log_k Q_\LL \sqrt{A_\LL} & \le z \le
\log_k \log \tfrac{1}{\theta} + \log_k T_0 + \log_k 2 \tfrac{ Q_\LL }{
  \sqrt{A_\LL} } + \log_k T_3.
\end{align}

We may choose
$K = \max \{ \log_k Q_\LL \sqrt{A_\LL}, \log_k T_0 + \log_k 2 \tfrac{ Q_\LL
}{ \sqrt{A_\LL} } + \log_k T_3 \} $, as required.

\medskip
Finally, if $\gamma$ intersects a leaf of $\Lambda_-$ instead of
$\Lambda_+$, the bounds on the lengths of the horizontal and vertical
measures of the rectangle are swapped but otherwise the entire
argument above goes through.  The bounds on the optimal height from
\eqref{eq:rectangle optimal height} then become
\[ \tfrac{1}{2} \log_k \frac{ 2 Q_\LL (T_\gamma + T_3) }{ \frac{A_\LL}{
    2 Q_\LL (T_\gamma + T_3) } } \le z \le \tfrac{1}{2} \log_k \frac{
  \frac{A_\LL}{ \tfrac{2}{ Q_\LL } (T_\gamma - T_3) } }{
  \tfrac{2}{Q_\LL}(T_\gamma - T_3) }. \]
Multiplying the line above by $-1$ reverses the inequalities and takes
the reciprocals of the fractions inside the log which exactly gives \eqref{eq:rectangle optimal height2}, except with the $z$ replaced by
$-z$.  In particular, the bounds from \eqref{eq:rectangle optimal height3} hold for $-z$, as required.
\end{proof}

%%%%%%%%%%%%%%%%%%%%%%%%%%%%%%%%%%%%%%%%%%%%%%%%%%%%%%%%%%%%%%%%%%%%%%%
\subsubsection{Truncated rectangles for small angles}\label{section:truncated}
%%%%%%%%%%%%%%%%%%%%%%%%%%%%%%%%%%%%%%%%%%%%%%%%%%%%%%%%%%%%%%%%%%%%%%%

Suppose that $\gamma(t)$ is an intersection point with a sufficiently small angle of a non-exceptional geodesic with a leaf of an invariant lamination.  
By \Cref{lemma:bottleneck} and \Cref{prop:small angle rectangle}, the corresponding point $\tau_\gamma(t)$ is a bottleneck. However, in \Cref{prop:small angle rectangle}, the geodesic $\gamma$ may
intersect the boundaries of the initial and terminal quadrants
arbitrarily far from the outermost transverse rectangle $R$. Hence, there 
is no upper bound on the length of $\gamma([u, v])$. 
We now show how to truncate $R$ to achieve a \emph{tame} bottleneck, i.e. a bottleneck where
there is a bound on the length of the path $\tau_\gamma([u, v])$
between the initial and terminal quadrants.
The measure of the sides of the truncated rectangle will be a definite proportion of the measure of the sides of $R$.

\begin{definition}
\label{def:epsilon-nested}
We say a rectangle $R^0 \subset R$ is \emph{$\epsilon$-nested} inside
$R$, if the leaves containing the sides of $R^0$ divide $R$ into
subrectangles $R^i$, all of which have measures of sides bounded below
as follows:
\[ dx(R^i) \ge \epsilon \ dx(R) \text{ and } dy(R^i) \ge
\epsilon \ dy(R).  \]
\end{definition}

Obviously, \Cref{def:epsilon-nested} needs $\epsilon < 1/3$ and our eventual choice in \Cref{prop:truncated} is smaller.

\begin{definition}
\label{def:epsilon truncated}
We say a rectangle $R^0 \subset R$ is an \emph{$\epsilon$-truncated rectangle}, if
it has the following properties.
\begin{enumerate}
\item The rectangle $R^0$ has the same optimal height as $R$.
\item The rectangle $R^0$ is $\epsilon$-nested inside $R$.
\end{enumerate}
\end{definition}

By definition the measure of the truncated rectangle is bounded below
in terms of the measure of the original rectangle.  In particular,
$dx(R^0)dy(R^0) \ge \epsilon^2 dx(R) dy(R)$.  We will label the sides
of $R^0$ using the same convention as the sides for $R$, using
superscripts to distinguish them, i.e. the sides in $\Lambda_+$ are
$\alpha^0_+$ and $\beta^0_+$, and the sides in $\Lambda_-$ are
$\alpha^0_-$ and $\beta^0_-$.  This is illustrated in
\Cref{fig:truncated optimal}.

\begin{figure}[h]
\begin{center}
\begin{tikzpicture}[scale=0.5]

\tikzstyle{point}=[circle, draw, fill=black, inner sep=1pt]
\tikzstyle{redpoint}=[circle, draw, fill=red, inner sep=1pt]

\draw [color=white, fill=black!10] (0, 0) rectangle (8, 8);

\draw [color=white, fill=black!20] (2, 2) rectangle (6, 6);

\draw [color=white, fill=black!20] (-1.5, -1) rectangle (2, 2);

\draw [color=white, fill=black!20] (6, 6) rectangle (15.5, 9);

\draw [color=white, fill=black!30] (-1.5, -1) rectangle (0, 0);

\draw [color=white, fill=black!30] (8, 8) rectangle (15.5, 9);

\draw (1, 4) node {$R$};

\draw (4, 5) node {$R^0$};

\draw (12.5, 7) node {$R^{\beta}$};

\draw (-1, 1) node {$U^0$};

\draw (7, 7) node {$V^0$};

\draw (-1, -0.5) node {$U$};

\draw (9, 8.5) node {$V$};

\draw [color=ForestGreen] (-1.5, 0) node [label=left:${\alpha_+}$] {} -- (10, 0);
\draw [color=ForestGreen] (-1.5, 8) node [label=left:${\beta_+}$] {} -- (15.5, 8);

\draw [color=ForestGreen] (-1.5, 2) node [label=left:${\alpha^0_+}$] {} -- (10, 2);
\draw [color=ForestGreen] (-1.5, 6) node [label=left:${\beta^0_+}$] {} -- (15.5, 6);

\draw [color=red] (0, -1) -- (0, 9) node [label=above:{$\alpha_-$}] {};
\draw [color=red] (8, -1) -- (8, 9) node [label=above:{$\beta_-$}] {};

\draw [color=red] (2, -1) -- (2, 9) node [label=above:{$\alpha^0_-$}] {};
\draw [color=red] (6, -1) -- (6, 9) node [label=above:{$\beta^0_-$}]
{};

\begin{scope}[xshift=0.5cm]
\draw [color=red] (13, 5) -- (13, 9) node [label=above:{$\rho^1_-$}] {};

\draw [color=red] (13.5, 5) node [label=below:{$\rho^2_-$}] {} -- (13.5, 6.75) -- (15, 6.75);

\draw [color=red] (13.5, 9) -- (13.5, 7.25) -- (15, 7.25);

\draw [color=ForestGreen] (13, 6.25) -- (15, 6.25);

\draw [color=ForestGreen] (13, 7.75) -- (15, 7.75);

\draw [color=ForestGreen] (15, 6.5) -- (14, 6.5) -- (14, 7.5) -- (15, 7.5);
\end{scope}

\draw [thick] plot coordinates {(-0.5, -1) (0, -0.2) (1, 1) (2, 1.6) (3, 2.8) (4, 4.3)
  (5, 4.9) (6, 5.2) (7, 5.4) (8.5, 6) (8.75, 6.7) (9, 7.2) (10, 8) (11,
  8.3) (11.7, 9) };

\draw (0, -0.2) node [point, label=right:$\gamma(u)$] {};

\draw (2, 1.6) node [point, label=right:$\gamma(u^0)$] {};

\draw (8.5, 6) node [point, label=below right:$\gamma(v^0)$] {};

\draw (10, 8) node [point, label=below right:$\gamma(v)$] {};

\end{tikzpicture}
\end{center}
\caption{A truncated rectangle at optimal height.} \label{fig:truncated optimal}
\end{figure}
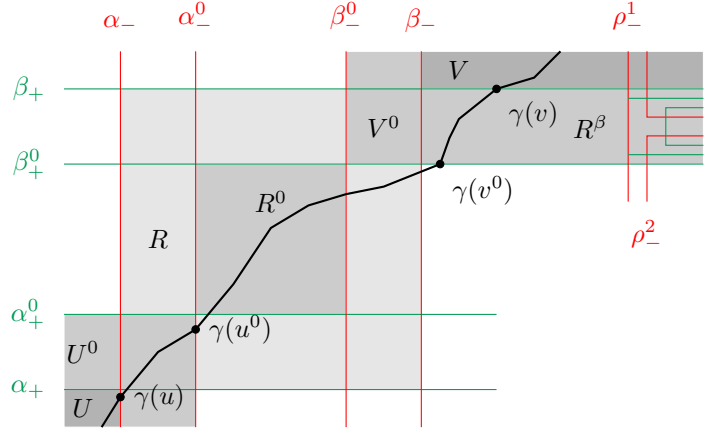

\begin{proposition}\label{prop:truncated}
Let $(S_h, \Lambda)$ be a choice of hyperbolic metric and suited pair
of a laminations.  Then there are constants $\theta_\LL > 0$ and
$\epsilon > 0$ such that for any point of intersection $\gamma(t)$
between a non-exceptional geodesic $\gamma$ and leaf
$\ell \in \Lambda_+$ of angle $\theta \le \theta_\LL$, and any
corresponding outermost rectangle $R$, there is an
$\epsilon$-truncated rectangle $R^0 \subset R$ such that the segment
of $\ell$ lying between $\alpha^0_-$ and $\beta^0_-$ is
$\ell([- r_1, r_2])$, where
$r_i = \tfrac{1}{2} \log \tfrac{1}{\theta_\LL}$, up to additive error
at most $\tfrac{1}{6} \log \tfrac{1}{\theta_\LL} - c_\LL$, where here
$\ell$ has a unit speed parametrization such that $\ell(0)$ is the
intersection point with $\gamma$.
\end{proposition}

\begin{proof}
Suppose that $\gamma(t)$ is an intersection point between $\gamma$ and
$\ell_+ \in \Lambda_+$ of angle $\theta \le \theta_\LL$, and let $R$
be the corresponding outermost rectangle.  We shall choose
$\epsilon = 1/(18 Q^2_\LL) > 0$, which only depends on
$(S_h, \Lambda)$.

\medskip
We now construct a smaller rectangle $R_0$ strictly contained inside
$R$.  We shall label the sides of $R$ and $R^0$ using the notation
from \Cref{fig:truncated optimal}.  We first choose the sides
$\alpha^0_-$ and $\beta^0_-$ of $R_0$ in $\Lambda_-$.  As before, give
$\ell$ a unit speed parametrization with the intersection point being
$\ell(0) = \gamma(t)$.  Let $a < a^0 < 0 < b^0 < b$ be the parameters
giving the intersections of $\ell$ with the sides of $R$ and $R^0$,
i.e.  $\ell(a) = \ell \cap \alpha_-$,
$\ell(a^0) = \ell \cap \alpha^0_-$, $\ell(b^0) = \ell \cap \beta^0_-$
and $\ell(b) = \ell \cap \beta_-$.

\medskip
Choose intervals $I_1 = \ell(-T - L_\LL, -T )$ and
$I_2 = \ell(T, T + L_\LL)$, where
$T = \tfrac{1}{2} \log \tfrac{1}{\theta}$, where $L_\LL$ is the
constant from \Cref{prop:cobounded intersections}.  Let $\alpha^0_-$
be the last leaf of intersection of $I_1$ with $\Lambda_-$, and let
$\beta^0_-$ be the first leaf of intersection of $I_2$ with
$\Lambda_-$.  Then the segment $\ell([ a^0, b^0 ] )$ between
$\alpha^0_-$ and $\beta^0_-$ is of the form $\ell(-r_1, r_2)$, where
$r_i = \tfrac{1}{2} \log \tfrac{1}{\theta}$, up to additive error
$L_\LL$.  The required bound on the additive error follows as
$L_\LL \le \tfrac{1}{6} \log \tfrac{1}{\theta_\LL} - c_\LL$, by our
choice of $\theta_\LL$ from \Cref{def:theta}.
%
%thetacondition
%

\medskip
We now verify that we may construct a rectangle $R^0 \subset R$ which
is $\epsilon$-truncated.  Let $\ell([a, a^0])$ be the segment of
$\ell$ between $\alpha_-$ and $\alpha^0_-$.  The hyperbolic length of
$\ell([a, a^0])$ is bounded below by
\[ \text{length}_{S_h}(\ell([a, a^0]) \ge \tfrac{1}{2} \log
\tfrac{1}{\theta} - \tfrac{1}{3} \log \tfrac{1}{\theta_\LL} + 2
c_\LL, \]
and so using the quasi-isometry between the hyperbolic and
Cannon--Thurston metrics, the measure of this segment is at least
\[ dy(\ell([a, a^0])) \ge \tfrac{1}{Q_\LL}( \tfrac{1}{2} \log
\tfrac{1}{\theta} - \tfrac{1}{3} \log \tfrac{1}{\theta_\LL} + c_\LL
), \]
which is positive as
$\theta \ge \theta_\LL$.  Using the upper bound on $dy(R)$ from
\Cref{prop:long side measure}, the ratio of the measure of this
segment and the measure of $dy(R)$ is then bounded by
\[ \frac{dy(\ell([a, a^0]))}{dy(R)} \ge \frac{ \tfrac{1}{Q_\LL}(
  \tfrac{1}{2} \log \tfrac{1}{\theta} - \tfrac{1}{3} \log
  \tfrac{1}{\theta_\LL} + c_\LL ) }{ 2 Q_\LL ( \log \tfrac{1}{\theta}
  + \tfrac{1}{2} \log \tfrac{1}{\theta_\LL} ) } \ge \frac{ 1 }{ 18
  Q^2_\LL} = \epsilon > 0, \]
where the right hand inequality follows as $Q_\LL \ge 1$.  The same
argument applies to the segment $\ell([b^0, b])$ between $\beta^0_-$
and $\beta_-$.

\medskip
Finally, a lower bound on the ratio $dy(R^0) / dy(R)$ is given by
\[ \frac{dy(R^0)}{dy(R)} \ge \frac{ \tfrac{1}{Q_\LL}( \log
  \tfrac{1}{\theta} - \tfrac{1}{3} \log \tfrac{1}{\theta_\LL} + c_\LL
  ) }{ 2 Q_\LL ( \log \tfrac{1}{\theta} + \tfrac{1}{2} \log
  \tfrac{1}{\theta_\LL} ) } \ge \frac{ 2 }{ 9 Q^2_\LL} = 4 \epsilon >
\epsilon > 0, \]
and as the measures of the three segments $\ell([a, a^0])$,
$\ell([a^0, b^0])$ and $\ell([b^0, b])$ sum to $dy(R)$, we obtain
\begin{equation}\label{eq:ratio bounds}
0 < 4 \epsilon \le  \frac{dy(R^0)}{dy(R)} \le 1 - 2\epsilon < 1.
\end{equation}

We may choose the other two sides $\alpha^0_+$ and $\beta^0_+$ of
$R_0$ such that the measure $dx(R^0) = dy(R^0) dx(R) / dy(R)$, so both
$R$ and $R^0$ have the same optimal height.  Furthermore, we may
choose the sides so that the leaf of $\Lambda_+$ that bisects $R$ in
terms of measure, also bisects $R^0$ in terms of measure.  In
particular, the bounds on the ratios of the measures of
$dy(R^0) / dy(R)$ from \eqref{eq:ratio bounds} applies to the ratio
of the measures $dx(R^0) / dx(R)$.  As $R^0$ is centered symmetrically
in $R$ with respect to the measure $dx$, each of the complementary
rectangles formed from the intersections of $\alpha^0_+$ and
$\beta^0_+$ with $R$ have sides with $dx$-measure ratio at least
$\epsilon$, as required.
\end{proof}

%%%%%%%%%%%%%%%%%%%%%%%%%%%%%%%%%%%%%%%%%%%%%%%%%%%%%%%%%%%%%%%%%%%%%%%
\subsubsection{Small angles create tame bottlenecks}%
\label{section:small angle bottleneck}
%%%%%%%%%%%%%%%%%%%%%%%%%%%%%%%%%%%%%%%%%%%%%%%%%%%%%%%%%%%%%%%%%%%%%%%

In this section, we show that the truncated rectangles give rise to tame bottlenecks.

\begin{corollary}\label{lemma:small angles}
Suppose that $(S_h, \Lambda)$ is a hyperbolic metric on $S$ together
with a suited pair of measured laminations.  Then there are
positive constants $\theta_\LL > 0$ (from \Cref{def:theta}), $r > 0$
and $K \ge 0$ such that for any non-exceptional geodesic $\gamma$ in
$S_h$ with unit speed parametrization $\gamma(t)$, if $\gamma$
intersects a leaf $\ell$ at $\gamma(t)$ with angle
$\theta \le \theta_\LL$, then there are parameters $u < t < v$ such
that the segment of the test path $\tau_\gamma([u, v])$ is an
$(r, K)$-bottleneck for the ladders over $\gamma((-\infty, u])$ and
$\gamma([v, \infty))$, and furthermore, the length of the test path
$\tau_\gamma([u, v])$ is at most $K$.
\end{corollary}

Except for the bound on the arc length of the segment of $\tau_\gamma$ between the bottleneck sets given by the initial and terminal
quadrants, all properties in \Cref{lemma:small angles} follow from \Cref{prop:truncated} and \Cref{prop:small angle rectangle}. We derive the arc length bound now. 

\medskip
A truncated rectangle at optimal height is shown in
\Cref{fig:truncated optimal}.  We shall write $U^0$ and $V^0$ for the
initial and terminal quadrants of $R^0$. Since 
$U \subset U^0$ and $V \subset V^0$, it follows that if $R$ is a
transverse rectangle for $\gamma$, then the truncated rectangle $R^0$
is also a transverse rectangle for $\gamma$.  We shall write
$\gamma(u^0)$ and $\gamma(v^0)$ for the points at which $\gamma$
intersects the boundaries of the initial and terminal quadrants $U^0$
and $V^0$.

\medskip
To prove \Cref{lemma:small angles}, it therefore suffices to
prove that the segment $\gamma_\tau([u^0, v^0])$ has bounded length.

\begin{proposition}\label{prop:truncated path bounded length}
Suppose that $(S_h, \Lambda)$ is a hyperbolic metric on $S$ together
with a suited pair of measured laminations.  There is a constant
$\theta_\LL$ such that for any $0 < \epsilon < 1$ there is a constant
$K$ such that for any $\epsilon$-truncated outermost rectangle $R^0$,
determined by an intersection point of $\gamma$ with a leaf of an
invariant lamination $\ell$ of angle at most $\theta_\LL$, the arc
length of $\tau_\gamma([u^0, v^0])$ is at most $K$.
\end{proposition}

We first verify \Cref{lemma:small angles} from
\Cref{prop:truncated path bounded length}.

\begin{proof}[Proof (of \Cref{lemma:small angles})]
By \Cref{prop:small angle rectangle}, there are constants $\theta_\LL$
and $A > 0$ such that if $\gamma(t)$ is an intersection point of
$\gamma$ with $\ell$ of angle $\theta \le \theta_\LL$, then there is
an outermost rectangle $R$ containing $\gamma(t)$ of area at least
$A$.
By \Cref{prop:truncated}, there is an $\epsilon > 0$ such that the
rectangle $R$ contains an $\epsilon$-truncated outermost rectangle
$R_0$ of area at least $\epsilon^2 A$ which is transverse to $\gamma$,
giving an $(r, K_1)$-bottleneck.
By \Cref{prop:truncated path bounded length}, the segment of the test
path $\tau_\gamma(u^0, v^0)$ between the bottleneck sets for $R^0$ has
length at most $K_2$, where $K_2$ depends only on $(S_h, \Lambda)$.
The result then follows for $K = K_1 + K_2$.
\end{proof}

The rest of this section contains the proof of
\Cref{prop:truncated path bounded length}, which has the
following two steps.

\begin{enumerate}
\item We consider $\gamma_z([u^0, v^0])$, the image of
$\gamma([u^0,v^0])$ at the optimal height $z$ for the rectangle $R^0$, and
show that the length of this path is bounded.
\item  We show that the height function varies by a bounded amount
over $\gamma([u^0, v^0])$, so projecting $\tau_\gamma([u^0, v^0])$ to
$\gamma_z([u^0, v^0])$ increases length by a bounded factor.
\end{enumerate}

The graph of a monotonic function on the unit interval
$f \colon I \to I$ has bounded path length.  We will use the
following analog of monotonicity for real valued functions obtained from
paths on surfaces with a suited pair of measured
laminations.

\begin{definition}
\label{def:monotonic}
Let $\widetilde{S}_h$ be the universal cover of a closed hyperbolic
surface, and let $\Lambda_+$ and $\Lambda_-$ be the lifts of a suited pair of
 measured laminations to $\widetilde{S}_h$.  Let
$\ell_+ \in \Lambda_+$ and $\ell_- \in \Lambda_-$ be two leaves which
intersect at a point $c$.  Let $\gamma$ be a path from a point $a$ on
$\ell_+$ to $b$ on $\ell_-$ such that the interior of $\gamma$ is
disjoint from the two leaves.  We say the path $\gamma$ is
\emph{monotonic} if
\begin{itemize}
\item Every leaf of $\Lambda_+$ which intersects $\ell_-$ between $c$
and $b$ intersects $\gamma$ in exactly one point, and $\gamma$ intersects no
other leaves of $\Lambda_+$. 
\item Every leaf of $\Lambda_-$ which intersects $\ell_+$ between $c$
and $a$ intersects $\gamma$ in exactly one point, and $\gamma$ intersects no
other leaves of $\Lambda_-$. 
\end{itemize}
\end{definition}

As the above properties are preserved by the vertical flow, a geodesic
segment $\gamma$ in $\ws_h = S_0$ is monotonic with respect to
$\ell_+$ and $\ell^-$ if an only if for all $z$, the path
$F_z(\gamma)$ in $S_z$ is monotonic with respect to $F_z(\ell_+)$ and
$F_z(\ell_-)$.
We now show that any geodesic segment in $\ws_h$ with endpoints on
intersecting leaves is monotonic.

\begin{proposition}\label{prop:geodesic is monotonic}
Suppose that $(S_h, \Lambda)$ is a hyperbolic metric on $S$ together
with a suited pair of measured laminations.  Let $\gamma$ be a
segment of a geodesic in $\ws_h$ with endpoints on two intersecting
leaves $\ell_+$ and $\ell_-$.  Then $\gamma$ is monotonic with respect
to $\ell_+$ and $\ell_-$.
\end{proposition}

\begin{proof}
Suppose that $c$ is a point of intersection of leaves $\ell_+$ and $\ell_-$.
Let $a$ be the endpoint of
$\gamma$ on $\ell_+$ and $b$ the endpoint of $\gamma$ on
$\ell_-$.  For the geodesic triangle with vertices $a, b$ and
$c$, let $\alpha$ be its side along $\ell_+$ and $\beta$ its side along $\ell_-$. Any leaf of $\Lambda_+$ that intersects $\beta$ crosses the triangle. Since the leaf cannot intersect $\alpha$ it must intersect $\gamma$ exactly once. Similarly, any leaf of
$\Lambda_-$ which intersects $\alpha$ intersects $\gamma$ exactly once.  Finally, a leaf of either
lamination which intersects $\gamma$, may only intersect it once, and
so must intersect the side contained in the other lamination.
\end{proof}

We now show that the arc length of a monotonic path is bounded by the
lengths of the other two sides of the triangle that it forms with the
leaves of intersection.

\begin{proposition}\label{prop:monotonic length bound}
Suppose that $(S_h, \Lambda)$ is a hyperbolic metric on $S$ together
with a suited pair of measured laminations.  Let $\gamma$ be a
monotonic path in $\ws_h$ with respect to the two intersecting leaves
$\ell_+$ and $\ell_-$.  Let $a$ and $b$ be the endpoints of $\gamma$,
and let $c$ be the intersection point of the two leaves.  Then for any
$z$,
\[ \textrm{length}(\gamma_z) \le d_{S_z}(a, c) + d_{S_z}(c, b),  \]
where $d_{S_z}$ is the Cannon--Thurston pseudometric on $S_z$, and
$length(\gamma_z)$ is the arc length of $\gamma_z$ in this metric.
\end{proposition}

\begin{proof}
In the triangle with vertices $a$, $b$, and $c$ we denote the other two sides as
$\alpha = [a, c]$ and $\beta = [b, c]$.  
By definition, the arc length of $\gamma_z$ is
\begin{equation}\label{eq:rectifiable}
\text{length}(\gamma_z) = \lim_{|P| \to 0} \sum_{i = 0}^n d_{S_z}(x_i,
x_{i+1}),
\end{equation}
where the limit is taken over partitions $P = \{ a = x_0 < x_1 < \cdots < x_{n+1} = b\}$ with  $|P| = \max \{ x_{i+1} - x_i \}$. 

\medskip
By \Cref{def:monotonic}, the union of intervals $\gamma_z \cap R_z$ over all regions $R$ complementary to both laminations is open and dense in $\gamma_z$. Hence, for any partition $P = \{ a = x_0 < x_1 < \cdots < x_{n+1} = b\}$ there exists a partition $P' = \{ a = x'_0 < x'_1 < \cdots < x'_{n+1} = b\}$ such that $|P'| < 2 |P|$, and each $x'_i$ lies in a region $R_i$ complementary to both laminations. In particular, $|P'| \to 0$ as $|P| \to 0$.

\medskip
By \Cref{def:monotonic} again, the intersection $\gamma_z \cap R_i$ cuts off a single vertex of $R_i$. 
Since the pseudometric distances between any pair of points in $R_i$ is zero, we may replace $x'_i$ by this vertex without changing each term in the sum in \Cref{eq:rectifiable} for the partition $P'$. 
In particular, we may now assume that each $x'_i$ is an intersection point of the two invariant laminations.

\medskip
Up to reversing the orientation of $\gamma$, we may assume that the
initial point of $\gamma$ lies on $\ell_+$ and the terminal point of
$\gamma$ lies on $\ell_-$.  Consider a pair of adjacent points $x'_i$
and $x'_{i+1}$, and let $\ell^i_+$ be the leaf of $\Lambda_+$ through
$x'_i$, and let $\ell^{i+1}_-$ be the leaf of $\Lambda_-$ through
$x'_{i+1}$.

\medskip
By monotonicity, $\ell^i_+$ intersects both $\gamma$ and $\ell_-$
exactly once, and separates $x'_{i+1}$ from $\ell^+$.  Therefore the
pair of leaves $\ell^i_+$ and $\ell^{i+1}_-$ intersect.  Let $c'_i$ be
the point of intersection.  By the triangle inequality in the
Cannon--Thurston metric,
$d_{S_z}(x'_i, x'_{i+1}) \le d_{S_z}(x'_i, c'_i) + d_{S_z}(c'_i, x'_{i+1})$.
The Cannon--Thurston distance along a leaf is equal to the measure of
the leaf with respect to the other invariant lamination, so
\[ \text{length}(\gamma_z) \le \lim_{|P| \to 0} \sum_{i = 0}^n \left(
d_{S_z}(x'_i, c'_i) + d_{S_z}(c'_i, x'_{i+1}) \right) = d_{S_z}(a, c) +
d_{S_z}(c, b), \]
as required.
\end{proof}

We now complete Step 1 by showing that the length of
$\gamma_z([u^0, v^0])$ in $S_z$ is bounded.

\begin{proposition}\label{prop:optimal height bounded length}
Suppose that $(S_h, \Lambda)$ is a hyperbolic metric on $S$ together
with a suited pair of measured laminations.  Then there is a
constant $\theta_\LL$ such that for any $0 < \epsilon < 1$ there is a
constant $K$ such that for any $\epsilon$-truncated outermost
rectangle $R^0$, with optimal height $z$, determined by an
intersection point of $\gamma$ with a leaf $\ell$ of one of the
invariant laminations of angle at most $\theta_\LL$, the length of
$\gamma_z([u^0, v^0])$ is at most $K$.
\end{proposition}

\begin{proof}
Up to switching the laminations, we may assume that
$\ell = \ell_+ \in \Lambda_+$.  By \Cref{def:epsilon truncated}, the outermost
transverse rectangle $R$ and the $\epsilon$-truncated rectangle $R^0$ have the same optimal height, which we shall denote by 
$z$.  As $\gamma$ is a geodesic in $S_h$, by \Cref{prop:geodesic is monotonic}, it is monotonic with respect to any two intersecting
sides of the rectangle $R$.  At height $z$, the measures of
the sides of $R_z$ are equal. Since the measure of any rectangle is bounded above by $A_{\max}$, the measure of each side
of $R_z$ is at most $\sqrt{A_{\max}}$.  Therefore, by
\Cref{prop:monotonic length bound}, it suffices to show that the
distances in $S_z$ from $\gamma_z(u^0)$ to $R^0_z$ and from
$\gamma_z(v^0)$ to $R^0_z$ are bounded.

\medskip
We will follow the notation illustrated in \Cref{fig:truncated
  optimal}.  Each endpoint of $\gamma([u^0, v^0])$ may lie
in either invariant lamination; for definiteness we have drawn
in \Cref{fig:truncated optimal} the case where $\gamma(u^0)$ lies in $\Lambda_-$ and $\gamma(v^0)$ lies in
$\Lambda_+$.

\medskip
Suppose that $\gamma_z (v^0)$ lies in $\Lambda_+$, as in \Cref{fig:truncated optimal}. We will derive a bound for the distance of $\gamma_z (v^0)$ from $R^0_z$ by using properties of the sides in $\Lambda_+$ of the $(\beta_+, \beta^0_+)$-maximal rectangle $R^{\beta}$. Exactly the same argument works for
when $\gamma_z (v^0)$ in $\Lambda_-$, in which case we use the sides in $\Lambda_-$ of the $(\beta_-, \beta^0_-)$-maximal
rectangle. 

\begin{figure}[h]
\begin{center}
\begin{tikzpicture}[scale=1.25]

\tikzstyle{point}=[circle, draw, fill=black, inner sep=1pt]
\tikzstyle{redpoint}=[circle, draw, fill=red, inner sep=1pt]

\draw [color=white, fill=black!20] (7, 6) rectangle (13.5, 8);

\draw (11.5, 7) node {$R^{\beta}$};

\draw (14.5, 7.25) node {$P$};

%\draw (7, 7) node {$V^0$};

%\draw (9, 8.5) node {$V$};

\draw [color=ForestGreen] (7, 8) node [label=left:${\beta_+}$] {} -- (16.5, 8);

\draw [color=ForestGreen] (7, 6) node [label=left:${\beta^0_+}$] {} -- (16.5, 6);

\draw [color=red] (8, 5) -- (8, 9) node [label=above:{$\beta_-$}] {};

%\draw [color=red] (6, 5) -- (6, 9) node [label=above:{$\beta^0_-$}]
%{};

\begin{scope}[xshift=0.5cm]
\draw [color=red] (13, 5) -- (13, 9) node [label=above:{$\rho^1_-$}] {};

\draw [very thick, color=red] (13, 6.25) -- node [midway,
label=right:{$s^1_-$}] {} (13, 7.75);

\begin{scope}[xshift=1cm]
\draw [color=red] (13.5, 5) node [label=below:{$\rho^2_-$}] {} -- (13.5, 6.75) -- (15, 6.75);

\draw [very thick, color=red] (13.5, 6.25) -- (13.5, 6.75) node
[label=left:${s^2_-}$] {} -- (14, 6.75);

\draw [color=red] (13.5, 9) -- (13.5, 7.25) -- (15, 7.25);

\draw [color=ForestGreen] (15, 6.5) -- (14, 6.5) -- (14, 7.5) -- (15, 7.5);
\end{scope}

\draw [color=ForestGreen] (13, 6.25) -- (16, 6.25);

\draw [color=ForestGreen] (13, 7.75) -- (16, 7.75);
\end{scope}

\draw [thick] plot coordinates {(7, 5.4) (8.5, 6) (8.75, 6.7) (9, 7.2) (10, 8) (11,
  8.3) (11.7, 9) };

\draw (8.5, 6) node [point, label=below right:$\gamma(v^0)$] {};

\draw (10, 8) node [point, label=below right:$\gamma(v)$] {};

\end{tikzpicture}
\end{center}
\caption{One side of the $(\beta_+, \beta^-_+)$-maximal rectangle $R^{\beta}$ from
  \Cref{fig:truncated optimal}.} \label{fig:one side of Rbeta}
\end{figure}
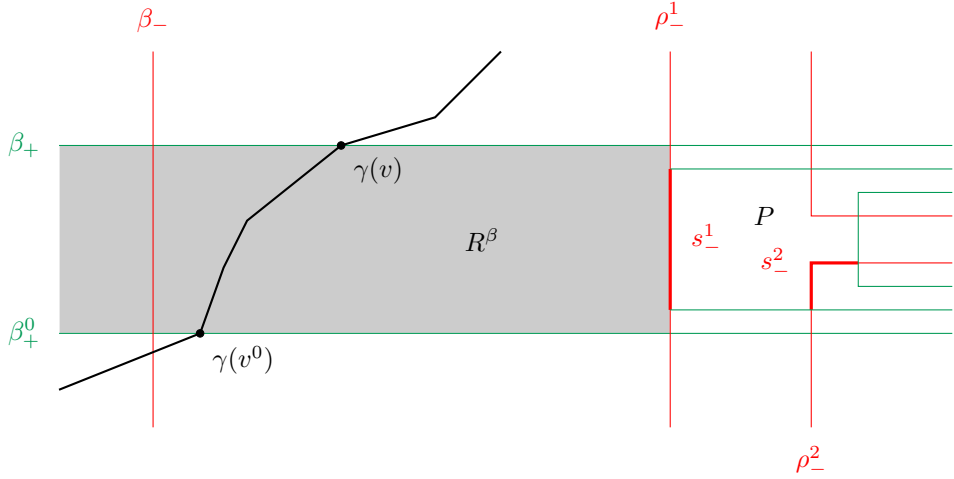

\medskip Let $\rho^1_-$ be the side of $R^{\beta}$ separated from
$\alpha_-$ by $\beta_-$. By maximality of $R^{\beta}$, $\rho^1_-$
contains a side $s^1_-$ of a non-rectangular polygon $P$.
This is illustrated in \Cref{fig:one side of Rbeta} above.  Note that
that shading in \Cref{fig:one side of Rbeta} is different from the
shading in \Cref{fig:truncated optimal};  in \Cref{fig:one side of
  Rbeta} we have only shaded the rectangle $R^{\beta}$.
In a cyclic
order on the sides in $\Lambda_-$ of the non-rectangular polygon $P$, let
$\rho^2_-$ be the leaf of $\Lambda_-$ such that $\rho^2_-$ contains a
side $s^2_-$ adjacent to $s^1_-$, and $\rho^2_-$ intersects $\beta_+^0$.
Suppose that
$\gamma_z$ intersects $\rho^2_-$.  Then $\gamma$ is disjoint from the
terminal quadrant $V$, a contradiction.  It follows that along
$\beta^0_+$, the intersection point $\gamma_z (v^0)$ is before the
intersection with $\rho^2_-$.

\medskip
The measure of the segment of $\beta^0_+$ from $\rho^1_-$ to
$\rho^2_-$ is zero, so the distance in $S_z$ from $R^0_z$ to
$\gamma_z(v^0)$ is at most the horizontal measure of $R^{\beta}$.  There is
an upper bound $A_{\max}$ on the measure of any rectangle, so
\begin{equation}\label{eq:r1 area}
dx(R^{\beta}_z) dy(R^{\beta}_z) \le A_{\max}.
\end{equation}
As $R^0$ is $\epsilon$-truncated with respect to $R$,
$dy(R^1) \ge \epsilon dy(R_z)$. Since $dy(R^{\beta}) = dy (R^1)$, we have $dy(R^{\beta}) \ge \epsilon dy(R_z)$.
At the optimal height
$dx(R_z) = dy(R_z)$, and by \Cref{prop:min area}, the area of $R_z$ is at least $A_\LL$. Hence $dy(R_z)^2 \ge A_\LL$. It follows that
\begin{equation}\label{eq:r1 side}
dy(R^{\beta}_z) \ge \epsilon \sqrt{A_\LL}.
\end{equation}
Combining \eqref{eq:r1 area} and \eqref{eq:r1 side} gives the
following upper bound on the measure of the other side of $R^{(1)}_z$,
\[ dx(R^{\beta}_z) \le \frac{A_{\max}}{\epsilon \sqrt{A_\LL}}. \]
For convenience, set $K_1 = A_{\max} / (\epsilon \sqrt{A_\LL})$, which
only depends on $(S_h, \Lambda), \theta_\LL$ and $\epsilon$.

\medskip
As $\alpha^0_-$ also intersects $R^{\beta}$, this gives the same bound on
the distance from $\gamma_z(v^0)$ to $\alpha^0_-$, i.e.
\[ d_{S_z}(\alpha^0_-, \gamma_z(v^0)) \le K_1.  \]
By reversing the orientation on $\gamma$, the same bounds holds for
the distance in $S_z$ from $R^0$ to $\gamma_z(u^0)$, i.e.
\[ d_{S_z}(\beta^0_+, \gamma_z(u^0)) \le K_1.  \]
As $\gamma_z([u^0, v^0])$ is monotonic in $S_0$, it is also monotonic
in $S_z$, so by \Cref{prop:monotonic length bound}, the arc length of
$\gamma_z([u^0, v^0])$ is at most
\[ \text{length}(\gamma_z([u^0, v^0])) \le d_{S_z}(\alpha^0_-,
\gamma_z(v^0)) + d_{S_z}(\beta^0_+, \gamma_z(u^0)) \le 2K_1 , \]
which only depends on $(S_h, \Lambda), \theta_\LL$ and $\epsilon$, as
required.
\end{proof}

We now finish Step 2 by showing that the height function changes by
a bounded amount along $\gamma([u^0, v^0])$.

\begin{proposition}\label{prop:height change main}
Suppose that $(S_h, \Lambda)$ is a hyperbolic metric on $S$ together
with a suited pair of measured laminations.  There is a constant
$\theta_\LL$ such that for any $0 < \epsilon < 1$ there is a constant
$K$ such that for any $\epsilon$-truncated outermost rectangle $R^0$,
with optimal height $z$, determined by an intersection point of
$\gamma$ with a leaf $\ell$ of an invariant lamination of angle
$\theta \le \theta_\LL$, for any $t \in [u^0, v^0]$, the value of the
height function $h_\gamma(t)$ is equal to
$\log_k \log \tfrac{1}{\theta}$ up to additive error at most $K$.
\end{proposition}

Up to swapping the laminations, we may assume the leaf of intersection
$\ell = \ell_+$ lies in $\Lambda_+$.

\medskip
We prove \Cref{prop:height change main} in two
parts. In \Cref{prop:height change one} we bound the change in the
height function along the segment of $\gamma_z$ lying between the
sides $\alpha^0_-$ and $\beta^0_-$ of $R^0$.
Note that if both endpoints of $\gamma([u^0, v^0])$ lie in $\Lambda_-$ then $\gamma([u^0, v^0])$ equals this segment and hence \Cref{prop:height change main} follows. 

\medskip
So we may suppose that one of the endpoints of $\gamma([u^0, v^0])$ is in $\Lambda_+$.
In this case, we use a  different set of estimates in $S_h$ to bound the change in the height
function in \Cref{prop:height change
  two}. This covers all cases and so
\Cref{prop:height change main} is an immediate consequence.
The constants labeled $K$ constructed
in each of \Cref{prop:height change one} and \Cref{prop:height change
  two} may be different, but we can just take their maximum for
\Cref{prop:height change main}.

\begin{proposition}\label{prop:height change one}
Suppose that $(S_h, \Lambda)$ is a hyperbolic metric on $S$ together
with a suited pair of measured laminations.  There is a constant
$\theta_\LL$ such that for any $0 < \epsilon < 1$ there is a constant
$K$ such that for any $\epsilon$-truncated outermost rectangle $R^0$,
with optimal height $z$, determined by an intersection point of
$\gamma$ with a leaf $\ell$ of an invariant lamination of angle
$\theta \le \theta_\LL$, for any $t$ such that $\gamma(t)$ lies
between the sides $\alpha^0_-$ and $\beta^0_-$ of $R^0$ which lie in
$\Lambda_-$, the value of the height function $h_\gamma(t)$ is equal
to $\log_k \log \tfrac{1}{\theta}$ up to additive error at most $K$.
\end{proposition}

\begin{proof}
We use the notation in \Cref{fig:truncated optimal}.  Up to
swapping the laminations, we may assume that $\gamma$ makes angle at
most $\theta_\LL$ at an intersection point $\gamma(t)$ with a leaf
$\ell_+ \in \Lambda_+$. Let $R^0$ be the corresponding
$\epsilon$-truncated outermost rectangle.  We denote the segment of $\gamma$ between the sides $\alpha^0_-$ and $\beta^0_-$ by $\gamma([a, b])$, as illustrated in
\Cref{fig:gamma truncated}. We will estimate the
change in the height function along $\gamma([a, b])$. 

\begin{figure}[h]
\begin{center}
\begin{tikzpicture}[scale=0.7]

\tikzstyle{point}=[circle, draw, fill=black, inner sep=1pt]
\tikzstyle{redpoint}=[circle, draw, fill=red, inner sep=1pt]
\tikzstyle{greenpoint}=[circle, draw, fill=ForestGreen, inner sep=1pt]

\draw [color=ForestGreen] (-3, 3) -- (13, 3) node [right] {$\ell$};

\draw [color=red] (-0.75, 0.5) -- (0.75, 5.5) node [redpoint, midway, label=above
right:$\ell(a^0)$] {} node [above] {$\alpha^0_-$};

\draw [color=red] (9.25, 0.5) -- (10.75, 5.5) node [redpoint, midway, label=below
right:$\ell(b^0)$] {} node [above] {$\beta^0_-$};

\draw [thick] (-1.75, 0.5) -- (-0.5, 2.5) -- (10.5, 3.5) -- (11.75, 5.5)
node [label=right:$\gamma$] {};

\draw (5, 3) node [point, label=below:${\gamma(t) = \ell(0)}$] {};

\draw (7, 3.5) node {$\theta$};

\draw ([shift=(0:2cm)]5,3) arc (0:5:2cm);

\draw [thick, arrows=<->] (-0.75, -0.5) -- (9.25, -0.5) node [midway,
label=below:$\sim \log \tfrac{1}{\theta_\LL}$] {};

\draw (-0.11, 2.55) node [point, label=below right:$\gamma(a)$] {};

\draw (10.125, 3.475) node [point, label=above left:$\gamma(b)$] {};

\draw [color=ForestGreen] (-1, 3) node [greenpoint, label=above:$p$] {};

\draw (-1, 1.7) node [point, label=left:$q$] {};

\end{tikzpicture}
\end{center}
\caption{Estimating the length of $\gamma$ between the sides of a
  truncated rectangle.} \label{fig:gamma truncated}
\end{figure}
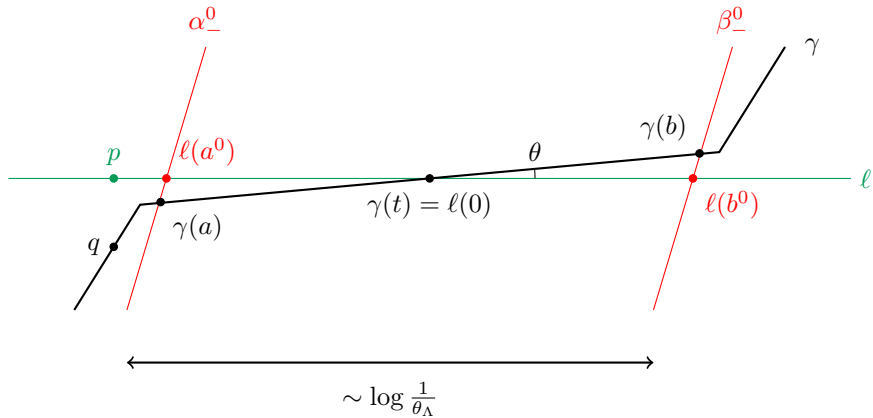

\medskip
By \Cref{prop:radius estimate}, the value of the radius function at
the intersection point $\gamma(t)$ is roughly equal to
$\log \tfrac{1}{\theta}$.  More precisely, using \eqref{eq:radius
  estimate} from \Cref{prop:radius estimate},
\[ \rho_{\gamma, \ell}(t) \le \rho_{\gamma, \Lambda_+}(t) \le
\rho_{\gamma, \ell}(t) + K_1, \]
where $K_1$ is the constant from \Cref{prop:radius estimate}, which
depends only on $(S_h, \Lambda)$.  Using the definition of the
radius function \eqref{eq:radius leaf}, gives
\[ \log \tfrac{1}{\theta} \le \rho_{\gamma, \Lambda_+}(t) \le \log
\tfrac{1}{\theta} + K_1. \]

We now find an upper bound for the hyperbolic length of the segment $\gamma[a,t]$.  The same argument will give an upper
bound for the hyperbolic length of $\gamma([t, b])$.  Let $\ell(a^0)$
be the point of intersection between $\ell$ and $\alpha^0_-$, and let
$\ell(b^0)$ be the point of intersection between $\ell$ and
$\beta^0_-$.  By \Cref{prop:truncated}, then length of
$\ell([a^0, 0])$ is bounded above by
\begin{equation}\label{eq:a0 bound}
\text{length}_{S_h}(\ell([a^0, 0])) \le \tfrac{1}{2} \log
\tfrac{1}{\theta} + \tfrac{1}{6} \log \tfrac{1}{\theta_\LL} - c_\LL
.
\end{equation}
By the triangle inequality, 
\[ d_{S_h}(\gamma(a), \gamma(t)) \le d_{S_h}(\gamma(a), \ell(a^0)) +
d_{S_h}(\ell(a^0), \ell(0) ) .  \]

We now find an upper bound for the hyperbolic distance from
$\gamma(a)$ to $\ell(a^0)$.  Let $p$ be a point on $\ell$ distance
$2 \rho_\LL$ from $\ell(a^0)$, such that $\ell(a^0)$ separates $p$
from $\ell(0)$, and let $q$ be the closest point on $\gamma$ to $p$.

\medskip
By \Cref{prop:projection interval} the distance from $p$ to $\gamma$
is at most $\theta e^{t}$, where
$t \le \tfrac{1}{2} \log \tfrac{1}{\theta} + \tfrac{1}{4} \log
\tfrac{1}{\theta_\LL} - c_\LL + 2 \rho_\LL$, so
\[ d_{S_h}(p, q) \le \theta e^{ \tfrac{1}{2} \log \tfrac{1}{\theta} + \tfrac{1}{4} \log
\tfrac{1}{\theta_\LL} - c_\LL + 2 \rho_\LL  },  \]
which simplifies to
\[ d_{S_h}(p, q) \le \theta^{\tfrac{1}{2}} \theta_\LL^{-\tfrac{1}{4}}
e^{-c_\LL + 2 \rho_\LL}, \]
and as $\theta \le \theta_\LL$,
\[ d_{S_h}(p, q) \le \theta_\LL^{\tfrac{1}{4}} e^{-c_\LL + 2 \rho_\LL}.  \]
%
%thetacondition
%
By our choice of $\theta_\LL$,
\begin{equation}\label{eq:pq}
d_{S_h}(p, q) \le \tfrac{1}{4} \rho_\LL,
\end{equation}
and so the nearest point projection of $q$ to $\ell$ lies within distance
$\tfrac{1}{2} \rho_\LL$ of $p$. In particular, it lies outside the
nearest point projection interval of $\alpha^0_-$ to $\ell$.  This
means that $q$ lies past $\gamma(a)$ from $\gamma(t)$, and so the
distance from $\gamma(t)$ to $q$ is an upper bound on the distance
from $\gamma(t)$ to $\gamma(a)$, i.e.
\[ d_{S_h}(\gamma(a), \gamma(t)) \le d_{S_h}(q , \gamma(t)),  \]
and then by the triangle inequality,
\[ d_{S_h}(\gamma(a), \gamma(t)) \le d_{S_h}(q , p) + d_{S_h}(p , \gamma(t)).  \]
Using \eqref{eq:pq}, and the fact that $\gamma(t) = \ell(0)$, and
$\ell(a^0)$ lies between $p$ and $\ell(0)$, gives 
\[ d_{S_h}(\gamma(a), \gamma(t)) \le \tfrac{1}{4} \rho_\LL + d_{S_h}(p
, \ell(a^0) ) + d_{S_h}( \ell(a^0) , \ell(0) ).  \]
By our choice of $p$, the distance from $p$ to $\ell(a^0)$ is $2
\rho_\LL$. Using the upper bound from \eqref{eq:a0 bound}, we get 
\[ d_{S_h}(\gamma(a), \gamma(t)) \le \tfrac{1}{2} \log
\tfrac{1}{\theta} + \tfrac{1}{6} \log \tfrac{1}{\theta_\LL} - c_\LL +
\tfrac{9}{4} \rho_\LL.  \]
%
%thetacondition
%
Our choice of $\theta_\LL$ ensures that
$\log \tfrac{1}{\theta_pa} \ge \tfrac{9}{4} \rho_\LL$, so we simplify
the inequality above to
\[ d_{S_h}(\gamma(a), \gamma(t)) \le \tfrac{1}{2} \log
\tfrac{1}{\theta} + \tfrac{1}{3} \log \tfrac{1}{\theta_\LL} .  \]
Similarly, we get exactly the same bound for
$d_{S_h}(\gamma(b), \gamma(t))$.

\medskip
As the radius function is $1$-Lipschitz with respect to the hyperbolic
metric, for any $t \in [a, b]$ the value of the radius function at
$\gamma(t)$ is at least
\begin{equation}\label{eq:radius lower}
\rho_{\gamma, \Lambda_+}(t) \ge \tfrac{1}{2} \log
\tfrac{1}{\theta} - \tfrac{1}{3} \log \tfrac{1}{\theta_\LL} \ge
\tfrac{1}{6} \log \tfrac{1}{\theta_\LL} > 0,
\end{equation}
and the value of the radius function at $t$ is at most
\begin{equation}\label{eq:radius upper}
\rho_{\gamma, \Lambda_+}(t) \le \tfrac{3}{2} \log
\tfrac{1}{\theta} + \tfrac{1}{3} \log \tfrac{1}{\theta_\LL} + K_1.
\end{equation}

As the height function is $\log_k$ of a floor function of the radius
function, the change in the height function is therefore bounded by the
logarithm base $k$ of the ratio between the values of the height
function along $\gamma([a, b])$.  In particular, using
\eqref{eq:radius lower} and \eqref{eq:radius upper} the change in
height function along $\gamma([a, b])$ is at most
\[ \log_k \frac{ \tfrac{3}{2} \log \tfrac{1}{\theta} + \tfrac{1}{3}
  \log \tfrac{1}{\theta_\LL} + K_1 }{ \tfrac{1}{6} \log
  \tfrac{1}{\theta_\LL} } \le \log_k \left( 11 \log
\tfrac{1}{\theta_\LL} + 6 K_1 \right) = K_2, \]
where the right hand side $K_2$ only depends on $(S_h, \Lambda)$, as
required.
\end{proof}

We now consider the case in which an endpoint of $\gamma([u^0, v^0])$  lies in $\Lambda_+$, and so  $\gamma([u^0, v^0])$ is not contained between the sides of $R^0$ that are in $\Lambda_-$.

\begin{proposition}\label{prop:height change two}
Suppose that $(S_h, \Lambda)$ is a hyperbolic metric on $S$ together
with a regular pair of measured laminations.  There is a constant
$\theta_\LL$ such that for any $0 < \epsilon < 1$ there is a constant
$K$ such that for any $\epsilon$-truncated outermost rectangle $R^0$,
with optimal height $z$, determined by an intersection point of
$\gamma$ with a leaf $\ell$ of an invariant lamination of angle
$\theta \le \theta_\LL$, for any $t$ such that $\gamma(t)$ lies
outside the sides $\alpha^0_-$ and $\beta^0_-$ of $R^0$ which lie in
$\Lambda_-$, the value of the height function $h_\gamma(t)$ is equal
to $\log_k \log \tfrac{1}{\theta}$ up to additive error at most $K$.
\end{proposition}

\begin{proof}
Up to reversing the orientation on $\gamma$, we may assume that the endpoint of $\gamma([u^0, v^0])$ that lies in $\Lambda_+$ is $\gamma(v^0)$, as illustrated in
\Cref{fig:truncated optimal}.

\medskip
The radius function $\rho_\ell$ for $\ell$ at $\gamma(t)$ has a local
maximum of height bounded below by $\log \tfrac{1}{\theta}$.

\medskip
Let $I_\ell$ be the projection interval for $\ell$ onto $\gamma$, and
$I_{\beta_+}$ the nearest point projection interval of
$\beta_+$ to $\gamma$.

\medskip
The leaves $\ell$ and $\beta_+$ meet $\alpha^0_-$ at a bounded angle, and
since $\alpha^0_-$ intersects $\gamma$, the nearest point projection
interval $I_\ell$ is coarsely contained in the nearest point
projection interval $I_{\beta_+}$.

\medskip
We now find an upper bound on the distance from $R^0$ to
$\gamma(v^0)$.  By \Cref{prop:square upper bound}, the measure of the rectangle $R^1$ is at most
$dx(R^1) dy(R^1) \le A_{\max}$, so
\[ dy(R^1) \le \frac{A_{\max}}{dx(R^1)}.  \]
The height of the rectangle $R^1$ is at least
$dx(R^1) \ge \epsilon dx(R)$ and the measure of $R$ is at least
$dx(R) dy(R) \ge A_\LL$, so
\[ dy(R^1) \le \frac{A_{\max}}{\epsilon A_\LL} dy(R).  \]
Using the bound on the measure of the side of the outermost rectangle
from \Cref{prop:long side measure} gives 
\[ dy(R^1) \le \frac{A_{\max}}{\epsilon A_\LL} 2 Q_\LL ( \log
\tfrac{1}{\theta} + \tfrac{1}{6} \log \tfrac{1}{\theta_\LL} ).  \]
Using the quasi-isometry between the hyperbolic metric and the
Cannon--Thurston metric gives
\[ d_{S_h}(\beta^0_-, \gamma(v^0)) \le Q_\LL \frac{A_{\max}}{\epsilon
  A_\LL} 2 Q_\LL ( \log \tfrac{1}{\theta} + \tfrac{1}{6} \log
\tfrac{1}{\theta_\LL} ) + c_\LL.
\]
By the triangle inequality, the distance along $\gamma$ from
$\beta^0_-$ to $\gamma(v^0)$ is at most $d_{S_h}(\beta^0_-,
\gamma(v^0))$ plus the distance along $\beta^0_-$ from $\beta^0_+$ to
$p = \gamma \cap \beta^0_-$.  This latter distance is bounded by the
side length of $R_0$ contained in $\beta^0_-$.  This gives
\[ d_{S_h}(p, \gamma(v^0)) \le Q_\LL \frac{A_{\max}}{\epsilon A_\LL} 2
Q_\LL ( \log \tfrac{1}{\theta} + \tfrac{1}{6} \log
\tfrac{1}{\theta_\LL} ) + c_\LL + \frac{ Q_\LL A_\LL }{ \tfrac{2}{Q_\LL} (
  \log \tfrac{1}{\theta} - \tfrac{1}{6} \log \tfrac{1}{\theta_\LL} )
}  + c_\LL. \]
Using the fact that $\theta \le \theta_\LL$ gives
\[ d_{S_h}(p, \gamma(v^0)) \le K_3 \log \tfrac{1}{\theta} + K_4,  \]
where $K_3$ and $K_4$ depend only on $\LL$.

\medskip
The radius function $\rho_{\beta_+}$ is a coarse lower bound for the
radius function $\rho_\gamma$, so in particular, $\rho_\gamma(v^0)$ is
at most $K_3 \log \tfrac{1}{\theta} + K_4$. Therefore the height difference between $\gamma(t)$ and $\gamma(v^0)$
is at most
\[ \log_k \frac{ (K_3 + 1) \log \tfrac{1}{\theta}+ K_4 }{ \log
  \tfrac{1}{\theta} } \le \log_k ( K_3 + 1 + K_4  ), \]
which depends only on $(S_h, \Lambda)$, as required.
\end{proof}

We finally show that the length of $\tau_\gamma([u^0, v^0])$ is
bounded.

\begin{proof}[Proof (of \Cref{prop:truncated path bounded length})]
By \Cref{prop:height estimate}, there is a constant $K_1$ such that
the value of the height function at the intersection point is equal to
$\log_k \log \tfrac{1}{\theta}$ up to additive error at most $K_1$.
By \Cref{prop:height change main}, there is a constant $K_2$ such that
the variation in height along $\gamma([u^0, v^0])$ is at most $K_2$.
Therefore projecting $\tau_\gamma([u^0, v^0])$ to
$\gamma_z([u^0, v^0])$ changes the length by at most a factor of
$k^{K_1 + K_2}$.

\medskip
By \Cref{prop:optimal height bounded length}, there is a constant $K_3$
such that $\gamma_z([u^0, v^0])$ has arc length at most $K_3$.
Hence, $\gamma_\tau([u^0, v^0])$ has arc length at most
$k^{K_1 + K_2} K_3$, where all the constants depend only on
$(S_h, \Lambda)$, as required.
\end{proof}

%%%%%%%%%%%%%%%%%%%%%%%%%%%%%%%%%%%%%%%%%%%%%%%%%%%%%%%%%%%%%%%%%%%%%%%
\subsubsection{Corner segments create tame bottlenecks}\label{section:large angles}
%%%%%%%%%%%%%%%%%%%%%%%%%%%%%%%%%%%%%%%%%%%%%%%%%%%%%%%%%%%%%%%%%%%%%%%

By \Cref{def:corner segment}, a \emph{corner segment} is be a segment of $\gamma$ that cuts off a
corner of an innermost rectangle.  For some terminology, we now distinguish between two possibilities for the corner segments.

\begin{definition}
A segment $\gamma(I)$ of a non-exceptional geodesic is a
\emph{positive length corner segment} if it is a corner segment with positive length. A segment $\gamma(I)$ is a \emph{zero length corner
segment} if $I$ consists of a single point, and $\gamma(I)$ is the
vertex of an innermost rectangle. 
\end{definition}

In either case, a corner segment determines exactly two angles of
intersection, one with each invariant lamination.  In this section, we
will show that a corner segment creates a bottleneck in $\wsr$.

\begin{cor:corner distance}
Suppose that $(S_h, \Lambda)$ is a hyperbolic metric on $S$ together
with a suited pair of measured laminations.  Then there are
constants $r$ and $K$, such that for any corner segment
$I = [t_1, t_2]$ of any non-exceptional geodesic $\gamma$, there are
parameters $u \le t_1 \le t_2 \le v$ such that for any $t \in [u, v]$,
the point $\tau_\gamma(t)$ is a $(r, K)$-bottleneck for the ladders
over $\gamma((-\infty, u])$ and $\gamma([v, \infty))$.  Furthermore,
the length of $\tau_\gamma([u, v])$ is at most $K$.
\end{cor:corner distance}

If at least one angle is small, then it suffices to observe that the corner segment is contained in the interval $[u,v]$ in \Cref{lemma:small
  angles}, and \Cref{cor:corner distance} follows.

\medskip
It therefore remains to show that if a corner segment meets both invariant laminations with large angles, then it creates a
bottleneck.

\begin{proposition}\label{prop:transverse corner}
Suppose that $(S_h, \Lambda)$ is a hyperbolic metric on $S$ together
with a suited pair of measured laminations.  Let $\theta_\LL$ be
the constant from \Cref{def:theta}.  Then there are constants $r > 0$
and $K \ge 0$ such that for any non-exceptional geodesic $\gamma$ in
$S_h$, for any corner segment $\gamma([t_1, t_2])$, meeting both
invariant laminations at angles at least $\theta_\LL$, there are
parameters $u \le t_1 \le t_2 \le v$ such that for any $t \in [u, v]$,
the point $\tau_\gamma(t)$ is a $(r, K)$-bottleneck for the ladders
over $\gamma((-\infty, u])$ and $\gamma([v, \infty))$.  Furthermore,
the length of $\tau_\gamma([u, v])$ is at most $K$.
\end{proposition}

We start by showing that any geodesic which intersects a leaf of a
lamination intersects a rectangle intersecting that leaf, with a lower
bound on its transverse measure, which depends only on the angle of
intersection.  The final result will then follow from constructing
these rectangles for both leaves at each end of the corner segment,
and then taking the intersection of these two rectangles.

\begin{proposition}\label{prop:lower bound transverse measure}
Suppose that $(S_h, \Lambda)$ is a hyperbolic metric on $S$ together
with a suited pair of measured laminations.  Then for any
non-exceptional geodesic $\gamma$ intersecting a leaf of a lamination
$\ell \in \Lambda_+$ at angle $0 < \theta \le \pi/2$, there are leaves
$\ell_1$ and $\ell_2$ in the other lamination $\Lambda_-$ with the
following properties.

\begin{enumerate}

\item The leaf $\ell$ is a common leaf of intersection for $\ell_1$
and $\ell_2$.

\item The nearest point projections of $\gamma, \ell_1$ and $\ell_2$
to $\ell$ are all disjoint.

\item The $(\ell_1, \ell_2)$-maximal rectangle $R$ intersects both
$\ell$ and $\gamma$.

\item The transverse measure of the rectangle is at least
\[ dx(R) \ge \frac{A_\LL}{2 Q_\LL( \log \tfrac{1}{\theta} + \log
  \tfrac{1}{\alpha} + 2 T_0 + L_\LL ) + c_\LL}, \]
where $Q_\LL$ and $c_\LL$ are the constants defined in \Cref{c:Q_pa},
$T_0$ is the constant defined in \Cref{prop:projection interval}, and $L_\LL$ is the
constant defined in \Cref{c:L_pa}, and all these constants depend only
on the choice of $(S_h, \Lambda)$.

\end{enumerate}

The same result holds with the invariant measured laminations
$(\Lambda_+, dx)$ and $(\Lambda_-, dy)$ swapped.
\end{proposition}

\begin{proof}
Up to swapping the laminations we may assume that $\ell$ is in
$\Lambda_+$.

\medskip
We parametrize $\ell$ with a unit speed such that the intersection
point with $\gamma$ is $\ell(0)$.  Then the nearest point projection
of $\gamma$ to $\ell$ is contained in the interval
$I = \ell(- \log \tfrac{1}{\theta} - T_0 , \log \tfrac{1}{\theta} +
T_0 )$.

\medskip
Consider intervals $I_1$ and $I_2$ of length $L_\LL$ on
either side of $I$, distance $\log \tfrac{1}{\alpha} + T_0$ from
$I$.  Each interval $I_i$ intersects a leaf $\ell^i_-$ of $\Lambda_-$, whose nearest point projection to $\ell$ is disjoint from
$I$.  The leaf $\ell$ is then common for both $\ell^1_-$ and
$\ell^2_-$, so there is a $(\ell^1_-, \ell^2_-)$-maximal rectangle $R$
with area at least $A_\LL$.  As the
measure of the side of the rectangle parallel to $\ell$ is at most
$2 Q_\LL( \log \tfrac{1}{\theta} + \log \tfrac{1}{\alpha} + 2 T_0 +
L_\LL ) + c_\LL$, the result follows.
\end{proof}

We now complete the proof of \Cref{prop:transverse corner}.

\begin{proof}[Proof (of \Cref{prop:transverse corner})]
Let $\ell^+$ and $\ell^-$ be the leaves of intersection at each end of
the corner segment. 
%Let $p$ be the point of intersection of $\ell^+$ and $\ell^-$.  
Up to reversing the orientation on $\gamma$, we may
assume that $\gamma$ hits $\ell^-$ first, as illustrated in
\Cref{fig:corner segment}.

\medskip
By \Cref{prop:lower bound transverse measure}, the intersection point
$\gamma \cap \ell^-$ is contained in a rectangle $R_1$ with measure
\[ dy(R_1) \ge \frac{A_\LL}{2Q_\LL ( \log \tfrac{1}{\theta_\LL} + \log
  \tfrac{1}{\alpha} + 2 T_0 + L_|pa ) + c_\LL} := K_1.  \]
such that the $\Lambda_+$ sides of $R_1$ have nearest point projections
to $\ell^-$ disjoint from the nearest point projection of
$\gamma$ to $\ell^-$.  In particular, $\ell^+$ intersects the interior of $R_1$, and the sides of $R_1$ in
$\Lambda_-$ intersect $\gamma$.

\medskip
Similarly by \Cref{prop:lower bound transverse measure}, the other
intersection point $\gamma \cap \ell^+$ is contained in a rectangle
$R_2$ with measure $dx(R_2) \ge K_1$ such that $\ell^-$ intersects the
interior of $R_2$, and the sides of $R_2$ in $\Lambda_+$ intersect
$\gamma$.

\medskip
The intersection $R = R_1 \cap R_2$ is thus a rectangle such that
$R$ contains the point of intersection of $\ell^+$ and $\ell^-$, the measure of $R$ satisfies $dx(R) dy(R) \ge K_1^2$, and all sides of $R$ intersect
$\gamma$. Hence $R$ is a transverse rectangle for $\gamma$.  We use the
previous notation for the sides of $R$, so the initial side of $R$ in
$\Lambda_-$ is $\alpha^-$, and the terminal side of $R$ in $\Lambda_+$
is $\beta^+$.

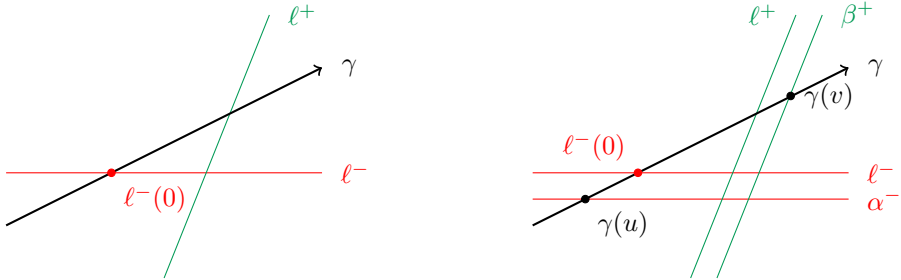
\begin{figure}[h]
\begin{center}
\begin{tikzpicture}[scale=0.7]

\tikzstyle{point}=[circle, draw, fill=black, inner sep=1pt]

\tikzstyle{redpoint}=[circle, draw, fill=red, inner sep=1pt]

\draw [color=red] (0, 0) -- (6, 0) node [label=right:$\ell^-$] {};
\draw [color=ForestGreen] (3, -2) -- (5, 3) node [label=right:$\ell^+$] {};

\draw [thick, arrows=->] (0, -1) -- (6, 2) node [label=right:$\gamma$] {};

\draw [color=red] (2, 0) node [color=red, redpoint, label=below right:$\ell^-(0)$] {};

\begin{scope}[xshift=10cm]

\draw [color=red] (0, 0) -- (6, 0) node [label=right:$\ell^-$] {};

\draw [color=red] (0, -0.5) -- (6, -0.5) node [label=right:$\alpha^-$] {};

\draw [color=ForestGreen] (3, -2) -- (5, 3) node [label=left:$\ell^+$] {};

\draw [color=ForestGreen] (3.5, -2) -- (5.5, 3) node [label=right:$\beta^+$] {};

\draw [thick, arrows=->] (0, -1) -- (6, 2) node [label=right:$\gamma$] {};

\draw [color=red] (2, 0) node [color=red, redpoint, label=above
left:$\ell^-(0)$] {};

\draw (1, -0.5) node [point, label=below right:$\gamma(u)$] {};

\draw (4.9, 1.46) node [point, label=right:$\gamma(v)$] {};

\end{scope}

\end{tikzpicture}
\end{center}
\caption{A corner segment with large angles.} \label{fig:corner segment}
\end{figure}

Let $U_R$ and $V_R$ be the initial and terminal quadrants for $R$, and
choose $u$ and $v$ such that
$U = U_R \cap F(\gamma) = F(\gamma((-\infty, u]))$ and
$V = V_R \cap F(\gamma) = F(\gamma([v, \infty)))$.  By
\Cref{lemma:bottleneck} there are constants $r$ and $K_2$ such that any point
$t \in [u, v]$ is a $(r, K_2)$-bottleneck for $U$ and $V$.

\medskip
We now bound the length of $\gamma([u, v])$.

\medskip
Let $p$ be the intersection point between $\alpha^-$ and $\beta^+$.
As $\gamma([u, v])$ is monotonic with respect to $\alpha^-$ and
$\beta^+$, by \Cref{prop:monotonic length bound} the length of $\gamma([u, v])$ is bounded by
$d(\gamma(u), p) + d(p, \gamma(v))$.

\medskip
Let $q$ be the nearest point projection of $\gamma(u)$ to $\ell^-$.
The point $\gamma(u)$ lies on $\gamma$, and so is contained in the
nearest point projection of $\gamma$ to $\ell^-$.

\medskip
By construction, the radius of the nearest point projection interval
of $\alpha^-$ to $\ell$ is at least $\log \tfrac{1}{\alpha} + T_0$,
and $q$ is at least distance $\log \tfrac{1}{\alpha} + T_0$ from the
endpoints of the projection interval.  Therefore the distance in
$PSL(2, \RR)$ between $\gamma(u)$ and $q$ is at most $\alpha$, so the
distance in $\HH^2$ is at most $\alpha + \pi$.

\medskip
Therefore $d(\gamma(u), p) \le \log \tfrac{1}{\theta_\LL} + T_0 +
\alpha + \pi$.

\medskip
Therefore the length of $\gamma([u, v])$ is at most
$K_3 := 2 ( \log \tfrac{1}{\theta_\LL} + T_0 + \alpha + \pi )$.

\medskip
We now bound the value of the height function along the corner
segment $I$.  As $\gamma(I)$ meets $\Lambda_-$ at angle at least
$\theta_\LL$ at $t_1$, the height function at $t_1$ is coarsely
non-positive, i.e.  $h_\gamma(t_1) \le K_4$, where $K_4$ is the
constant from \Cref{prop:height estimate}.  Similarly, as $\gamma(I)$
meets $\Lambda_+$ at angle at least $\theta_\LL$ at $t_2$, the height
function at $t_2$ is coarsely non-negative, i.e.
$h_\gamma(t_1) \ge - K_4$.  By \Cref{prop:ell lip}, the height function
is $(1/\log k)$-Lipschitz, and the length of $I$ is at most $L_\LL$,
where $L_\LL$ is the constant from \Cref{prop:cobounded
  intersections}.  Therefore, for any $t \in I$, the value of the
height function is bounded above and below, i.e.
$| h_\gamma(t) | \le K_4 + L_\LL / \log k =: K_5$.

\medskip
The lengths of $\gamma([u, t_1])$ and $\gamma([t_2, u])$ are at most
$K_3$, and the height function is $(1/\log k)$-Lipschitz.  Therefore,
the value of the height function on $[u, v]$ is at most
$| h_\gamma(t) | \le K_5 + K_3/\log k =: K_6$.

\medskip
Nearest point projection of $\tau_\gamma$ to $S_0$ therefore distorts
distance by at most a factor of $k^{K_6}$.

\medskip
Therefore the length of $\tau_\gamma([u, v])$ is at most
$K_3 k^{K_6}$.
\end{proof}

%%%%%%%%%%%%%%%%%%%%%%%%%%%%%%%%%%%%%%%%%%%%%%%%%%%%%%%%%%%%%%%%%%%%%%%
\subsection{Straight segments are quasigeodesic}\label{section:straight qg}
%%%%%%%%%%%%%%%%%%%%%%%%%%%%%%%%%%%%%%%%%%%%%%%%%%%%%%%%%%%%%%%%%%%%%%%

In this section we prove that segments of the test path over straight
segments are quasigeodesic.

\begin{lemma:straight qg}
Suppose that $(S_h, \Lambda)$ is a hyperbolic metric on $S$ together
with a suited pair of measured laminations.  Then there are
constants $Q$ and $c$ such that for any non-exceptional geodesic
$\gamma$, and any straight interval $I$, the test path
$\tau_\gamma(I)$ is $(Q, c)$-quasigeodesic.
\end{lemma:straight qg}

In \Cref{section:short}, we show that if the intersection
interval of $\gamma$ with an ideal complementary region $R$ is short, then the
test path over that interval is also short.  We may then consider
segments $\gamma \cap R$ which are reasonably long.  In
\Cref{section:long segments sign}, we show that if $\gamma \cap R$ is
sufficiently long, then the height function does not change sign along
this segment.  Inside the flow set $F(R)$, the Cannon--Thurston metric is quasi-isometric to the union of a
number of hyperbolic halfspaces glued along a common bi-infinite boundary
geodesic.  In \Cref{section:qg paths}, we show that specific paths in
these halfspaces are quasigeodesic.  In \Cref{section:long segments},
we show that the test path over a long segment $\gamma \cap R$ is
close to one of these quasigeodesic paths, and is hence quasigeodesic.
We then complete the final step to show that this also
applies to a straight segment that is a union of two intersection
segments meeting in a non-rectangular polygon.

%%%%%%%%%%%%%%%%%%%%%%%%%%%%%%%%%%%%%%%%%%%%%%%%%%%%%%%%%%%%%%%%%%%%%%%
\subsubsection{Short segments in complementary regions}\label{section:short}
%%%%%%%%%%%%%%%%%%%%%%%%%%%%%%%%%%%%%%%%%%%%%%%%%%%%%%%%%%%%%%%%%%%%%%%

In this section, we show that segments of bounded length in
ideal complementary regions give rise to bounded length segments of the test
path.

\begin{proposition}\label{prop:short segments}
Suppose that $(S_h, \Lambda)$ is a hyperbolic metric on $S$ together
with a suited pair of measured laminations.  Then for any constant
$L > 0$ there is a constant $K > 0$ such that for any non-exceptional
geodesic $\gamma$ in $\widetilde{S}_h$ and for any ideal complementary
region $R$, if the length of the intersection interval $I_R$ is at
most $L$, then the length of the test path $\tau_\gamma(I_R)$ over
that interval is at most $K$.
\end{proposition}

\begin{proof}
Up to swapping laminations and reversing the sign on the
height function, we may assume that the complementary region $R$ has boundary in $\Lambda_+$. 

\medskip
The interior of the intersection interval $I_R$ is disjoint from
$\Lambda_+$. By the quasi-isometry between the hyperbolic and Cannon--Thurston metrics (\Cref{prop:qi}), the $\Lambda_-$ measure of $I_R$ is at most $Q_\LL L + c_\LL$. 

\medskip
If the interior of $I_R$ is disjoint from both laminations, then $I_R$
is an innermost segment. By \Cref{prop:innermost bounded length}, the arc length of $\tau_\gamma(I_R)$
is then at most $K_1$. 
So we may assume that the intersection interval $I_R$ intersects
$\Lambda_-$.  The metric on the ladder $F(\gamma(I_R))$ is then
determined by the $z$-coordinate, and the measure of the intersections
with $\Lambda_-$.

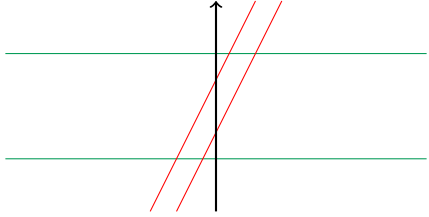
\begin{figure}[h]
\begin{center}
\begin{tikzpicture}[scale=0.7]

\tikzstyle{point}=[circle, draw, fill=black, inner sep=1pt]

\draw [color=ForestGreen] (0, 0) -- (8, 0);

\draw [color=ForestGreen] (0, 2) -- (8, 2);

\draw [color=red]  (2.75, -1) -- (4.75, 3);

\draw [color=red]  (3.25, -1) -- (5.25, 3);

\draw [thick, arrows=->] (4, -1) -- (4, 3);

\end{tikzpicture}
\end{center}
\caption{Outermost transverse rectangles determined by
  $\ell^+_1$ and $\ell^+_2$.} \label{fig:outermost two leaves}
\end{figure}

Moving in the positive
$z$-direction scales the $\Lambda_-$-measure by a factor of $k^{-z}$.
Thus if the height function along $I_R$ is bounded below by
$L_1 = -2 L -L/\log k$, then the arc length of $\tau_\gamma(I_R)$ in
the Cannon--Thurston metric is at most $(Q_\LL L + c_\LL)k^{L_1}$.

\medskip
Now suppose there is a point $t$ on $I_R$ with height function
$h_\gamma(t)$ less than $L_1$.  As the height function is
$(1/\log k)$-Lipschitz, this implies that the height function is at
most $- 2 L$ on all of $I_R$.

\medskip
Let $\ell^-_1$ and $\ell^-_2$ be outermost leaves of $\Lambda_-$
intersecting $I_R$ at $\gamma(t_1)$ and $\gamma(t_2)$ with angles
$\theta_1$ and $\theta_2$.  By \Cref{prop:height estimate}, the height
function determines the angle up to bounded error, so both $\theta_1$
and $\theta_2$ are at most $\theta_\LL e^{-k^{2L - K_1}}$, where $K_1$
is the constant from \Cref{prop:height estimate}.

\medskip
Consider the maximal rectangle with sides contained in $\ell_1$ and
$\ell_2$.  These two sides fellow travel for distance at least $\log
\tfrac{1}{\theta_1}$, up to additive error.

\medskip
Therefore by \Cref{prop:square upper bound} the measure of the other sides is at most
$A_{\max} / \log \tfrac{1}{\theta_1}$.

As the height function is $(1/\log k)$-Lipschitz, the length of
$\tau_\gamma(I_R)$ is at most $L / \log k + k^{A_{\max} / \log
  \tfrac{1}{\theta_1}} \le L / \log k + k^{A_{\max} / \log
  \tfrac{1}{\theta_\LL}}$.
\end{proof}

%%%%%%%%%%%%%%%%%%%%%%%%%%%%%%%%%%%%%%%%%%%%%%%%%%%%%%%%%%%%%%%%%%%%%%%
\subsubsection{Long intersection intervals have height functions with
  consistent signs}
\label{section:long segments sign}
%%%%%%%%%%%%%%%%%%%%%%%%%%%%%%%%%%%%%%%%%%%%%%%%%%%%%%%%%%%%%%%%%%%%%%%

In this section, we show that if $R_+$ is an ideal complementary region of
$\Lambda^+$, and if the length of the intersection interval is
sufficiently long, then the height function along the intersection
interval is non-negative.  Swapping laminations a similar statement holds for ideal complementary regions of $\Lambda_-$. 

\begin{proposition}\label{prop:long segment sign}
Suppose that $(S_h, \Lambda)$ is a hyperbolic metric on $S$ together
with a suited pair of measured laminations.  Then there is a
constant $L_R$, such that for any ideal polygon $R$ with boundary in
$\Lambda_+$, and any non-exceptional geodesic $\gamma$ crossing $R$ in
an intersection interval $I_R$ of length at least $L_R$, the height
function along the intersection interval $I_R$ satisfies
$h_\gamma(t) \ge 0$.  Similarly, if $R$ has ideal boundary contained
in $\Lambda_-$, then the height function along $I_R$ satisfies
$h_\gamma(t) \le 0$.
\end{proposition}

\begin{proof}
Up to swapping laminations and reversing the sign of the height function, we may assume that $R$ is an ideal complementary region of $\Lambda_+$.
We choose $L_R > 2D_\LL + 4 L_\LL + 2 \rho_\LL$, where
$D_\LL$ is the constant from \Cref{prop:innermost polygon diameter
  bound}, $L_\LL$ is the constant from \Cref{prop:cobounded
  intersections} and $\rho_\LL$ is the constant from
\Cref{prop:overlap}.
We denote $\gamma \cap R$ by $\gamma(I_R)$ and assume that the interval $\gamma(I_R)$ has length at least $L_R$.
Suppose that the interior of $\gamma(I_R)$ is disjoint from $\Lambda_-$. Then $\gamma(I_R)$ is disjoint from both laminations. By \Cref{prop:innermost polygon diameter bound}, it is an innermost interval with
length at most $D_\LL$ which is not possible as $L_R > D_\LL$.
Thus $\gamma(I_R)$ intersects $\Lambda_-$.

\medskip
Suppose that $\gamma(t_1)$ is a point on $\gamma(I_R)$ such that $h_\gamma (t_1) < 0$.  Since extended laminations are closed, by \Cref{def:height function} there is a leaf $\ell^-_1$ of $\overline{\Lambda}_-$ with distance in $\PSL(2, \RR)$ at most $\theta_\LL$ from $\gamma(t_1)$. 

\medskip
Up to reversing the orientation of $\gamma$, let $t_2 \ge t_1$ be the smallest time such that there is a leaf $\ell^-_2$ of $\Lambda^-$ intersecting $\gamma(I_R)$ at the point $\gamma(t_2)$.
If $t_2 \neq t_1$ then $\ell^-_2$ is a boundary leaf of an ideal complementary region $R'$ of $\Lambda_-$ and the segment $\gamma(t_1, t_2)$ is contained in 
$R'$. It follows that $t_2 - t_1 < D_\LL$.
By
\Cref{prop:exp 1lip}, the radius function is $1$-Lipschitz.
Therefore, the difference between the values 
$\rho_\gamma(t_1)$ and $\rho_\gamma(t_2)$ of the radius function is at most $D_\LL$.
Therefore the radius of the projection interval for $\ell^-_2$ is at
least $\log \tfrac{1}{\theta_\LL} - D_\LL$.

\medskip
There is a boundary leaf $\ell^+$ of $R$ that intersects 
$\ell^-_2$ within distance $L_\LL$ of $\gamma(t_2)$.
Let the intersection point be $p$, and let $q$ be the
nearest point on $\gamma$ to $p$.  
 We now show that $\ell^+$
intersects $\gamma$, as illustrated in \Cref{fig:intersection point
  close to negative height function}, where we have drawn $\ell^+$
intersecting $\gamma$ at $\gamma(t_3)$.

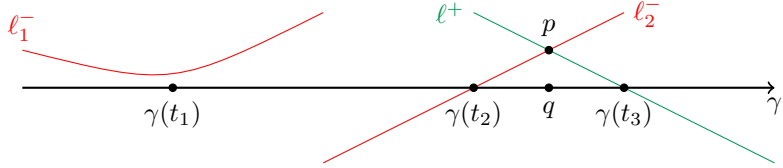
\begin{figure}[h]
\begin{center}
\begin{tikzpicture}

\tikzstyle{point}=[circle, draw, fill=black, inner sep=1pt]

\draw [color=red] (0, 0.5) node [color=red, above] {$\ell^-_1$} .. controls (2, 0) .. (4, 1);

\draw [color=red] (4, -1) -- (8, 1) node [color=red, right] {$\ell^-_2$};

\draw [color=ForestGreen] (6, 1) node [left] {$\ell^+$} -- (10, -1);

\draw [thick, arrows=->] (0, 0) -- (10, 0) node [below] {$\gamma$};

\draw (2, 0) node [point, label=below:$\gamma(t_1)$] {};

\draw (6, 0) node [point, label=below:$\gamma(t_2)$] {};

\draw (8, 0) node [point, label=below:$\gamma(t_3)$] {};

\draw (7, 0.5) node [point, label=above:$p$] {};

\draw (7, 0) node [point, label=below:$q$] {};

\end{tikzpicture}
\end{center}
\caption{An intersection point close to a negative value of the height
  function.} \label{fig:intersection point close to negative height
  function}
\end{figure}

\medskip Suppose $\ell^+$ does not intersect $\gamma$.  Then one of
the endpoints of $\ell^+$ lies between the endpoints of $\ell^-_2$ and
$\gamma_+$ on the boundary circle at infinity. This means that the
projection interval for $\ell^+$ onto $\gamma$ must extend past the
end of the projection interval for $\ell^-_2$ onto $\gamma$ in the
direction of the endpoint $\gamma_+$.  The distance from $\gamma(t_2)$
to $q$ is at most $2 L_\LL$.  Thus the length of the overlap of the
projection intervals for $\ell^-_2$ and $\ell^+$ to $\gamma$ is at
least $\log \tfrac{1}{\theta_\LL} - D_\LL - 2 L_\LL $,
%
%thetacondition
%
which by our choice of $\theta_\LL$ is greater than
$\rho_\LL$, the maximal overlap between projections of leaves to
$\gamma$, contradicting
\Cref{prop:overlap}.

\medskip
We conclude that $\ell^+$ intersects $\gamma(I_R)$ at a point $\gamma(t_3)$. 
The point $\gamma(t_3)$ lies in the nearest point projection interval
of $\ell^+$ to $\gamma$. So the distance from $q$ to $\gamma(t_3)$ is
at most $\rho_\LL$. 
Therefore the distance from $\gamma(t_1)$ to
$\gamma(t_3)$ is at most $D_\LL + 2 L_\LL + \rho_\LL$.
As this argument
applies to traveling along $\gamma$ in the other direction, this
shows that the length of $I_R$ is at most $2(D_\LL + 2 L_\LL + \rho_\LL)$ which is less than $L_R$, a contradiction. 
\end{proof}

%%%%%%%%%%%%%%%%%%%%%%%%%%%%%%%%%%%%%%%%%%%%%%%%%%%%%%%%%%%%%%%%%%%%%%%
\subsubsection{Quasigeodesic paths in the upper half space
  model}\label{section:qg paths}
%%%%%%%%%%%%%%%%%%%%%%%%%%%%%%%%%%%%%%%%%%%%%%%%%%%%%%%%%%%%%%%%%%%%%%%

Recall that the metric on the ladder $F(\ell_-)$ is not the standard metric on
the upper half space but is quasi-isometric to it under the map
$(x, z) \mapsto (x, k^{-z})$.  

\medskip In this section, we show that paths in the upper half space
arising as graphs of 1-Lipschitz functions $\RR \to \RR_+$ are
quasi-geodesic.  To begin with, a line with slope one in the upper
half space model stays a constant distance from a vertical line, and
is hence a quasigeodesic. This is illustrated in both models in
\Cref{fig:metrics}. From this we will deduce that absolute value
functions are quasi-geodesic. We will then show that
$1$-Lipschitz functions are contained in bounded neighborhoods of
certain absolute value functions, to get the required conclusion.

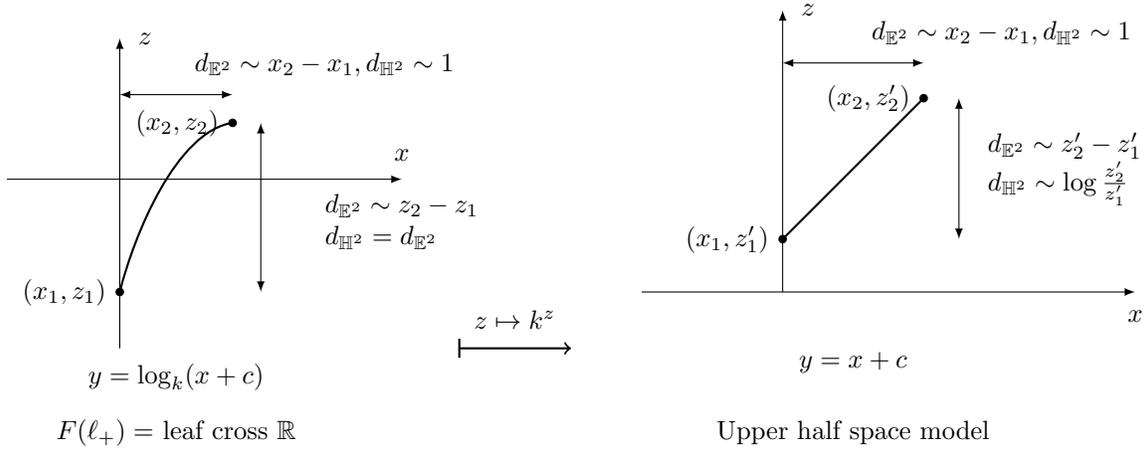
\begin{figure}[h]
\begin{center}
\begin{tikzpicture}[scale=0.75]

\tikzstyle{point}=[circle, draw, fill=black, inner sep=1pt]

\draw [-latex] (-1, -3) -- (-1, 2.5) node [label=right:$z$] {};

\draw [-latex] (-3, 0) -- (4, 0) node [label=above:$x$] {};

\draw (-1, -2) node [point, label=left:${(x_1, z_1)}$] {};
\draw (1, 1) node [point, label=left:${(x_2, z_2)}$] {};

\draw [thick] (-1, -2) .. controls (-0.75, -1) and (0, 0.85) .. (1, 1);

\draw [latex-latex] (-1, 1.5) -- (1, 1.5) node [midway,
label=above right:${d_{\mathbb{E}^2} \sim x_2 - x_1, d_{\mathbb{H}^2} \sim 1}$] {};

\draw [latex-latex] (1.5, -2) -- (1.5, 1);

\draw (4, -0.75) node [align=left] {$d_{\mathbb{E}^2} \sim z_2 - z_1$
  \\ $d_{\mathbb{H}^2} = d_{\mathbb{E}^2}$};

\draw (0, -3.5) node {$y = \log_k ( x + c )$};

\draw (0, -4.5) node {$F(\ell_+) = $ leaf cross $\mathbb{R}$};

\begin{scope}[xshift=12cm, yshift=-2cm, scale=1.25]

\draw [-latex] (-3, 0) -- (4, 0) node [label=below:$x$] {};
\draw [-latex] (-1, 0) -- (-1, 4) node [label=right:$z$] {};

\begin{scope}[yshift=-0.25cm]

\draw [thick] (-1, 1) -- (1, 3);

\draw (-1, 1) node [point, label=left:${(x_1, z'_1)}$] {};

\draw (1, 3) node [point, label=left:${(x_2, z'_2)}$] {};

\draw [latex-latex] (-1, 3.5) -- (1, 3.5) node [midway, label=above
right:${d_{\mathbb{E}^2} \sim x_2 - x_1, d_{\mathbb{H}^2} \sim 1}$] {};

\draw [latex-latex] (1.5, 1) -- (1.5, 3);

\draw (3, 2) node [align=left] {$d_{\mathbb{E}^2} \sim z'_2 -
  z'_1$ \\ $d_{\mathbb{H}^2} \sim \log \frac{z'_2}{z'_1}$};
  
\end{scope}

\draw (0, -1) node {$y = x + c$};

\draw (0, -2) node {Upper half space model};

\end{scope}

\draw [thick, arrows=|->] (5, -3) -- (7, -3) node [midway,
label=above:$z \mapsto k^z$] {};

\end{tikzpicture}
\end{center}
\caption{Quasigeodesics in two models for $\HH^2$.} \label{fig:metrics}
\end{figure}

\medskip
In subsequent sections, we will show that the height function over a single leaf is contained in the above class of functions. Hence, a test path in the
upper half space model is a 
quasigeodesic.  

\begin{definition}
Suppose that $I \subseteq \RR$ is a (possibly infinite) subinterval of $\RR$.
Suppose that $a \colon I \to \RR_+$ is a function. 
We call the path $\alpha$ in
the upper half space model of $\HH^2$ given by
$\alpha(t) = (t, a(t))$ \emph{the path determined
  by the function $a$}.
\end{definition}

We consider a specific collection of paths determined by absolute
value functions.

\begin{definition}
Given constants $h > 1$ and $x$, an \emph{absolute value path} is
the path in the upper half space determined by $a(t) = h - |t - x|$
defined on the interval $I = |t - x| \le h - 1$.  We have chosen $I$
so that $a(t) \ge 1$ for all $t \in I$.
\end{definition}

We start by showing that the paths determined by these functions are
quasigeodesic.

\begin{proposition}\label{prop:abs val qg}
There are constants $Q \ge 1$ and $c \ge 0$ such that every absolute
value path in the upper half space model, with unit speed
parametrization, is $(Q, c)$-quasigeodesic.
\end{proposition}

\begin{proof}
Set $z_0 = x + hi$.
By the isometry $z \to \tfrac{1}{h}(z-z_0)$ of the upper half space, every absolute value path is isometric to a subpath of the path 
determined by $1 - |t|$ for $|t| < 1$. 
So it suffices to show that the
path determined by $1 - |t|$ for $|t| < 1$ is quasigeodesic.

\medskip
The line $y = x$ in the upper half space is invariant under the
isometry $z \mapsto \lambda z$, and so is a constant distance $c_1$
from the vertical geodesic given by the vertical $y$-axis.  Nearest
point projection from $y = x$ to the vertical axis contracts distances
by a constant factor $Q_1$, so $y = x$ is a
$(Q_1, c_1)$-quasigeodesic.

\medskip
The absolute value path $1 - |t|$ is therefore a union of two
quasigeodesic paths with distinct endpoints, and so is
$(Q_2, c_2)$-quasigeodesic, where the constants depends on $Q_1$ and
$c_1$, and the Gromov product of the limit points of the two paths
based at their common point.
\end{proof}

\begin{definition}
Suppose that $\alpha(t) = (t, a(t))$ is the path determined by the function $a \colon I \to \RR_+$. 
Given a constant $K \ge 0$, the \emph{vertical $K$-neighborhood of
  $\alpha$}, which we shall denote $V_K(\alpha)$, consists of all
points whose vertical distance to $\alpha$ (along lines with fixed real part in the upper half space model) in the Euclidean metric is
at most $K$.
\end{definition}

We now show that $1$-Lipschitz functions contained in bounded vertical
neighborhoods of absolute value paths are quasigeodesic.

\begin{proposition}\label{prop:K-neighborhood}
Given $K > 0$ there are constants $Q > 0$ and $c \ge 0$ such that for any $1$-Lipschitz function $b \colon I \to \RR_+$ for which
\begin{itemize}
    \item $b(t) \ge 1$, and
    \item the path $\beta(t) = (t, b(t))$ determined by $b$ is  contained in a vertical $K$-neighborhood of an absolute value path,
\end{itemize}
the unit speed parametrization of $\beta$ is $(Q, c)$-quasigeodesic. 
\end{proposition}

\begin{proof}
Let $\alpha(t)$ be a unit speed parametrization of the absolute value
path, and let $\beta(t)$ be a unit speed parametrization of $\beta$.
As $\beta$ is a unit speed parametrization,
$d_{\HH^2}(\beta(s), \beta(t)) \le t - s$.  We now find a lower bound
on distances along the path $\beta$.

\medskip
Let $\alpha(s')$ %and $\alpha(t')$ 
be the vertical projection of $\beta(s)$, and let
$\alpha(t')$ be the vertical projection of $\beta(t)$.  The vertical
projection from $\beta$ to $\alpha$ changes distances by at most a
factor of $e^{K}$, so
\begin{equation}\label{eq:length}
e^{-K} \text{length} (\alpha([s', t'])) - 2 K \le
\text{length}( \beta([s, t] )) \le e^{K} \text{length} (\alpha([s',
t'] )) + 2 K.
\end{equation}
As both $\alpha$ and $\beta$ have unit speed parametrizations, this shows
\[ e^{-K} (t' - s') - 2 K \le t - s \le e^{K} ( t' - s' ) + 2 K.  \]
As $\beta$ lies in a vertical $K$-neighborhood of $\alpha$,
\[ d_{\HH^2}( \beta(s), \beta(t)) \ge d_{\HH^2}( \alpha(s'),
\alpha(t') ) - 2K.  \]
By \Cref{prop:abs val qg}, the path $\alpha$ is
$(Q, c)$-quasigeodesic, so
\[ d_{\HH^2}( \beta(s), \beta(t)) \ge \tfrac{1}{Q} (t' - s') -
c - 2K.  \]
Using \eqref{eq:length} gives
\[ d_{\HH^2}( \beta(s), \beta(t)) \ge \tfrac{1}{Q} (
\tfrac{1}{e^{K}}(t - s) - 2 K ) - c - 2K,  \]
so $\beta$ is $(Q, c)$-quasigeodesic, as required.
\end{proof}

%%%%%%%%%%%%%%%%%%%%%%%%%%%%%%%%%%%%%%%%%%%%%%%%%%%%%%%%%%%%%%%%%%%%%%%
 \subsubsection{Long segments in complementary regions}
\label{section:long segments}
%%%%%%%%%%%%%%%%%%%%%%%%%%%%%%%%%%%%%%%%%%%%%%%%%%%%%%%%%%%%%%%%%%%%%%%

We now complete the proof of \Cref{lemma:straight qg}, showing that
segments of the test path over straight segments are quasigeodesic.
Recall that there are exactly two types of straight intervals.  Either
a straight interval is an intersection interval, with both endpoints
in the cusps of the ideal complementary region, or the straight
interval is the union of two intersection intervals, for ideal
complementary regions intersecting in a non-rectangular polygon.

\medskip
We shall start by showing the result for intersection intervals for a
single ideal complementary region.

\begin{lemma}\label{lemma:intersection interval qg}
Suppose that $(S_h, \Lambda)$ is a hyperbolic metric on $S$ together
with a suited pair of measured laminations.  Then there are
constants $Q$ and $c$ such that for any non-exceptional geodesic
$\gamma$, and any ideal complementary region $R$ with boundary in
either $\Lambda_+$ or $\Lambda_-$, the test path over the intersection
interval $\tau_\gamma(I_R)$ is $(Q, c)$-quasigeodesic.
\end{lemma}

Once we have shown this, it will be simple to show that the test path
over the union of two intersection segments meeting in a
non-rectangular polygon is also quasigeodesic.

\begin{proposition}\label{prop:union qg}
Suppose that $(S_h, \Lambda)$ is a hyperbolic metric on $S$ together
with a suited pair of measured laminations.  Then there are
constants $Q \ge 1$ and $c > 0$ such that for any non-exceptional
geodesic $\gamma$ intersecting an innermost polygon
$P = R_+ \cap R_-$, the test path over the union of the two
intersection intervals $\tau_\gamma(I_{R_+} \cup I_{R_-})$ is
$(Q, c)$-quasigeodesic.
\end{proposition}

\Cref{lemma:straight qg} is then an immediate consequence of
\Cref{lemma:intersection interval qg} and \Cref{prop:union qg}.

\medskip
We prove \Cref{lemma:intersection interval qg} first.  By \Cref{prop:short segments}, if $\gamma(I_R)$ has bounded length then $\tau_\gamma(I_R)$ also has bounded length. So it suffices to show that
test paths over sufficiently long intersection intervals are quasigeodesic.

\begin{lemma}\label{lemma:long segments}
Suppose that $(S_h, \Lambda)$ is a hyperbolic metric on $S$ together
with a suited pair of measured laminations.  Then there are
constants $L_R> 0, Q > 0$ and $c \ge 0$ such that for any ideal
complementary region $R$ of an invariant lamination, and any
non-exceptional geodesic $\gamma$ that intersects $R$ in a segment of
length at least $L_R$, the test path restricted to the intersection
interval $I_R$ is $(Q, c)$-quasigeodesic in
$\widetilde S_h \times \RR$.
\end{lemma}

The $R$ subscript for $L_R$ is to distinguish the constant in
\Cref{lemma:long segments} from similar constants in other parts of
the paper. Since there are finitely many complementary regions, the
order of quantifiers in \Cref{lemma:long segments} is correct, and
$L_R$ depends only on $(S_h, \LL)$.  We shall choose $L_R$ to be the
constant $L$ from \Cref{prop:long segment sign}, which implies that
the sign of the height function does not change on $I_R$.

\medskip

An ideal complementary region contains finitely many boundary leaves and extended leaves.  We will show that the height function for points of $\gamma(I_R)$ in an ideal
complementary region $R$ depends only on the distances to these leaves. To be precise, we show that up to
bounded error, the height function is determined by distance to a
single leaf of the extended lamination contained in that region.  Breaking symmetry, we give the argument for an ideal complementary region $R$ of $\Lambda_+$. The same holds for $\Lambda_-$ with the sign of the height function reversed.

\begin{definition}
Suppose that $S_h$ a hyperbolic metric and $\LL$ a suited pair of
laminations.  Given a constant $K \ge 0$, an ideal polygon $R$ in
$\widetilde{S}_h \setminus \Lambda_+$, a non-exceptional geodesic
$\gamma$ intersecting $R$, we say that a leaf $\ell$ of the extended
lamination $\overline{\Lambda}_+$ is \emph{$K$-dominant} on the
intersection interval $I_R$ if
\begin{itemize}
    \item $\ell$ is contained in $R$, and
    \item up to an additive error of $K$, the radius function
    $\rho_\gamma(t)$ on the intersection interval $I_R$ equals the
    radius function $\rho_{\gamma, \ell}(t)$ for the leaf $\ell$.
\end{itemize}
\end{definition}

\begin{proposition}\label{prop:dominant}
Suppose that $(S_h, \Lambda)$ is a hyperbolic metric on $S$ together
with a suited pair of measured laminations.  Then there are
constants $L > 0, K > 0$ such that for any ideal polygon $R$ in
$\widetilde{S}_h \setminus \Lambda_+$ and any non-exceptional geodesic
$\gamma$ intersecting $R$ in a segment of length at least $L$, there
is a leaf $\ell$ of the extended lamination $\overline{\Lambda}_+$
such that $\ell$ is $K$-dominant on the intersection interval $I_R$.
\end{proposition}

We first show that \Cref{lemma:long segments} follows from
\Cref{prop:dominant}.

\begin{proof}
Suppose that $\gamma$ intersects an ideal polygon $R$ in $\widetilde{S}_h \setminus \Lambda_+$.
By \Cref{prop:dominant}, there is a leaf $\ell$ of the
extended lamination $\overline{\Lambda}_+$ such that $\ell$ is
$K$-dominant on $I_R$.
As $\gamma(I_R)$ lies in a complementary region, the pseudo-metric
distance from $\gamma$ to $\ell$ is zero.

\medskip
As $\ell \in \overline{\Lambda}_+$, the radius function for $\ell$ is
\[ \rho_{\gamma, \ell}(t) = \log \frac{1}{d_{\PSL(2,
    \RR)}(\gamma^1(t), \ell^1)}.  \]

Let $\gamma(t_\ell)$ be the closest point on $\gamma$ to $\ell$, and
assume that the closest distance in $\PSL(2, \RR)$ is $\theta_\ell$.
Since the height function vanishes if the angle is greater than
$\theta_\LL$ and since we have chosen $\theta_f \le \theta_0$ in \Cref{def:theta}, \Cref{prop:fellow travel} applies and there is a constant $K$ such
that
\[ \log \tfrac{1}{\theta_\ell} - | t - t_\ell | - K \le \rho_{\gamma,
    \ell}(t) \le \log \tfrac{1}{\theta_\ell} - | t - t_\ell | + K. \]
Thus the radius function lives in a vertical $K$-neighborhood of an
absolute value function, and hence the test path is
$(Q, c)$-quasigeodesic by \Cref{prop:K-neighborhood}, where $Q$ and
$c$ depend only on $K$, as required.
\end{proof}

\medskip

We now show that if the projection intervals of two leaves to $\gamma$
are coarsely nested then the radius function
of the leaf with the longer interval is a coarse upper
bound for the radius function of the other leaf.

\begin{proposition}\label{prop:comparing height functions}
Suppose that $(S_h, \Lambda)$ is a hyperbolic metric on $S$ together
with a suited pair of measured laminations.  Given a constant
$K > 0$ there is a constant $L > 0$ such that for any three geodesics
$\gamma, \ell_1$ and $\ell_2$ with distinct endpoints, such that the
projection intervals $I_{\ell_1}$ and $I_{\ell_2}$ for $\ell_1$ and
$\ell_2$ onto $\gamma$ are coarsely nested, i.e.
$I_{\ell_2} \subseteq N_{K}(I_{\ell_1})$, then the radius function
$\rho_{\gamma, \ell_1}$ determined by $\ell_1$ is a coarse upper bound
for the radius function $\rho_{\gamma, \ell_2}$ determined by
$\ell_2$, i.e. for all $t$,
%
%\[ \log \frac{1}{d_{\PSL(2, \RR)}( \gamma(t), \ell_2 )} \le \log
%\frac{1}{d_{\PSL(2, \RR)}( \gamma(t), \ell_1 )} + K_2.  \]
%
\[ \rho_{\gamma, \ell_2}(t) \le \rho_{\gamma, \ell_1}(t) + L.  \]
\end{proposition}

\begin{proof}
Suppose that the smallest $\PSL(2, \RR)$-distance between $\gamma$ and
$\ell_1$ is $\theta_1 > 0$, and assume that this occurs at time $t_1$,
the midpoint of $I_{\ell_1}$.  Similarly, suppose that the smallest
distance between $\gamma$ and $\ell_2$ is $\theta_2$, and assume that
this occurs at $t_2$, the midpoint of $I_{\ell_2}$.

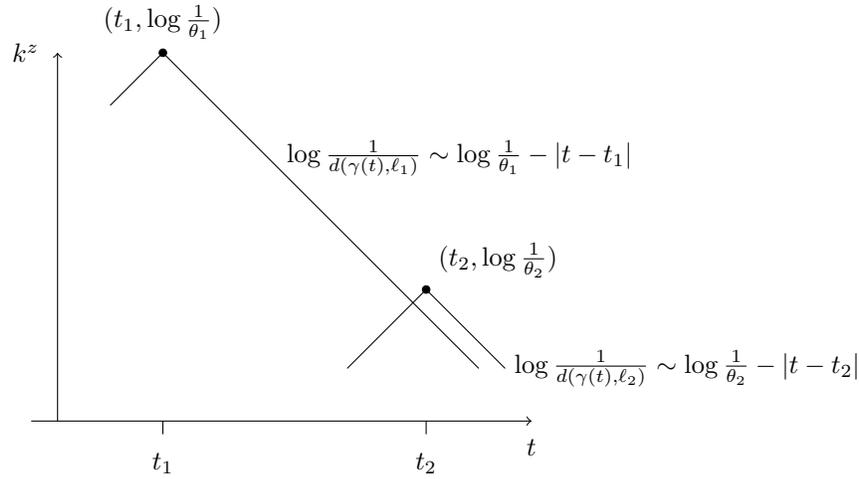
\begin{figure}[h]
\begin{center}
\begin{tikzpicture}[scale=0.7]

\tikzstyle{point}=[circle, draw, fill=black, inner sep=1pt]

\draw (-3, 5) -- (-2, 6) node [point,
label=above:${(t_1, \log \frac{1}{\theta_1})}$] {} -- (2, 2) node
[midway,
label=right:${\log \frac{1}{d(\gamma(t), \ell_1)}} \sim \log
\frac{1}{\theta_1} - |t - t_1|$] {} -- (4, 0);

\draw (1.5, 0) -- (3, 1.5) node [point,
label=above right:${(t_2, \log \frac{1}{\theta_2})}$] {} -- (4.5, 0) node [right]
{${\log \frac{1}{d(\gamma(t), \ell_2)}} \sim \log
\frac{1}{\theta_2} - |t - t_2|$};

\draw [arrows=->] (-4.5, -1) -- (5, -1) node [label=below:$t$] {};
\draw [arrows=->] (-4, -1) -- (-4, 6) node [label=left:$k^z$] {};

\draw (-2, -1) -- (-2, -1.25) node [label=below:$t_1$] {};

\draw (3, -1) -- (3, -1.25) node [label=below:$t_2$] {};

\end{tikzpicture}
\end{center}
\caption{The radius functions for truncated projection
  intervals with close endpoints.} \label{fig:comparing height
  functions}
\end{figure}

\medskip

Recall that the exponential interval
$E_{\ell_i} = [t_i - \log \tfrac{1}{\theta_i}, t_i + \log
\tfrac{1}{\theta_i}]$ is contained in the projection interval
$I_{\ell_i}$, and furthermore, there is a constant $K_1$ such that
$I_{\ell_i} \subseteq N_{K_1}(E_{\ell_i})$, where $K_1$ is the
constant $T_0$ from \Cref{prop:projection interval}.  So if the
projection intervals are coarsely nested,
$I_{\ell_2} \subseteq N_{K}( I_{\ell_1} )$, then the exponential
intervals are also coarsely nested,
$E_{\ell_2} \subseteq N_{K + K_1}( E_{\ell_1} )$.

\medskip
From \eqref{eq:interval abs}, the exponential height function for a
leaf is equal to an absolute value function, up to additive error
$K_2$. Thus, for $i = 1, 2$, we have for $t \in E_{\ell_i}$,
\[ \log \tfrac{1}{\theta_i} - | t - t_i | - K_2 \le \rho_{\gamma,
    \ell_i}(t) %= \log \frac{1}{d_{\PSL(2, \RR)}(\gamma^1(t), \ell^1_i )}
\le \log \tfrac{1}{\theta_i} - | t - t_i | + K_2,  \]
and for $t \not \in E_{\ell_i}$, the corresponding radius function is
bounded $\rho_{\gamma, \ell_i}(t) \le K_2$.

\medskip
If $t \not \in E_{\ell_1}$, then $t$ is within distance $K + K_1$ of
an endpoint of $E_{\ell_2}$, so
$\rho_{\gamma, \ell_2}(t) \le K + K_1 + K_2$.  As
$\rho_{\gamma, \ell_1}(t) > 0$ for all $t$, this implies
$\rho_{\gamma, \ell_2}(t) \le \rho_{\gamma, \ell_1}(t) + K + K_1 +
K_2$.

\medskip
We now consider points $t \in E_{\ell_1}$.  Consider two absolute
value functions $|\cdot|_{I_a} \colon I_a \to \RR$, with
$I_a = [t_a - a, t_a +a]$ and $|t|_{I_a} = a - |t_a - t|$ and
$|\cdot|_{I_b} \colon I_b \to \RR$, with $I_b = [t_b - b, t_b + b]$
and $|t|_{I_b} = b - |t_b - t|$.  If the first interval is contained
in the second, $I_a \subseteq I_b$, then $|t|_{I_a} \le |t|_{I_b}$.
If the first interval is coarsely contained in the second,
$I_a \subseteq N_{K + K_1}(I_b)$, then
$|t|_{I_a} \le |t|_{I_b} + K + K_1$.  It immediately follows that
\[ \rho_{\gamma, \ell_2}(t) \le \rho_{\gamma, \ell_1}(t) + K +
K_1 + K_2. \]
The result then follows choosing $L = K + K_1 + K_2$, which only
depends on $L$, and constants depending on the geometry of
$\PSL(2, \RR)$, as required.
\end{proof}

We now show that if $I_R$ is sufficiently
long, then the radius function along $\gamma(I_R)$ is equal to the radius function with respect
to a single extended leaf in $R$, up to bounded additive error.

\begin{proposition}\label{prop:single leaf}
Suppose that $(S_h, \Lambda)$ is a hyperbolic metric on $S$ together
with a suited pair of measured laminations.  Then there are
constants $L_R > 0$ (the constant from \Cref{prop:long segment sign})
and $K \ge 0$ such that for any ideal polygon $R$ in
$\widetilde{S}_h \setminus \Lambda_+$, and any non-exceptional
geodesic $\gamma$ such that the intersection interval $I_R$ has length
at least $L_R$, there is a single leaf $\ell$ of the extended
lamination $\overline{\Lambda}_+$ contained in $R$, such that for all
$t \in I_R$
\[ | \rho_\gamma(t) - \rho_{\gamma, \ell}(t) | \le K.  \]
\end{proposition}

\begin{proof}
Up to swapping laminations, we may assume that $R$ is an ideal complementary region of $\Lambda_+$.

\medskip 
Suppose that $\gamma$ is a non-exceptional geodesic that intersects $R$ and the length of $\gamma(I_R)$ is at  least $L_R$.  By \Cref{prop:long segment sign}, $h_\gamma(t)$ is always non-negative along $I_R$.  By definition of the height function, $\rho_{\gamma, \overline{\Lambda}_-}(t) \le 1 + \log
\tfrac{1}{\theta_\LL}$ for all $t \in I_R$, and so up to bounded
additive error,
$\rho_{\gamma}(t) = \rho_{\gamma, \overline{\Lambda}_+}(t)$ for
all $t \in I_R$.

\medskip

We will now identify which leaf $\ell$ of the extended lamination is $K$-dominant.  For any geodesic $\ell$ in $\HH^2$,
the nearest point projection of $\ell$ to $\gamma$ equals the segment of $\gamma$ bounded by the projections of
the ideal points of $\ell$.  Denote by $p_i$ the finitely many ideal
points of $R$, and by $z_i$ their closest point projections on
$\gamma$.  Suppose that $z_i$ and $z_j$ are the outermost points, that
is, all other $z_u$ are contained in the segment $[z_i, z_j]$ of
$\gamma$.  We set $\ell$ to be the leaf of the extended lamination
$\overline{\Lambda}_+$ connecting $p_i$ to $p_j$, which may be a
boundary leaf of $R$.

\medskip

We now show that for any point $t$ in $I_R$, the height function $h_\gamma(t)$ equals up to a bounded
additive error, the height function $h_{\gamma, \ell}(t)$ determined
by this leaf.

\medskip
Let $\ell_1$ and $\ell_2$ be the first and the last boundary leaves of $R$ that $\gamma$ intersects. We have illustrated this in \Cref{fig:dominant
  leaf} in the case where the leaves $\ell, \ell_1$ and
$\ell_2$ are all distinct.  It may be that one of $\ell_1$ or
$\ell_2$ equals $\ell$, but this makes no difference to the
argument.

\begin{figure}[h]
\begin{center}
\begin{tikzpicture}[scale=0.75]

\tikzstyle{point}=[circle, draw, fill=black, inner sep=1pt]

\def\boundary{(0, 0) circle (4)}

\begin{scope}
\clip \boundary;
\draw [color=ForestGreen] (0:5.66) circle (4);
\draw [color=ForestGreen] (90:5.66) circle (4);
\draw [color=ForestGreen] (180:5.66) circle (4);
\draw [color=ForestGreen] (270:5.66) circle (4);

\draw [color=ForestGreen] (180:4.27) circle (1.5);
\draw [color=ForestGreen] (67.5:4.27) circle (1.5);
\end{scope}

\draw [color=ForestGreen] (0, -2) node {$\ell_1$};
\draw [color=ForestGreen] (0, 2) node {$\ell_2$};
\draw [color=ForestGreen] (-3.25, 0) node {$\ell_3$};
\draw [color=ForestGreen] (0, 3.25) node {$\ell_4$};

\draw (-1.6, 1.6) node [color=ForestGreen] {$R$};

\draw [color=ForestGreen] (225:4) -- (45:4) node [midway, label=left:$\ell$] {};

\draw (230:4) -- (50:4) node [midway, label=right:$\gamma$] {};

\draw \boundary;

\end{tikzpicture}
\end{center}
\caption{A non-exceptional geodesic intersecting an ideal
  polygon.} \label{fig:dominant leaf}
\end{figure}
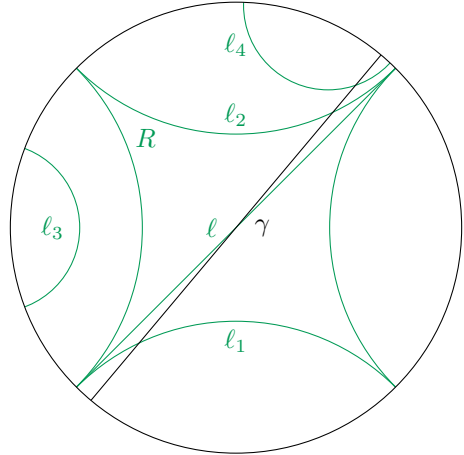

\medskip First consider a leaf $\ell'$ of the extended lamination
which is either a boundary leaf of $R$, or contained in $R$.  By our
choice of $\ell$, the nearest point projection of $\ell'$ to $\gamma$
is contained in the nearest point projection of $\ell$ to $\gamma$.
By \Cref{prop:comparing height functions},
\[ \rho_{\gamma, \ell'}(t) \le \rho_{\gamma, \ell}(t) + L, \]
where $L$ is the constant from \Cref{prop:comparing height functions}
with $K = 0$. Note that this same argument works for any leaf (such as
$\ell_3$ in \Cref{fig:dominant leaf}) of the extended lamination
$\overline{\Lambda}_+$ contained in a component of
$\widetilde{S}_h \setminus R$ disjoint from $\gamma$.

\medskip

Finally, suppose $\ell_4$ is a leaf of $\overline{\Lambda}_+$ in one of the two components of $\widetilde{S}_h \setminus R$
that $\gamma$ intersects. The leaf $\ell_4$ itself may or may not
intersect $\gamma$.  Breaking symmetry, suppose that $\ell_4$ lies
in the component of $\widetilde{S}_h \setminus R$ with boundary
$\ell_2$, as in \Cref{fig:dominant leaf} above.  
The same argument holds for the component with boundary $\ell_1$.
  
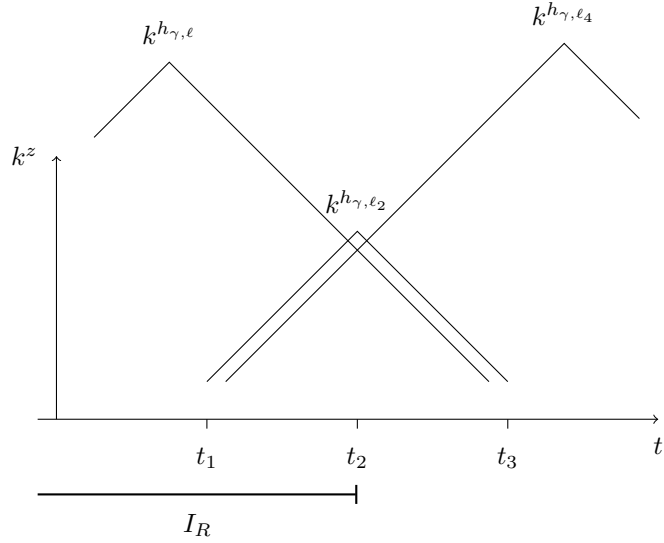
\begin{figure}[h]
\begin{center}
\begin{tikzpicture}[scale=0.5]

\tikzstyle{point}=[circle, draw, fill=black, inner sep=1pt]

\draw (-5, 6.5) -- (-3, 8.5) node [ label=above:$k^{h_{\gamma,
    \ell}}$] {} -- (5.5, 0);

\draw (-2, 0) -- (2, 4) node [ label=above:$k^{h_{\gamma, \ell_2}}$]
{} -- (6, 0);

\draw (-1.5, 0) -- (7.5, 9) node [ label=above:$k^{h_{\gamma, \ell_4}}$] {} -- (9.5, 7);

\draw [arrows=->] (-6.5, -1) -- (10, -1) node [label=below:$t$] {};
\draw [arrows=->] (-6, -1) -- (-6, 6) node [label=left:$k^z$] {};

\draw (-2, -1) -- (-2, -1.25) node [label=below:$t_1$] {};

\draw (2, -1) -- (2, -1.25) node [label=below:$t_2$] {};

\draw (6, -1) -- (6, -1.25) node [label=below:$t_3$] {};

\draw [thick, arrows=-|] (-6.5, -3) -- (2, -3) node [midway,
label=below:${I_R}$] {};

\end{tikzpicture}
\end{center}
\caption{Radius functions for leaves with overlapping projection intervals.} \label{fig:overlaps}
\end{figure}

\medskip
Let $\gamma(t_2)$ be the intersection point of $\gamma$ with the
boundary leaf $\ell_2$. Note that
\begin{itemize}
     \item $\gamma(t_2)$ is the endpoint of the intersection interval $I_R$, and
     \item $\gamma(t_2)$ is the midpoint of the nearest point projection interval $I_{\ell_2}$ of $\ell_2$ onto $\gamma$.
\end{itemize}
Let $\gamma(t_1)$ and $\gamma(t_3)$ be the initial and terminal points of $I_{\ell_2}$.
Let $I_{\ell}$ and $I_{\ell_4}$ be the nearest point projection intervals onto $\gamma$ for the leaves
$\ell$ and $\ell_4$. 
Then $I_{\ell_2} \subseteq I_{\ell}$, with common endpoint $\gamma(t_3)$,
and $I_{\ell_4} \subseteq I_{\ell_2}$, with common initial point
$\gamma(t_1)$.  This is illustrated in \Cref{fig:overlaps}.

\medskip
As $I_{\ell_2} \subseteq I_{\ell}$, by \Cref{prop:comparing height
  functions} the radius function $\rho_{\gamma, \ell}$ is a coarse
upper bound for the radius function $\rho_{\gamma, \ell_2}$. As the
initial point of $I_{\ell_4}$ is further along $\gamma$ than
$\gamma(t_1)$, the initial point of $I_{\ell_2}$, the radius function
$\rho_{\gamma, \ell_2}$ is a coarse upper bound for
$\rho_{\gamma, \ell_4}$ on $[t_1, t_2]$, the first half of
$I_{\ell_2} = [t_1, t_3]$ on which both $\rho_{\gamma, \ell_2}$ and
$\rho_{\gamma, \ell_4}$ are increasing.  As $t_2$ is the endpoint of
$\gamma(I_R)$, we deduce that $\rho_{\gamma, \ell}$ is a coarse upper
bound on $I_R$ for the radius functions
corresponding to all leaves of the extended lamination
$\overline{\Lambda}_+$, as required.
\end{proof}

\Cref{prop:dominant} now follows directly from \Cref{prop:single
  leaf}.

\medskip
Finally, we prove \Cref{prop:union qg}, that the test path over the
union of two intersections intervals meeting in a non-rectangular
polygon is also quasigeodesic.

\begin{proof}{Proof (of \Cref{prop:union qg})}
Let $p_1$ be the endpoint of $I_R$ in the innermost polygon $P$.  Let
$p_0$ be the other endpoint of $I_R$, and let $p_2$ be the endpoint of
$I_{R'}$ disjoint from $I_R$.  As both $\tau_\gamma(I_R)$ and
$\tau_\gamma(I_{R'} \setminus I_R)$ are quasigeodesic, it suffices to
prove that the Gromov product $(p_0, p_2)_{p_1}$ is bounded.

\medskip By \Cref{prop:dominant}, there is an extended leaf $\ell_+$
in $R_+$ that is $K$-dominant for $\gamma(I_{R_+})$.  In particular,
the test path over $\gamma_{I_{R_+}}$ lies in a bounded neighborhood
of a half space in the ladder $F(\ell_+)$.  Similarly, the
complementary region of $\Lambda_-$ containing the cusp $C$ has a
segment of an extended leaf $\ell_-$ that is $K$-dominant for
$\gamma(I_C)$.  Thus the test path over $\gamma(I_R)$ lies in a
bounded neighborhood of a half space in the ladder $F(\ell_-)$.  Let
$q$ be the intersection point of the two extended leaves $\ell_+$ and
$\ell_-$.  Both $p$ and $q_1$ lie in $P$ and so are a bounded distance
apart.  The union of the two quasigeodesics is thus contained in a
bounded neighborhood of the union of two half spaces meeting along the
suspension flow line through $q$.  Each half space is quasi-isometric
to a half space in $\HH^2$ with geodesic boundary, so the union of the
two half spaces is quasi-isometric to $\HH^2$.  The nearest point
projection of the first quasigeodesic to suspension flow line $F(q)$
is contained in a bounded neighborhood of the positive half of $F(q)$,
and the nearest point projection of the second quasigeodesic is
contained in a bounded neighborhood of the negative part of $F(q)$.
Therefore the Gromov product of the endpoints is bounded, and so the
union of the two quasigeodesics is a quasigeodesic.
\end{proof}

%%%%%%%%%%%%%%%%%%%%%%%%%%%%%%%%%%%%%%%%%%%%%%%%%%%%%%%%%%%%%%%%%%%%%%%%%%%%%%%%%%
\subsection{Vertical projection to the test path is distance decreasing}\label{section:proj}
%%%%%%%%%%%%%%%%%%%%%%%%%%%%%%%%%%%%%%%%%%%%%%%%%%%%%%%%%%%%%%%%%%%%%%%%%%%%%%%%%%

Finally, we combine the fact that the test path is quasigeodesic with
the fact that the height function is Lipschitz, to show that the
vertical projection from $\iota(\gamma)$ to $\tau_\gamma$ is coarsely
distance decreasing.

\begin{prop:vertical flow distance decreasing}
Suppose that $(S_h, \Lambda)$ is a hyperbolic metric on $S$ together
with a suited pair of measured laminations.  There are constants
$K$ and $c$ such that for any non-exceptional geodesic $\gamma$ with
unit speed parametrization, and any $s$ and $t$,
\[ d_{\wsr}( \tau_\gamma(s), \tau_\gamma(t) ) \le K d_{\wsr} \left(
\iota(\gamma(s)), \iota(\gamma(t)) \right) + c , \]
where here the test path has the parametrization inherited from the
unit speed parametrization on $\gamma$.
\end{prop:vertical flow distance decreasing}

We remark that we do not show that the nearest point projection of
$\iota(\gamma(t))$ to the geodesic $\ogamma$ in $\wsr$ is close to the
corresponding test path point $\tau_\gamma(t)$.  This is equivalent to
the statement that there is an upper bound on the length of the fellow
traveling interval between the vertical flow lines from
$\iota(\gamma(t))$ to $\tau_\gamma(t)$ and the geodesic $\ogamma$.
Although this may be true for our choice of test path, this property
does not hold for all quasigeodesic paths with the same endpoints as
$\ogamma$, and so depends on the exact choice of quasigeodesic.

\begin{proof}
Let $\gamma$ be a non-exceptional unit speed geodesic in $\ws_h$, and let $\tau_\gamma(t)$ be the test path with the
(non-unit speed) parametrization determined by $\gamma(t)$.  Let
$\ogamma$ be the geodesic in $\wsr$ determined by $\gamma$.

\medskip
Let $k > 1$ be the constant from the definition of the Cannon--Thurston
metric, and let $\delta_3$ be the constant of hyperbolicity for
$\wsr$.  We will choose $K = 1 + 2 / \log k$, and
$c = 12 \delta_3 + 6 L$.  Here $L$ is the Morse constant such that the
test path $\tau_\gamma$ is contained in an $L$-neighborhood of the
geodesic $\ogamma$. For notational convenience, set
$D = d_{\wsr}(\iota(\gamma(t_1)), \iota(\gamma(t_2)))$.

\medskip
Let $p_i$ be the nearest point projection of the test path location
$\tau_\gamma(t_i)$ to $\ogamma$ in $\wsr$, and let $q_i$ be the
nearest point projection of $\iota(\gamma(t_i))$ to $\ogamma$.

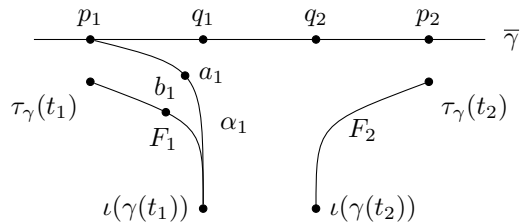
\begin{figure}[h]
\begin{center}
\begin{tikzpicture}[scale=0.75]

\tikzstyle{point}=[circle, draw, fill=black, inner sep=1pt]

\draw (0, 1) -- (8,1) node [label=right:$\overline{\gamma}$] {};

\draw (1, 0.25) node (a) [point, label=below left:$\tau_\gamma(t_1)$] {};

\draw (1, 1) node (c) [point, label=above:$p_1$] {};

\draw (3, 1) node [point, label=above:$q_1$] {};

\draw (7, 0.25) node (b) [point, label=below right:$\tau_\gamma(t_2)$] {};

\draw (7, 1) node [point, label=above:$p_2$] {};

\draw (5, 1) node [point, label=above:$q_2$] {};

\draw (3, -2) node (d) [point, label=left:$\iota(\gamma(t_1))$] {};

\draw (d) .. controls (3, 0.5) and (3, 0.5) .. (c) node [pos=0.25,
label=right:$\alpha_1$] {} node [pos=0.55, point, label=right:$a_1$] {};

\draw (a) .. controls (3, -0.5) and (3, -0.5) .. (d) node [pos=0.6,
label=left:$F_1$] {} node [pos=0.3, point, label=above:$b_1$] {};

\draw (b) .. controls (5, -0.5) and (5, -0.5) ..  (5, -2) node [midway,
label=right:$F_2$] {} node [point, label=right:$\iota(\gamma(t_2))$] {} ;

\end{tikzpicture}
\end{center}
\caption{The geodesic $\overline{\gamma}$ and the vertical flow
  lines.} \label{fig:two vertical flow line}
\end{figure}

We denote by $F_i$ the vertical flow segment from
$F_0(\iota(\gamma(t_i)))$ to
$\tau_\gamma(t_i) = F_{h_\gamma(t_i)}(\iota(\gamma(t_i)))$.  The path
$\alpha_i$ consisting of the union of the two geodesics
$[p_i, q_i] \cup [q_i , \iota(\gamma(t_i))]$ is close to being a
geodesic.  In particular, by
\cite{kapovich-sardar}*{Lemma 1.102} there is a point $a_i$ on $\alpha_i$
distance at most $2 \delta_3$ from $q_i$.  By thin triangles and the distance from $\tau_\gamma(t_i)$ to $p_i$ being at most $L$,
there is a point $b_i$ on $F_i$ within distance
$L + 3 \delta_3$ of $q_i$.

\medskip
Suppose that $q_1$ lies between $p_1$ and $p_2$. As $\tau_\gamma$ is
contained in an $L$-neighborhood of $\ogamma$, there is a point
$\tau_\gamma(t)$ distance at most $L$ from $q_1$,
with $t_1 \le t \le t_2$.  As the height function is
$(1 / \log k)$-Lipschitz, the height difference between
$\tau_\gamma(t_1)$ and $\tau_\gamma(t)$ is at most $D / \log k$, and
so the height difference between $p_1$ and $q_1$ is at most
$D / \log k + 2 L$.

\medskip
By the same argument from the previous two paragraphs applied to $F_2$,
if $q_2$ lies between $p_1$ and $p_2$, then the distance
between $p_2$ and $q_2$ is also at most $D / \log k + 2 L$.

\medskip
As $q_1$ is the nearest point projection of $\iota(\gamma(t_1))$ to
$\ogamma$, the path consisting of the concatenation of the three
geodesics
$[\iota(\gamma(t_1)), q_1] \cup [q_1 , q_2] \cup [q_2,
\iota(\gamma(t_2))]$ is close to being a geodesic.  By
\cite{kapovich-sardar}*{Lemma 1.120}, if the distance between $q_1$
and $q_2$ is at least $8 \delta_3$, then the distance between
$\iota(\gamma(t_1))$ and $\iota(\gamma(t_2))$ is at least
$\left| [q_1 , q_2] \right| - 12 \delta_3$.

\medskip
The distance in $\wsr$ between $\iota(\gamma(t_n))$ and
$\iota(\gamma(t_{n+1}))$ is at most $D$, so the distance between $q_1$
and $q_2$ is at most $D + 12 \delta_3$.  This implies that the
distance between $p_1$ and $p_2$ is at most
$D + 12 \delta_3 + 2 D / \log k + 4L$, and so the distance between
$\tau_\gamma(t_1)$ and $\tau_\gamma(t_2)$ is at most
$(1 + 2 / \log k) D + 12 \delta_3 + 6L$, as required.
\end{proof}

\newpage
%%%%%%%%%%%%%%%%%%%%%%%%%%%%%%%%%%%%%%%%%%%%%%%%%%%%%%%%%%%%%%%%%%%%%%%%%%%%%%%%%%
\begin{appendices}
%%%%%%%%%%%%%%%%%%%%%%%%%%%%%%%%%%%%%%%%%%%%%%%%%%%%%%%%%%%%%%%%%%%%%%%%%%%%%%%%%%

%%%%%%%%%%%%%%%%%%%%%%%%%%%%%%%%%%%%%%%%%%%%%%%%%%%%%%%%%%%%%%%%%%%%%%%%%%%%%%%%%%
\section{Distance bounds in \texorpdfstring{$\HH^2$}{H2}}\label{section:hyperbolic}
%%%%%%%%%%%%%%%%%%%%%%%%%%%%%%%%%%%%%%%%%%%%%%%%%%%%%%%%%%%%%%%%%%%%%%%%%%%%%%%%%%

In this section we record some standard estimates on distances between
geodesics in the hyperbolic plane $\HH^2$.  We start by finding bounds
on the distances in $\HH^2$ between two non-intersecting geodesics
distance $\theta$ apart.

\begin{proposition}\label{prop:hyperbolic bounds disjoint}
Suppose that $\theta \le 1$ is a positive constant and suppose that $\gamma_1$ and $\gamma_2$ are two bi-infinite geodesics in $\HH^2$ which do not
intersect and are distance $\theta$ apart.  Suppose that $\gamma_1$ is parametrized with unit speed such that $\gamma_1(0)$ is the closest point
on $\gamma_1$ to $\gamma_2$.  Then
\[ \tfrac{1}{3} \theta e^{|t|} \le d_{\HH^2}(\gamma_1(t), \gamma_2)
\le \tfrac{3}{2} \theta e^{|t|}, \text{ if } |t| \le \log
\tfrac{1}{\theta}. \]
Furthermore, the lower bound at $|t| = \log \tfrac{1}{\theta}$ holds
for all $|t| \ge \log \tfrac{1}{\theta}$, i.e.
$d_{\HH^2}(\gamma_1(t), \gamma_2) \ge \tfrac{1}{3}$.
\end{proposition}

\begin{proof}
Choose unit speed parametrizations of $\gamma_1$ and $\gamma_2$ so
that their closest points are $\gamma_1(0)$ and $\gamma_2(0)$.  Let
$p$ be the closest point on $\gamma_2$ to $\gamma_1(t)$, and set
$d = d_{\HH^2}(\gamma_1(t), \gamma_2) = d_{\HH^2}(\gamma_1(t), p)$.
As the distances are symmetric for $t$ and $-t$, it suffices to
consider the case $t \ge 0$.

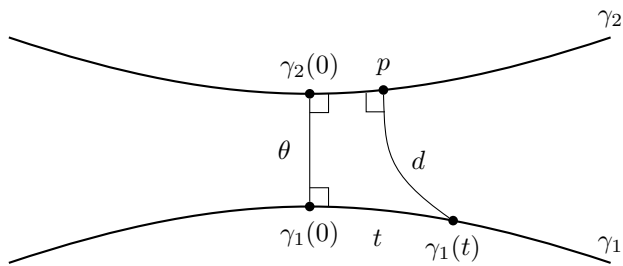
\begin{figure}[h]
\begin{center}
\begin{tikzpicture}%[scale=0.9]

\tikzstyle{point}=[circle, draw, fill=black, inner sep=1pt]

\draw [thick] (0, 0) .. controls (3, 1) and (5, 1) .. (8, 0) node (a)
[midway, point, label=below:$\gamma_1(0)$] {} node [pos=0.62, label=below:$t$] {} node (c)
[pos=0.75, point, label=below:$\gamma_1(t)$] {} node [above] {$\gamma_1$}; 

\draw [thick] (0, 3) .. controls (3, 2) and (5, 2) ..  (8, 3) node (b)
[midway, point, label=above:$\gamma_2(0)$] {} node [above]
{$\gamma_2$} node (d) [pos=0.63, point, label=above:$p$] {};

\draw (a) -- (b) node [midway, label=left:$\theta$] {};

\draw (c) .. controls (5, 1.2) and (5, 1.5) .. (d) node [midway,
label=right:$d$] {};

\begin{scope}[yshift=0.75cm, xshift=3cm]
\draw (1.25, 0) -- (1.25, 0.25) -- (1, 0.25);
\end{scope}

\begin{scope}[yshift=2.25cm, yscale=-1, xshift=3cm]
\draw (1.25, 0) -- (1.25, 0.25) -- (1, 0.25);
\end{scope}

\begin{scope}[xshift=6cm, yshift=2.26cm, scale=-1]
\draw (1.25, 0) -- (1.25, 0.25) -- (1, 0.25);
\end{scope}

\end{tikzpicture}
\end{center}
\caption{The geodesic $\gamma_3$ connecting non-adjacent limit points
  of $\gamma_1$ and $\gamma_2$.} \label{fig:gamma3}
\end{figure}

The four points $\gamma_1(0), \gamma_2(0), \gamma_1(t)$ and $p$
determine a hyperbolic quadrilateral with three right angles, which is
known is as a Lambert quadrilateral, and its sides satisfy
\begin{equation}\label{eq:lambert}
\sinh d = \cosh t \sinh \theta,
\end{equation}
see for example \cite{martin}*{Section 32.2}.  We will use the
following elementary estimates: $x \le \sinh x \le \tfrac{3}{2}x$ for
$0 \le x \le \sinh(1) < 1.18$, and $\tfrac{1}{2} e^x \le \cosh x \le e^x$ for all
values of $x$.  Applying these estimates for $0 \le \theta \le 1$, and
$0 \le d \le \sinh(1)$ gives
\[ \tfrac{1}{3} \theta e^t \le d \le \tfrac{3}{2} \theta e^t,  \]
for $|t| \le \log \tfrac{1}{\theta}$, as for these values of $t$,
$d \le \sinh(1)$.  Differentiating \eqref{eq:lambert} shows that for
fixed $\theta$, $d$ is increasing in $t$, and so for all
$t \ge \log \tfrac{1}{\theta}$, the lower bound at
$t = \log \tfrac{1}{\theta}$ holds, i.e.
$d_{\HH^2}(\gamma_1(t), \gamma_2) \ge \tfrac{1}{3}$, as required.
\end{proof}

We now find bounds on the distances between two geodesics in $\HH^2$
which intersect at angle $\theta$.

\begin{proposition}\label{prop:hyperbolic bounds intersect}
Suppose that $\gamma_1$ and $\gamma_2$ are geodesics in $\HH^2$ which
intersect at angle $\theta < 1$, and suppose that $\gamma_1$ has unit
speed parametrization so that the intersection point is
$\gamma_1(0)$.  Then
\[ \tfrac{1}{8} \theta (e^t - 1) \le d_{\HH^2}(\gamma_1(t), \gamma_2)
\le \tfrac{1}{2} \theta e^t, \text{ if } |t| \le \log \tfrac{1}{\theta}, \]
and furthermore, the lower bound at $t = \log \tfrac{1}{\theta}$ holds
for all $t \ge \log \tfrac{1}{\theta}$, i.e.
$d_{\HH^2}(\gamma_1(t), \gamma_2) \ge \tfrac{1}{8}( 1 - \theta)$.
\end{proposition}

\begin{proof}
Let $\gamma_1(t)$ be the point distance $t$ along $\gamma_1$ from the
intersection point, and let $d$ be the distance from $\gamma_1(t)$ to
$\gamma_2$.

\begin{figure}[h]
\begin{center}
\begin{tikzpicture}%[scale=0.9]

\tikzstyle{point}=[circle, draw, fill=black, inner sep=1pt]

\draw [thick] (2, 0) -- (8, 0) node (a)
[pos=0.33, point, label=below:$\gamma_1(0)$] {} node (a)
[pos=0.75, point, label=below:$\gamma_1(t)$] {} node [below] {$\gamma_1$}; 

\draw [thick] (2, -1) -- (8, 2) node [label=above:$\gamma_2$] {};

\draw (a) .. controls (6.2, 0.2) and (6, 0.25) .. (5.5, 0.75) node
[midway, label=above:$d$] {};

\begin{scope}[yshift=1.45cm, xshift=6.25cm]
\begin{scope}[rotate=220]
\draw (1.25, 0) -- (1.25, 0.25) -- (1, 0.25);
\end{scope}
\end{scope}

\end{tikzpicture}
\end{center}
\caption{The geodesics of $\gamma_1$ and $\gamma_2$ intersect at angle
  $\theta$.} \label{fig:intersecting}
\end{figure}
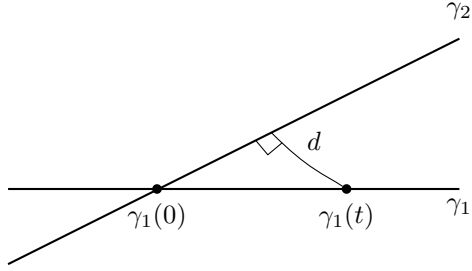

Using the sine formula for right angled triangles in hyperbolic space gives
\[ \sin \theta = \frac{ \sinh d   }{ \sinh t }.  \]

\medskip
Using the elementary estimates: $\tfrac{1}{2} \theta \le \sin \theta \le \theta$ for
$0 \le \theta \le 1$ and $ \tfrac{1}{2} ( e^t - 1 ) \le \sinh t \le \tfrac{1}{2} e^t$ for $t
\ge 0$, we get 

\[ \tfrac{1}{4} \theta (e^t - 1)  \le  \sinh d \le \tfrac{1}{2} \theta e^t    \]

\medskip
For $t \le \log \tfrac{1}{\theta}$, $\sinh d \le \tfrac{1}{2} < 1$, so have the elementary estimate $ d \le \sinh d \le 2d  $ for $0 \le d \le 1$, which gives

\[ \tfrac{1}{8} \theta (e^t - 1) \le d \le \tfrac{1}{2} \theta e^t  \]

as required.

\medskip
Finally, the bound $\sinh(d) \ge \tfrac{1}{4}(e^t - 1)$ holds for all
$t \ge 0$, and $e^t$ is increasing, so for all
$t \ge \log \tfrac{1}{\theta}$,
$\sinh d \ge \tfrac{1}{4}(1 - \theta)$.  The right hand side takes
values in $[0, \tfrac{1}{4}]$, so $d$ takes values in
$[0, \sinh^{-1}(\tfrac{1}{4}) < 1]$, and we may apply the elementary
bound for $\sinh d$, so $d \ge \tfrac{1}{8}(1 - \theta)$ for all
$t \ge \log \tfrac{1}{\theta}$.
\end{proof}

Finally, we prove the bounds on the size of nearest point projection
intervals.

\begin{proposition:pi}
There is a constant $T_0 > 0$ such that for any geodesics $\gamma_1$
and $\gamma_2$ in $\HH^2$ that
\begin{itemize}
    \item intersect at an angle $0 < \theta \le \pi/2$, and
    \item $\gamma_1$ is parametrized with unit speed so that $\gamma_1(0)$ is the point of intersection,
\end{itemize}
then the image of $\gamma_2$ under nearest point
projection to $\gamma_1$ equals $\gamma_1( [-T, T] )$, where
\[ \log \tfrac{1}{\theta} \le T \le \log \tfrac{1}{\theta} +
T_0. \]
\end{proposition:pi}

\begin{proof}
We may parametrize $\gamma_1$ and $\gamma_2$ with unit speed so that
they intersect at $\gamma_1(0) = \gamma_2(0)$.  Let
$\gamma_1(s)$ be the nearest point projection of $\gamma_2(t)$ to
$\gamma_1$. Up to reversing the parametrization for $\gamma_1$ we may
assume that $s \ge 0$ whenever $ t \ge 0$.  Then the three points
$\gamma_1(0), \gamma_2(t)$ and $\gamma_1(s)$ form a right angled
triangle, with right angle at $\gamma_1(s)$ and hypotenuse of length
$t$.  This is illustrated in \Cref{fig:nearest point projection}.

\begin{figure}[h]
\begin{center}
\begin{tikzpicture}%[scale=0.75]

\tikzstyle{point}=[circle, draw, fill=black, inner sep=1pt]

\draw (-1, 0) -- (5, 0) node [below] {$\gamma_1$};

\draw (-1, -0.5) -- (4, 2) node [above] {$\gamma_2$};

\draw (0, 0) node [point, label=above left:${\gamma_1(0) = \gamma_2(0)}$] {};

\draw (3, 1.5) node (a) [point, label=above left:$\gamma_2(t)$] {};

\draw (a) .. controls (2.75, 1) and (2.5, 0.5) .. (2.5, 0) node (b)
[point, label=below:$\gamma_1(s)$] {};

\draw ($(b)+(-0.25, 0)$) -- ($(b)+(-0.25, 0.25)$) -- ($(b)+(0.03, 0.25)$);

\draw ([shift=(0:1cm)]0,0) arc (0:26:1cm) node [midway, right] {$\theta$};

\end{tikzpicture}
\end{center}
\caption{The nearest point projection from one leaf to an intersecting
  leaf.} \label{fig:nearest point projection}
\end{figure}
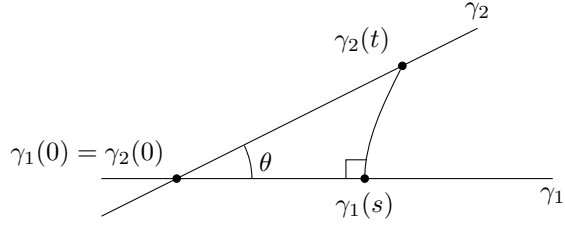

\medskip
For right angled triangles in
$\HH^2$, we have
\[ \cos \theta = \frac{ \tanh s}{ \tanh t}. \]
We will take the limit as $t$ tends to infinity, so we may replace
$\tanh(t)$ by one.  We will use the elementary bounds that for
$0 \le x \le \pi/2$, $1 - x^2/2 \le \cos x \le 1 - x^2 / 4$.  This
gives
\[ 1 - \tfrac{1}{2} \theta^2 \le \tanh s \le 1 - \tfrac{1}{4}
\theta^2. \]
Using the formula for
$\tanh^{-1}(x) = \tfrac{1}{2} \log \frac{1+x}{1-x}$, we get
\[ \tfrac{1}{2} \log \frac{2 - \tfrac{1}{2} \theta^2}{\tfrac{1}{2}
  \theta^2} \le s \le \tfrac{1}{2} \log \frac{ 2 - \tfrac{1}{4}
  \theta^2}{ \tfrac{1}{4} \theta^2 }. \]
\[ \tfrac{1}{2} \log \frac{4}{\theta^2} - 1 \le s \le \tfrac{1}{2} \log \frac{ 2 }{ \tfrac{1}{4} \theta^2 }. \]
\[ \tfrac{1}{2} \log \frac{3}{\theta^2}  \le s \le \tfrac{1}{2} \log
\frac{ 8 }{ \theta^2 }. \]
\[ \log \tfrac{1}{\theta} + \tfrac{1}{2} \log 3 \le s \le \log \tfrac{1}{\theta} + \tfrac{1}{2} \log
8. \]
As the constant on the left is positive, we can choose
$T_0 = \tfrac{1}{2} \log 8 \le 2$.
\end{proof}

%%%%%%%%%%%%%%%%%%%%%%%%%%%%%%%%%%%%%%%%%%%%%%%%%%%%%%%%%%%%%%%%%%%%%%%%%%%%%%%%%%
\section{Distance bounds in \texorpdfstring{$\PSL(2,\RR)$}{PSL(2,R)} }\label{section:fellow travel}
%%%%%%%%%%%%%%%%%%%%%%%%%%%%%%%%%%%%%%%%%%%%%%%%%%%%%%%%%%%%%%%%%%%%%%%%%%%%%%%%%%

The fellow traveling results in this section are standard, but we give
detailed proofs for the convenience of the reader.  All constants in
this section depend only on the geometry of $\PSL(2, \RR)$, and in
particular do not depend on a choice of hyperbolic surface $S_h$ or
the suited pair of laminations $\LL$.

\medskip
We abuse notation by using the same notation for geodesics in $\HH^2$
and their lifts in $\PSL(2, \RR)$.

% \begin{proposition:ft}%\label{prop:fellow travel}
% There are constants $\theta_0 > 0$ and $L_0 \ge 1$ such that for any
% geodesics $\gamma_1$ and $\gamma_2$ in $\HH^2$ 
% \begin{itemize}
%     \item  whose lifts in $\PSL(2, \RR)$  are distance $\theta \le \theta_0$ apart, and
%     \item the lift $\gamma_1$ has unit speed parametrization such that $\gamma_1(0)$ is the closest point to $\gamma_2$,
% \end{itemize} 
% then for any $t$ such that $|t| \le \log \frac{1}{\theta}$
% %
% \[ \tfrac{1}{L_0} \theta e^{|t|} \le d_{\PSL(2, \RR)} ( \gamma^1_1(t),
% \gamma^1_2 ) \le L_0 \theta e^{|t|}. \]
% %
% Furthermore, the lower bound at $|t| = \log \tfrac{1}{\theta}$ holds
% for all $|t| \ge \log \tfrac{1}{\theta}$.
% \end{proposition:ft}

For $t$ reasonably large, distances in $\HH^2$ give reasonable bounds
on distances in $\PSL(2, \RR)$. However, they are not precise enough
for our purposes close to $t = 0$.  For small $t$ we use the fact that
the left invariant metric is bilipschitz to the corresponding matrix
norm on $\PSL(2, \RR)$.

\begin{proposition}\label{prop:small t bounds}
There are constants $\theta_0 > 0$, $L \ge 1$ and
$K \le \log \tfrac{L}{\theta_0}$, such that for any two geodesics
$\gamma_1$ and $\gamma_2$ in $\HH^2$ 
\begin{itemize}
    \item whose lifts in $\PSL(2, \RR)$ are distance $\theta \le \theta_0$ apart, and
    \item the lift $\gamma_1(0)$ has unit speed parametrization such that $\gamma_1(0)$ is the closest point to $\gamma_2$,
\end{itemize}
then for all
$|t| \le \log \frac{1}{\theta} - K$,
\[ \tfrac{1}{L} \theta e^{|t|} \le d_{\PSL(2, \RR)} ( \gamma^1_1(t),
\gamma^1_2 ) \le L \theta e^{|t|}. \]
In particular, $\gamma_1(t)$ lies within a $\theta_0$-neighborhood of
$\gamma_2$ for a symmetric interval of length at least
$2(\log \frac{1}{\theta} - K)$ centered at $t = 0$.
\end{proposition}

We will use the fact that there is a neighborhood $U$ of the identity
matrix in $\PSL(2, \RR)$ such that any left invariant metric is
bilipschitz to the metric arising from any matrix norm, see
for example \cite{einsiedler-ward}*{Lemma 9.12}, where this is shown
for any Lie group.  In fact, this holds as long as $U$ is the image of
a bounded neighborhood $V$ of the zero matrix in the Lie algebra $\text{sl}(2, \RR)$
on which the exponential map is injective.

\begin{proposition}[\cite{einsiedler-ward}*{Lemma 9.12}]\label{prop:lipschitz}
There is a constant $\theta_0 > 0$ such that for any matrix norm
$\norm{A}$, and any left invariant metric $d_{\PSL(2, \RR)}(A, B)$ on
$\PSL(2, \RR)$, there is a constant $L_0 > 0$ such that for any matrix
$A \in \PSL(2, \RR)$, with $d_{\PSL(2, \RR)}(I, A) \le \theta_0$,
\[ \tfrac{1}{L_0} \norm{I - A} \le d_{\PSL(2, \RR)}(I, A) \le L_0 \norm{I
  - A}. \]
\end{proposition}

We shall use the matrix norm determined by the Frobenius or
Hilbert-Schmidt inner product
$\langle A, B \rangle = 2 \text{tr}(A^T B) $ on $2 \times 2$ matrices.
The scaling factor of $2$ is chosen so that the following basis for
$\text{sl}(2, \RR)$ is orthonormal,
\[\{ A_1, A_2, A_3 \} = %
\left\{ \begin{bmatrix} \tfrac{1}{2} & 0 \\ 0 &
-\tfrac{1}{2} \end{bmatrix},
\begin{bmatrix} 0 & \tfrac{1}{2} \\ \tfrac{1}{2} & 0 \end{bmatrix},
\begin{bmatrix} 0 & -\tfrac{1}{2} \\ \tfrac{1}{2} & 0 \end{bmatrix} \right\}. \]

We will make use of the following standard forms for a pair of
geodesics in $\HH^2$.  Let $\alpha$ and $\beta$ be oriented geodesics.
Then there is an oriented geodesic $\gamma$ whose positive limit point
is the same as the positive limit point of $\alpha$, and whose
negative limit point is the same as the negative limit point for
$\beta$.  Up to the action of $\PSL(2, \RR)$, we may assume that
$\gamma$ is given by
\[ \gamma(t) = \begin{bmatrix} e^{t/2} & 0 \\ 0 &
e^{-t/2} \end{bmatrix}. \]

\medskip
All geodesics with the same positive limit point are known as the
stable manifold for $\gamma$ and are given by
\[ \alpha_s(t) = \begin{bmatrix} 1 & s \\ 0 &
1 \end{bmatrix} \gamma(t). \]
Similarly, all geodesics with the same negative limit point as
$\gamma$, known as the unstable manifold for $\gamma$, are given by
\[ \beta_{s'}(t) = \begin{bmatrix} 1 & 0 \\ s' &
1 \end{bmatrix} \gamma(t). \]
Given geodesics $\alpha$ and $\beta$, we can, by a reflection if required and a translation of the parametrization of $\gamma$, arrange that in the standard form $\alpha = \alpha_s$ with $s > 0$ and $\beta = \beta_{s'}$ with $s' = \pm s$.

\begin{proposition}
\label{prop:distance upper lower}
For $a+b \le \theta_0 / L_0 \sqrt{2}$,
\[ \tfrac{\sqrt{2}}{L_0} \sqrt{a^2 + b^2} \le d_{\PSL(2, \RR)}
\left( \begin{bmatrix} 1 & a \\ 0 & 1 \end{bmatrix}, \begin{bmatrix} 1
& 0 \\ b & 1 \end{bmatrix} \right) \le L_0 \sqrt{2}(a+b) \]
\end{proposition}

\begin{proof}
Note that $\begin{bmatrix} 1 & a \\ 0 & 1 \end{bmatrix} = e^{Ua}$, where $U = \begin{bmatrix} 0 & 1 \\ 0 & 0 \end{bmatrix}$. 
In our chosen orthonormal basis for $\text{sl}(2, \RR)$, the length of the matrix $U$ is $\sqrt{2}$. 
Since $L_0 \norm{I - e^{Ua}} = L_0 \norm{Ua} = L_0 \sqrt{2} a \le \theta_0$, we may use the upper bound in \Cref{prop:lipschitz} to obtain that the distance of $e^{Ua}$ from the identity matrix is at most $L_0 \sqrt{2}a $. The required upper bound follows directly from the triangle inequality.

\medskip
For the lower bound, we may again use \Cref{prop:lipschitz} to obtain

\begin{align*}
    d_{\PSL(2, \RR)} \left( \begin{bmatrix} 1 & a \\ 0 &
1 \end{bmatrix}, \begin{bmatrix} 1 & 0 \\ b & 1 \end{bmatrix} \right)
&\ge \tfrac{1}{L_0} \norm{ I - \begin{bmatrix} 1 & -a \\ 0 &
  1 \end{bmatrix} \begin{bmatrix} 1 & 0 \\ b & 1 \end{bmatrix} } 
\ge \tfrac{1}{L_0} \norm{ \begin{bmatrix} ab & a \\ -b &
  0 \end{bmatrix} } \\
&\ge \tfrac{\sqrt{2}}{L_0} \sqrt{ a^2 b^2 + a^2 + b^2 } 
\ge \tfrac{\sqrt{2}}{L_0} \sqrt{ a^2 + b^2 },
\end{align*}
as required.
\end{proof}

\begin{proposition}
\label{prop:distance parameter}
For $s$ such that $2s \le \theta_0/ L_0 \sqrt{2}$, the distance $\theta$ between the geodesics $\alpha_s$ and $\beta_s$ satisfies
\[ \tfrac{2}{L_0} s \le \theta \le L_0 2 \sqrt{2} s. \]
\end{proposition}

\begin{proof}
Since an upper bound may be obtained by picking any two points, we have
\[ d_{\PSL(2, \RR)}(\alpha_s, \beta_s) \le d_{\PSL(2, \RR)}(\alpha_s(0),
\beta_s(0)). \]
By the previous proposition, $d_{\PSL(2, \RR)}(\alpha_s(0),
\beta_s(0)) \le L_0 2 \sqrt{2} s $.

\medskip
We now obtain a lower bound.  By definition,
\[ \theta = \inf_{t, t'} d_{\PSL(2, \RR)} (\alpha_s (t), \beta_s(t')).  \]
and moreover the right hand side will be infimized over small values of $t$ and $t'$.
In particular, we may use the lower bounds from \Cref{prop:lipschitz}, to obtain 
\[ \theta \ge \inf_{t, t'} \tfrac{1}{L_0} \norm{ I
    - \begin{bmatrix} e^{t/2} & s e^{-t/2} \\ 0 &
    e^{-t/2} \end{bmatrix}^{-1} \begin{bmatrix} e^{t'/2} & 0 \\ s
    e^{t'/2} & e^{-t'/2} \end{bmatrix} }. \]
Combining the matrix terms gives
\[ \theta \ge \tfrac{1}{L_0} \inf_{t, t'} \norm{ \begin{bmatrix} 1 -
  (1+s^2) e^{(t'-t)/2} & -s e^{-(t+t')/2} \\ -s e^{(t+t')/2} & 1 -
  e^{(t-t')/2} \end{bmatrix} }. \]

\medskip
The matrix norm is the sum of the squares of the entries. So we may use the sum of the squares of the off-diagonal entries as a lower bound. We obtain
\[ \theta \ge \tfrac{\sqrt{2}}{L_0} \inf_{t, t'} \sqrt{ s^2 e^{-(t+t')} + s^2
  e^{t+t'} } \]
which is minimized when $t + t' = 0$, so
\[ \theta \ge \tfrac{2}{L_0} s, \]
as required.
\end{proof}

\begin{proposition}\label{prop:t bound}
There are constants $\theta_0 > 0$ and $L_0 \ge 0$ and
$K = \log \tfrac{2 L_0}{\theta_0}$, such that for any two distinct
geodesics $\alpha$ and $\beta$ distance $\theta \le \theta_0$ apart,
there are unit speed parametrizations $\alpha(t)$ and $\beta(t)$ such
that for $|t| \le \log \tfrac{1}{\theta} - K$,
\[\tfrac{1}{2 L_0^2}
\theta e^{|t|} \le d_{\PSL(2, \RR)}(\alpha^1(t), \beta^1(t)) \le
\sqrt{2} L_0^2 \theta e^{|t|}. \]
\end{proposition}

\begin{proof}
Up to the $\PSL(2, \RR)$ action, we may assume that $\alpha$ and
$\beta$ have the standard forms $\alpha = \alpha_s(t)$ and
$\beta = \beta_s (t)$, defined above.
Since the distance is left invariant, we have 
\[
d_{\PSL(2, \RR)}( \alpha^1_s(t), \beta^1_s (t) ) =  d_{\PSL(2, \RR)}( \gamma^1(-t) \alpha^1_s(t), \gamma^1(-t) \beta^1_s (t) )
\]
Note that $\gamma^1(-t) \alpha^1_s (t) = \begin{bmatrix} 1 & s e^{-t} \\ 0 & 1 \end{bmatrix}$ and $\gamma^1(-t) \beta^1_s(t) = \begin{bmatrix} 1 & 0 \\ \pm s e^{t} & 1 \end{bmatrix}$.
For the small values of $t$ stated in the hypothesis, we use the bounds from \Cref{prop:distance upper lower} to obtain 
\[ \tfrac{\sqrt{2}}{L_0} s e^{|t|} \le d_{\PSL(2, \RR)}(\alpha^1(t), \beta^1(t))
\le L_0 2 \sqrt{2} s e^{|t|} \]
By \Cref{prop:distance parameter}, we obtain from above
\[ \tfrac{1}{2 L_0^2} \theta e^{|t|} \le d_{\PSL(2,
  \RR)}(\alpha^1(t), \beta^1(t)) \le \sqrt{2} L_0^2 \theta e^{|t|} \]
as required.
\end{proof}

We have obtained bounds on the distance from $\gamma^1_1(t)$ to
$\gamma^1_2(t)$.  We wish to show that this gives bounds on the distance
from $\gamma^1_1(t)$ to $\gamma^1_2$.  The upper bound is immediate.  We
start by showing a lower bound that holds in a general geodesic metric
space.

\begin{proposition}\label{prop:t bounds d}
Suppose that $(X, d)$ is a geodesic metric space.
Suppose that $\gamma_1$ and $\gamma_2$ are geodesics distance $\theta$ apart with unit speed parametrizations chosen
such that $\gamma_1(0)$ and $\gamma_2(0)$ are closest points between
$\gamma_1$ and $\gamma_2$. Then for all $t$,
\[ d_X( \gamma_1(t), \gamma_2) \ge \tfrac{1}{2} \left( d_X(
\gamma_1(t), \gamma_2(t)) - \theta \right).  \]
\end{proposition}

\begin{proof}
Suppose that the point of $\gamma_2$ closest to $\gamma_1(t)$ is $\gamma_2(r)$, and suppose that it lies distance $s$ from $\gamma_1(t)$.  We have
illustrated this in \Cref{fig:d bound}.

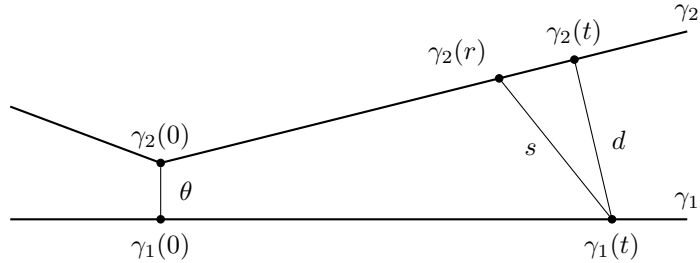
\begin{figure}[h]
\begin{center}
\begin{tikzpicture}%[scale=0.9]

\tikzstyle{point}=[circle, draw, fill=black, inner sep=1pt]

\draw [thick] (-1, -0.5) -- (8, -0.5) node [above] {$\gamma_1$};

\draw [thick] (-1, 1) -- (1, 0.25) -- (8, 2) node [above]
{$\gamma_2$};

\draw (1, -0.5) node (a) [point, label=below:$\gamma_1(0)$] {};

\draw (1, 0.25) node (b) [point, label=above:$\gamma_2(0)$] {};

\draw (7, -0.5) node [point, label=below:$\gamma_1(t)$] {};

\draw (6.5, 1.625) node [point, label=above:$\gamma_2(t)$] {};

\draw (a) -- (b) node [midway, label=right:$\theta$] {};

\draw (7, -0.5) -- (6.5, 1.625) node [midway, label=right:${d}$] {};

\draw (5.5, 1.375) node [point, label=above left:$\gamma_2(r)$] {};

\draw (7, -0.5) -- (5.5, 1.375) node [midway, label=left:$s$] {};

\end{tikzpicture}
\end{center}
\caption{A lower bound for the distance from $\gamma_1(t)$ to
  $\gamma_2$.} \label{fig:d bound}
\end{figure}

Applying the triangle inequality to the path from $\gamma_1(0)$ to
$\gamma_1(t)$ via $\gamma_2(0)$ and $\gamma_2(r)$, gives
\[ t \le \theta + r + s. \]
Similarly, applying the triangle inequality to the path from
$\gamma_1(t)$ to $\gamma_2(t)$ via $\gamma_2(r)$, gives
\[ d \le s + t - r. \]
Adding these two inequalities gives $d \le \theta + 2s$, equivalently,
$s \ge \tfrac{1}{2}(d - \theta)$, as required.
\end{proof}

We now complete the proof of \Cref{prop:small t bounds}.

\begin{proof}[Proof of \Cref{prop:small t bounds}]
By \Cref{prop:t bound} there is a constant $L \ge 1$ such
that for any geodesics $\gamma_1$ and $\gamma_2$, distance
$\theta \le \theta_0$ apart, with unit speed parametrizations
$\gamma_1(t)$ and $\gamma_2(t)$, such that the distance between
$\gamma^1_1(0)$ and $\gamma^1_2(0)$ is equal to $\theta$,
\[ \tfrac{1}{L} \theta e^t \le d = d_{\PSL(2, \RR)}(\gamma^1_1(t),
\gamma^1_2(t)) \le L \theta e^{t}. \]
We now show the following estimate for the distance from $\gamma_1(t)$
to $\gamma_2$.
\begin{equation}\label{eq:gamma bounds}
\tfrac{1}{4 L} \theta e^t \le d_{\PSL(2, \RR)}(\gamma^1_1(t),
\gamma^1_2) \le L \theta e^{t}.
\end{equation}

Since the distance from $\gamma^1_1 (t)$ to $\gamma^1_2 (t)$ is an upper bound for the distance from $\gamma^1_1(t)$ to $\gamma^1_2$, the upper bound follows immediately.
For the lower bound, \Cref{prop:t bounds d} implies that
$d_{\PSL(2, \RR)}(\gamma^1_1(t), \gamma^1_2) \ge \tfrac{1}{2} \left( d -
\theta \right)$.  For $d \ge 2 \theta$, we obtain
$d_{\PSL(2, \RR)}(\gamma^1_1(t), \gamma^1_2) \ge \tfrac{1}{4} d \ge
\tfrac{1}{4 L} \theta e^t$.  If $d \le 2 \theta$ instead, then
$t \le \log 8 L$. In this case, $\tfrac{1}{4L} \theta e^t \le \tfrac{\theta}{2}$. Since $d_{\PSL(2, \RR)}(\gamma^1_1(t), \gamma^1_2) \ge d_{\PSL(2, \RR)}(\gamma^1_1, \gamma^1_2) = \theta$, the lower bound holds trivially.

\medskip
The estimate for $d_{\PSL(2, \RR)}(\gamma^1_1(t), \gamma^1_2)$ in
\eqref{eq:gamma bounds} is valid as long as the bilipschitz bounds
hold in a small neighborhood of the identity in $\PSL(2, \RR)$,
i.e. as long as $d \le \theta_0$, and this holds as long as
$t \le \log \tfrac{1}{\theta} + \log \tfrac{\theta_0}{L}$.  Therefore
we may choose $K = - \log \tfrac{\theta_0}{L}$, which is non-negative
as $\theta_0 \le 1$ and $L \ge 1$.  Furthermore, this choice of $K$
implies that $\gamma^1_1(t)$ lies within a $\theta_0$-neighborhood of
$\gamma^1_2$ for an interval of size at least
$2( \log \tfrac{1}{\theta} - K )$, as required.
\end{proof}

Consider the map $p \colon \PSL(2, \RR) \to \HH^2$ given by applying 
the matrix $A$ to $i$ in the upper half space as a M\"obius
transformation, i.e.
\[ p \colon A = \begin{bmatrix} a & b \\ c & d \end{bmatrix} \mapsto
\frac{ a i + b }{ci + d}. \]
Recalling our basis for the Lie algebra, note that 
\[ e^{tA_1} = \begin{bmatrix} e^{t/2} & 0 \\ 0 & e^{-t/2} \end{bmatrix}, \]
and so $p(e^{tA_1}) = e^{t} i$, which is a unit speed geodesic in the upper half space.  Similarly $p(e^{t A_2})$ is the unit speed geodesic
obtained by rotating the geodesic for $e^{tA_1}$ by $\pi/2$ clockwise, and $e^{tA_3}$ gives a rotation about $i$. Thus, the derivative at $I$ of
the map $p \colon \PSL(2, \RR) \to \HH^2$ is given by the projection
matrix
\[ \begin{bmatrix} 1 & 0 & 0 \\ 0 & 1 & 0 \end{bmatrix}, \]
where we use the basis $\{ i, 1 \}$ for the tangent space to $i$ in the upper
half space.  As the map is equivariant, it is
distance non-increasing, and so distances in $\HH^2$ are lower bounds
for distances in $\PSL(2, \RR)$.  Since the kernel of the map is the
compact subgroup $SO(2) \cong S^1$, which has diameter $\pi/2$, we get
\begin{equation}\label{eq:dist reducing}
d_{\HH^2}( p(A), p(B) ) \le d_{\PSL(2, \RR)}(A, B) \le d_{\HH^2}(
p(A), p(B) ) + \pi.
\end{equation}

We may now complete the proof of \Cref{prop:fellow travel}
by combining the large $t$ estimates from distances in $\HH^2$ with
the small $t$ estimates from \Cref{prop:small t bounds}.

\begin{proof}[Proof of \Cref{prop:fellow travel}]
For small values of $t$, we have the bounds from \Cref{prop:small t bounds}, i.e. there are constants $\theta_0$, $L$
and $K = \log \tfrac{L}{\theta_0}$ such that for $\theta \le \theta_0$
and $|t| \le \log \tfrac{1}{\theta} - K$,
\[ \tfrac{1}{L} \theta e^t \le d_{\PSL(2, \RR)}( \gamma^1_1(t), \gamma^1_2
) \le L \theta e^t.  \]

We now use the distance estimates in $\HH^2$ to extend these bounds to
$|t| \le \log \tfrac{1}{\theta}$.  Using \eqref{eq:dist reducing}, and
the distance bounds from \Cref{prop:hyperbolic bounds disjoint} and
\Cref{prop:hyperbolic bounds intersect}, we obtain the following
bounds for distances in $PSL(2, \RR)$.
\[ \tfrac{1}{8}(e^t - 1) \le d_{\PSL(2, \RR)}( \gamma^1_1(t), \gamma^1_2 )
\le \tfrac{3}{2} \theta e^t + \pi.  \]
So for $t \in [ \log \tfrac{1}{\theta} - 9, \log \tfrac{1}{\theta} ]$,
\[ \tfrac{1}{M} \theta e^t \le d_{\PSL(2, \RR)}( \gamma^1_1(t), \gamma^1_2
) \le M \theta e^t,  \]
for $M = \tfrac{3}{2} + \pi e^9 \le 10^5$.

\medskip
We can now choose the constant $L_0$ in the statement of \Cref{prop:fellow travel} to be the maximum of $M$ and $L$ from
\Cref{prop:small t bounds}.
\end{proof}

\end{appendices}

\bibliography{references}

\begin{bibdiv}
\begin{biblist}

\bib{bh}{book}{
      author={Bridson, Martin~R.},
      author={Haefliger, Andr\'{e}},
       title={Metric spaces of non-positive curvature},
      series={Grundlehren der mathematischen Wissenschaften [Fundamental
  Principles of Mathematical Sciences]},
   publisher={Springer-Verlag, Berlin},
        date={1999},
      volume={319},
        ISBN={3-540-64324-9},
         url={https://doi.org/10.1007/978-3-662-12494-9},
      review={\MR{1744486}},
}

\bib{Casson-Bleiler}{book}{
      author={Casson, Andrew~J.},
      author={Bleiler, Steven~A.},
       title={Automorphisms of surfaces after {N}ielsen and {T}hurston},
      series={London Mathematical Society Student Texts},
   publisher={Cambridge University Press, Cambridge},
        date={1988},
      volume={9},
        ISBN={0-521-34203-1},
  url={https://doi-org.ezproxy.library.wisc.edu/10.1017/CBO9780511623912},
      review={\MR{964685}},
}

\bib{cannon-thurston}{article}{
      author={Cannon, James~W.},
      author={Thurston, William~P.},
       title={Group invariant {P}eano curves},
        date={2007},
        ISSN={1465-3060},
     journal={Geom. Topol.},
      volume={11},
       pages={1315\ndash 1355},
         url={https://doi.org/10.2140/gt.2007.11.1315},
      review={\MR{2326947}},
}

\bib{einsiedler-ward}{book}{
      author={Einsiedler, Manfred},
      author={Ward, Thomas},
       title={Ergodic theory with a view towards number theory},
      series={Graduate Texts in Mathematics},
   publisher={Springer-Verlag London, Ltd., London},
        date={2011},
      volume={259},
        ISBN={978-0-85729-020-5},
         url={https://doi.org/10.1007/978-0-85729-021-2},
      review={\MR{2723325}},
}

\bib{farb}{book}{
      author={Farb, Benson~Stanley},
       title={Relatively hyperbolic and automatic groups with applications to
  negatively curved manifolds},
   publisher={ProQuest LLC, Ann Arbor, MI},
        date={1994},
  url={http://gateway.proquest.com/openurl?url_ver=Z39.88-2004&rft_val_fmt=info:ofi/fmt:kev:mtx:dissertation&res_dat=xri:pqdiss&rft_dat=xri:pqdiss:9429137},
        note={Thesis (Ph.D.)--Princeton University},
      review={\MR{2690989}},
}

\bib{gmpueffective}{article}{
      author={Gadre, Vaibhav},
      author={Maher, Joseph},
      author={Pfaff, Catherine},
      author={Uyanik, Caglar},
       title={Singularity of {C}annon-{T}hurston maps},
        date={2025},
}

\bib{hamenstaedt}{article}{
      author={Hamenstaedt, Ursula},
       title={Word hyperbolic extensions of surface groups},
        date={2005},
      eprint={0505244},
         url={https://arxiv.org/abs/math/0505244},
}

\bib{hoffoss}{article}{
      author={Hoffoss, Diane},
       title={Suspension flows are quasigeodesic},
        date={2007},
        ISSN={0022-040X},
     journal={J. Differential Geom.},
      volume={76},
      number={2},
       pages={215\ndash 248},
         url={https://doi.org/10.4310/jdg/1180135678},
      review={\MR{2330414}},
}

\bib{kapovich}{book}{
      author={Kapovich, Michael},
       title={Hyperbolic manifolds and discrete groups},
      series={Progress in Mathematics},
   publisher={Birkh\"{a}user Boston, Inc., Boston, MA},
        date={2001},
      volume={183},
        ISBN={0-8176-3904-7},
         url={https://mathscinet.ams.org/mathscinet-getitem?mr=1792613},
      review={\MR{1792613}},
}

\bib{kapovich-sardar}{book}{
      author={Kapovich, Michael},
      author={Sardar, Pranab},
       title={Trees of hyperbolic spaces},
      series={Mathematical Surveys and Monographs},
   publisher={American Mathematical Society, Providence, RI},
        date={2024},
      volume={282},
        ISBN={978-1-4704-7425-6; [9781470477783]},
         url={https://doi.org/10.1090/surv/282},
      review={\MR{4783431}},
}

\bib{manning}{incollection}{
      author={Manning, Anthony},
       title={Dynamics of geodesic and horocycle flows on surfaces of constant
  negative curvature},
        date={1991},
   booktitle={Ergodic theory, symbolic dynamics, and hyperbolic spaces
  ({T}rieste, 1989)},
      series={Oxford Sci. Publ.},
   publisher={Oxford Univ. Press, New York},
       pages={71\ndash 91},
         url={https://mathscinet.ams.org/mathscinet-getitem?mr=1130173},
      review={\MR{1130173}},
}

\bib{martin}{book}{
      author={Martin, George~E.},
       title={The foundations of geometry and the non-{E}uclidean plane},
      series={Undergraduate Texts in Mathematics},
   publisher={Springer-Verlag, New York},
        date={1996},
        ISBN={0-387-90694-0},
        note={Corrected third printing of the 1975 original},
      review={\MR{1410263}},
}

\bib{mcmullen}{article}{
      author={McMullen, Curtis~T.},
       title={Local connectivity, {K}leinian groups and geodesics on the blowup
  of the torus},
        date={2001},
        ISSN={0020-9910},
     journal={Invent. Math.},
      volume={146},
      number={1},
       pages={35\ndash 91},
         url={https://doi.org/10.1007/PL00005809},
      review={\MR{1859018}},
}

\bib{mitra}{article}{
      author={Mitra, Mahan},
       title={Cannon-{T}hurston maps for hyperbolic group extensions},
        date={1998},
        ISSN={0040-9383},
     journal={Topology},
      volume={37},
      number={3},
       pages={527\ndash 538},
         url={https://doi.org/10.1016/S0040-9383(97)00036-0},
      review={\MR{1604882}},
}

\bib{mj-sardar}{article}{
      author={Mj, Mahan},
      author={Sardar, Pranab},
       title={A combination theorem for metric bundles},
        date={2012},
        ISSN={1016-443X,1420-8970},
     journal={Geom. Funct. Anal.},
      volume={22},
      number={6},
       pages={1636\ndash 1707},
         url={https://doi-org.csi.ezproxy.cuny.edu/10.1007/s00039-012-0196-1},
      review={\MR{3000500}},
}

\bib{thurston}{incollection}{
      author={Thurston, William~P.},
       title={Hyperbolic structures on 3-manifolds, {II}: surface groups and
  3-manifolds which fiber over the circle},
        date={2022},
   booktitle={Collected works of {W}illiam {P}. {T}hurston with commentary.
  {V}ol. {II}. 3-manifolds, complexity and geometric group theory},
   publisher={Amer. Math. Soc., Providence, RI},
       pages={79\ndash 110},
        note={August 1986 preprint, January 1998 eprint},
      review={\MR{4556467}},
}

\end{biblist}
\end{bibdiv}

%%%%%%%%%%%%%%%%%%%%%%%%%%%%%%%%%%%%%%%%%%%%%%%%%%%%%%%%%%%%%%%%%%%%%%%%%%%%%%%%%%

\medskip

School of Mathematics and Statistics, University of Glasgow \href{mailto:Vaibhav.Gadre@glasgow.ac.uk}{Vaibhav.Gadre@glasgow.ac.uk}

\medskip

CUNY College of Staten Island and CUNY Graduate Center
\href{mailto:joseph.maher@csi.cuny.edu}{joseph.maher@csi.cuny.edu}

\medskip 
Queen's University at Kingston \href{mailto:catherine.pfaff@gmail.com}{catherine.pfaff@gmail.com}

\medskip

Department of Mathematics,
University of Wisconsin--Madison
\href{mailto:caglar@math.wisc.edu}{caglar@math.wisc.edu}

\end{document}